\documentclass[12pt]{article}
\usepackage{amssymb, amsmath, amsthm, amscd}
\usepackage[dvips]{graphics}
\usepackage[utf8]{inputenc}
\usepackage[all,cmtip]{xy}
\usepackage{bbm}
\usepackage{enumitem}
\usepackage{setspace}
\usepackage[colorlinks=true]{hyperref}
\hypersetup{colorlinks   = true}
\hypersetup{linkcolor=blue}
\begin{document}

\newtheorem{The}{Theorem}[section]
\newtheorem{Lem}[The]{Lemma}
\newtheorem{Prop}[The]{Proposition}
\newtheorem{Cor}[The]{Corollary}
\newtheorem{Rem}[The]{Remark}
\newtheorem{Obs}[The]{Observation}
\newtheorem{SConj}[The]{Standard Conjecture}
\newtheorem{Titre}[The]{\!\!\!\! }
\newtheorem{Conj}[The]{Conjecture}
\newtheorem{Question}[The]{Question}
\newtheorem{Prob}[The]{Problem}
\newtheorem{Def}[The]{Definition}
\newtheorem{Not}[The]{Notation}
\newtheorem{Claim}[The]{Claim}
\newtheorem{Conc}[The]{Conclusion}
\newtheorem{Ex}[The]{Example}
\newtheorem{Fact}[The]{Fact}
\newtheorem{Formula}[The]{Formula}
\newtheorem{Formulae}[The]{Formulae}
\newcommand{\C}{\mathbb{C}}
\newcommand{\R}{\mathbb{R}}
\newcommand{\N}{\mathbb{N}}
\newcommand{\Z}{\mathbb{Z}}
\newcommand{\Q}{\mathbb{Q}}
\newcommand{\Proj}{\mathbb{P}}
\newcommand{\Rc}{\mathcal{R}}
\newcommand{\Oc}{\mathcal{O}}
\newcommand{\diff}{\textit{diff}}

\begin{center}

{\Large\bf Non-K\"ahler Mirror Symmetry of the Iwasawa Manifold}

\end{center}

\begin{center}

{\large Dan Popovici}

\end{center}

\vspace{1ex}

\noindent{\small{\bf Abstract.} We propose a new approach to the Mirror Symmetry Conjecture in a form suitable to possibly non-K\"ahler compact complex manifolds whose canonical bundle is trivial. We apply our methods by proving that the Iwasawa manifold $X$, a well-known non-K\"ahler compact complex manifold of dimension $3$, is its own mirror dual to the extent that its Gauduchon cone, replacing the classical K\"ahler cone that is empty in this case, corresponds to what we call the local universal family of essential deformations of $X$. These are obtained by removing from the Kuranishi family the two ``superfluous'' dimensions of complex parallelisable deformations that have a similar geometry to that of the Iwasawa manifold. The remaining four dimensions are shown to have a clear geometric meaning including in terms of the degeneration at $E_2$ of the Fr\"olicher spectral sequence. On the local moduli space of ``essential'' complex structures, we obtain a canonical Hodge decomposition of weight $3$ and a variation of Hodge structures, construct coordinates and Yukawa couplings while implicitly proving a local Torelli theorem. On the metric side of the mirror, we construct a variation of Hodge structures parametrised by a subset of the complexified Gauduchon cone of the Iwasawa manifold using the sGG property of all the small deformations of this manifold proved in earlier joint work of the author with L. Ugarte. Finally, we define a mirror map linking the two variations of Hodge structures and we highlight its properties.}

\vspace{1ex}

\section{Introduction}\label{section:introduction} The standard mirror symmetry conjecture predicts that the Calabi-Yau (C-Y) threefolds, defined as compact K\"ahler manifolds of complex dimension $3$ whose canonical bundle is trivial, ought to occur in pairs $(X,\,\widetilde{X})$ such that the local universal family of deformations of the complex structure (i.e. the Kuranishi family) of $X$ is isomorphic to the moduli space of K\"ahler structures enriched with $B$-fields (i.e. the complexified K\"ahler cone) of $\widetilde{X}$, and vice-versa. 

As is well known, there is an obvious cohomological obstruction to some K\"ahler C-Y threefolds $X$ having K\"ahler mirror duals $\widetilde{X}$. The Kuranishi family $(X)_{t\in\Delta}$ of a given K\"ahler C-Y manifold $X=X_0$ is unobstructed (i.e. its base space $\Delta$ is {\bf smooth}, hence can be viewed as an open ball in the classifying space $H^{0,\,1}(X,\,T^{1,\,0}X)$) by the Bogomolov-Tian-Todorov theorem ([Bog78], [Tia87], [Tod89]). The triviality of the canonical bundle $K_X$ implies the isomorphism $H^{0,\,1}(X,\,T^{1,\,0}X)\simeq H^{n-1,\,1}(X,\,\C) = H^{2,\,1}(X,\,\C)$, where the last identity follows from the assumption $\mbox{dim}_{\C}X  :=n=3$. On the other hand, the complexified K\"ahler cone $\widetilde{\cal K}_{\widetilde{X}}$ of $\widetilde{X}$ is an open subset of $H^{1,\,1}(\widetilde{X},\,\C)$. So a necessary condition for $X$ and $\widetilde{X}$ to be mirror dual is that the tangent space to $\Delta$ at $0$ (i.e. $H^{2,\,1}(X,\,\C)$) be isomorphic to the tangent space to the complexified K\"ahler cone $\widetilde{\cal K}_{\widetilde{X}}$ at some point (i.e. $H^{1,\,1}(\widetilde{X},\,\C)$), and vice-versa. It is thus necessary to have

$$h^{2,\,1}(X) = h^{1,\,1}(\widetilde{X}) \hspace{3ex} \mbox{and} \hspace{3ex} h^{2,\,1}(\widetilde{X}) = h^{1,\,1}(X).$$

\noindent However, there exist K\"ahler C-Y threefolds $X$ such that $h^{2,\,1}(X)=0$ (the so-called {\it rigid} such threefolds, those that do not deform). Consequently, the mirror dual $\widetilde{X}$, if it exists, cannot be K\"ahler since $h^{1,\,1}(\widetilde{X})=0$.

This standard observation has prompted many authors so far to conjecture the mirror symmetry only for {\it generic} K\"ahler C-Y threefolds 
so that the discussion is confined to the K\"ahler realm. 
The idea of investigating the possible existence of a mirror symmetry phenomenon beyond the K\"ahler world was loosely suggested in [Rei87] 
and received attention recently in [LTY15]. This investigation is our main motivation in the present work. Our methods and point of view are very different from those in [LTY15].  

The standard approach to the study of the K\"ahler side of the mirror is to use Gromov-Witten invariants attached to pseudo-holomorphic curves and to count rational curves. However, on many non-K\"ahler compact complex threefolds with trivial canonical bundle, there exist no rational curves.

\vspace{2ex}

We propose in this paper a new approach to mirror symmetry by means of transcendental methods in the general, possibly non-K\"ahler context of compact complex manifolds whose canonical bundle is trivial. By extension of the classical definition, we shall still call them {\bf Calabi-Yau (C-Y) manifolds}. We test our new point of view on the Iwasawa manifold, a well-known non-K\"ahler compact complex C-Y manifold, and take full advantage of the explicit nature of extensive computations for this particular manifold found in the works [Nak75], [Ang11] and [Ang14] of Nakamura and Angella.

We hope that our methods will apply to larger classes of C-Y manifolds in the future and that this paper is the first in a series. One of the new ideas it introduces is the notion of local universal family of {\it essential deformations}, viewed as a subfamily of the Kuranishi family, of the Iwasawa manifold $X$. Three equivalent definitions are given: by removing the complex parallelisable small deformations from the Kuranishi family; by selecting the small deformations that have a kind of {\it polarisation} by the holomorphic non-closed $1$-form $\gamma$ associated with $X$ (cf. Definition \ref{Def:Delta_gamma}); and by selecting the vector subspace of the Dolbeault cohomology space $H^{n-1,\,1}(X,\,\C)$ (known to parametrise all the small deformations of a C-Y manifold $X$, while the complex dimension of $X$ is $n=3$ here) that is naturally isomorphic to the vector space $E_2^{n-1,\,1}(X)$ featuring in bidegree $(n-1,\,1)$ on the second page of the Fr\"olicher spectral sequence of $X$.

Looking ahead beyond the special case of the Iwasawa manifold treated in this paper, we come up against the question of what makes a deformation of a general, possibly non-K\"ahler, C-Y manifold {\it essential}. Our hunch is that a definition in terms of the Fr\"olicher spectral sequence, that will yield a replacement for the Hodge decomposition in middle degree $n$, is the best bet in a general pattern that will hopefully emerge in the future after further examples of C-Y manifolds have been investigated.

\subsection{The Gauduchon cone}\label{subsection:G-cone}

One of our main ideas in this work is to overcome the double whammy of a possible non-existence of both K\"ahler metrics and rational curves by using the Gauduchon cone (cf. [Pop15], see definition reminder (\ref{eqn:G-cone_def}) below) of the given non-K\"ahler C-Y manifold $X$. This furnishes both an alternative to the classical K\"ahler cone (that is empty on a non-K\"ahler manifold) and a transcendental substitute for cohomology classes of (currents of integration on) curves (e.g. by virtue of its elements' bidegree $(n-1,\,n-1)$, but also in a far deeper sense). We stress that the Gauduchon cone is relevant even on projective and on K\"ahler non-projective manifolds where it might be preferable to the K\"ahler cone in certain circumstances (for example, when it is strictly bigger, allowing for more flexibility).

Recall that if $X$ is a compact complex manifold with $\mbox{dim}_\C X=n$, a Hermitian metric on $X$ is any $C^{\infty}$ positive definite $(1,\,1)$-form $\omega>0$ on $X$. 
It is called a {\bf Gauduchon metric} (cf. [Gau77]) if $\partial\bar\partial\omega^{n-1}=0$ 
and it is called a {\bf strongly Gauduchon (sG) metric} (cf. [Pop13a]) if $\partial\omega^{n-1}$ is $\bar\partial$-exact. 
Every strongly Gauduchon metric is Gauduchon. Gauduchon metrics always exist (cf. [Gau77]), while strongly Gauduchon metrics need not exist although they do on many manifolds.

The compact complex manifolds $X$ on which every Gauduchon metric is strongly Gauduchon were introduced under the name of {\bf sGG manifolds} and studied in [PU14]. 
They contain the Iwasawa manifold and all its small deformations, so they will feature prominently in this paper (cf. sections \ref{section:metric-side} and \ref{section:mirror-map}).

On the other hand, recall that the following two cohomologies are especially relevant on non-K\"ahler compact complex manifolds $X$. For $p,q=0,\dots , n$, 
\begin{equation}\label{eqn:B-C_A_cohomologies_def}H^{p,\,q}_{BC}(X,\,\C)=\frac{\ker\partial\cap\ker\bar\partial}{\mbox{Im}\,\partial\bar\partial} \hspace{2ex} 
\mbox{and} \hspace{2ex} H^{p,\,q}_A(X,\,\C)=\frac{\ker\partial\bar\partial}{\mbox{Im}\,\partial + \mbox{Im}\,\bar\partial}\end{equation}
stand for the Bott-Chern, respectively Aeppli cohomology groups of bidegree $(p,\,q)$ of $X$, 
where all the spaces involved are sub-quotients of the space $C^{\infty}_{p,\,q}(X,\,\C)$ of smooth $(p,\,q)$-forms on $X$. 
As with the other familiar cohomology theories, these spaces can be computed using either smooth forms or currents. 
Although it is the coarsest one, the Aeppli cohomology is suited to the study of Gauduchon metrics (since $\omega^{n-1}\in\ker\partial\bar\partial$ in that case) 
and its consideration in bidegree $(n-1,\,n-1)$ enables one to eventually get information in bidegree $(1,\,1)$ thanks to the following canonical bilinear pairing
\begin{equation}\label{eqn:duality}H^{1,\,1}_{BC}(X,\,\C)\times H^{n-1,\,n-1}_A(X,\,\C)\to\C, \hspace{3ex} ([\alpha]_{BC},\,[\beta]_A)\mapsto\int\limits_X\alpha\wedge\beta,\end{equation}
being non-degenerate (cf. e.g. [Aep62], or [Sch07], or [Pop15]), hence establishing a {\bf duality} in the transcendental context that parallels the classical duality in algebraic geometry between divisors (defining $(1,\,1)$-cohomology classes) and curves (defining $(n-1,\,n-1)$-cohomology classes).

Bringing the metric and cohomological points of view together, the {\bf Gauduchon cone} of $X$ (cf. [Pop15]) 
is defined as the set of Aeppli cohomology classes of $(n-1)^{st}$ powers of all the Gauduchon metrics on $X$:
\begin{equation}\label{eqn:G-cone_def}{\cal G}_X
:=\bigg\{[\omega^{n-1}]_A\in H^{n-1,\,n-1}_A(X,\,\R)\, \mid \, \omega \mbox{ is a Gauduchon metric on } X\bigg\}.\end{equation}
It is an open convex cone in $H^{n-1,\,n-1}_A(X,\,\R)$.

\subsection{Hodge decomposition on certain non-K\"ahler manifolds}\label{subsection:H-decomp}

Although the class of manifolds described in this subsection does not contain the Iwasawa manifold, 
it furnishes us with a model situation into which the Iwasawa manifold will partially fit after suitable adjustments. 

Recall that a compact complex manifold $X$ is said to be a {\bf $\partial\bar\partial$-manifold} if for every bidegree $(p,\,q)$ 
and every smooth $d$-closed $(p,\,q)$-form $u$ on $X$, the following exactness conditions are equivalent:
\begin{equation}\label{eqn:dd-bar_def}u\in\mbox{Im}\,\partial \iff u\in\mbox{Im}\,\bar\partial \iff u\in\mbox{Im}\,d  \iff u\in\mbox{Im}\,\partial\bar\partial.\end{equation}
The $\partial\bar\partial$ property is equivalent to all the canonical linear maps $H^{p,\,q}_{BC}(X,\,\C)\to H^{p,\,q}_A(X,\,\C)$ being isomorphisms for all bidegrees $(p,\,q)$. 
Every compact K\"ahler (or merely {\it class} ${\cal C}$) manifold is known to be a $\partial\bar\partial$-manifold, but there are examples (see e.g. [Pop14, Observation 4.10]) 
of $\partial\bar\partial$-manifolds that are not in the {\it class} ${\cal C}$  (i.e. are not bimeromorphically equivalent to a compact K\"ahler manifold). 
Thus, the class of $\partial\bar\partial$-manifolds is much larger than the K\"ahler class.

The $\partial\bar\partial$ property implies the Hodge decomposition and the Hodge symmetry 
in the sense that for all $k\in\{0,\dots , 2n\}$ and all $p,q\in\{0,\dots , n\}$ there exist {\it canonical} (i.e. depending only on the complex structure) isomorphisms
\begin{equation*}H^k_{DR}(X,\,\C)\stackrel{\simeq}{\longrightarrow}\bigoplus\limits_{p+q=k}H^{p,\,q}_{\bar\partial}(X,\,\C) \hspace{2ex} 
\mbox{and} \hspace{2ex} H^{p,\,q}_{\bar\partial}(X,\,\C)\stackrel{\simeq}{\longrightarrow}\overline{H^{q,\,p}_{\bar\partial}(X,\,\C)},\end{equation*}
\noindent where the last isomorphism is defined by conjugation on $H^k_{DR}(X,\,\C)=H^k_{DR}(X,\,\R)\otimes\C$. 
See [DGMS75] for the origin of the notion of $\partial\bar\partial$-manifold and e.g. [Pop14] for a rundown on the basic facts about this class. 
Moreover, if $X$ is a $\partial\bar\partial$-manifold with trivial canonical bundle $K_X$, 
the Bogomolov-Tian-Todorov unobstructedness theorem mentioned above in the K\"ahler context still holds 
(see e.g. [Pop13b]). Since the $\partial\bar\partial$ property is open under deformations of the complex structure (cf. [Wu06]), 
the facts just mentioned add up to the following picture.

\begin{Conc}\label{Conc:picture_dd-bar_def} For every $\partial\bar\partial$-manifold $X$ such that $K_X$ is trivial, 
the base $\Delta$ of the Kuranishi family $(X_t)_{t\in\Delta}$ of $X=X_0$ is {\bf smooth} 
and can be viewed as an open ball in $H^{0,\,1}(X,\,T^{1,\,0}X)\simeq H^{n-1,\,1}(X,\,\C)$, 
while all the fibres $X_t$ with $t\in\Delta$ sufficiently close to $0$ are again $\partial\bar\partial$-manifolds with trivial canonical bundle $K_{X_t}$. 
Hence, for every $t\in\Delta$, we have a Hodge decomposition and a Hodge symmetry in the form of canonical isomorphisms
\begin{equation}\label{eqn:H-decomp}H^k(X,\,\C)\stackrel{\simeq}{\longrightarrow}\bigoplus\limits_{p+q=k}H^{p,\,q}_{\bar\partial}(X_t,\,\C) \hspace{2ex} 
\mbox{and} \hspace{2ex} H^{p,\,q}_{\bar\partial}(X_t,\,\C)\stackrel{\simeq}{\longrightarrow}\overline{H^{q,\,p}_{\bar\partial}(X_t,\,\C)},  \hspace{3ex} t\in\Delta,\end{equation}
where $H^k(X,\,\C)\simeq H^k_{DR}(X_t,\,\C)$ for all $t\in\Delta$ is the constant bundle ${\cal H}^k$ over $\Delta$ 
induced by the $C^{\infty}$ triviality of the family $(X_t)_{t\in\Delta}$ (that implies the invariance of the De Rham cohomology w.r.t. the complex structure of the fibre $X_t$).

Moreover, the Hodge numbers $h^{p,\,q}(t):=\mbox{dim}_\C H^{p,\,q}(X_t,\,\C)$ are independent of $t\in\Delta$ after possibly shrinking $\Delta$ about $0$. 
Hence, thanks to classical results of Kodaira and Spencer [KS60], we get $C^{\infty}$ vector bundles ${\cal H}^{p,\,q}$ over $\Delta$ defined as
$$\Delta\ni t\mapsto H^{p,\,q}_{\bar\partial}(X_t,\,\C):={\cal H}^{p,\,q}_t, \hspace{3ex} p,q=0, \dots , n,$$
\noindent and {\bf holomorphic} subbundles $F^p{\cal H}^k$ of the constant bundles ${\cal H}^k$ over $\Delta$ defined as
$$\Delta\ni t\mapsto F^pH^k(X_t,\,\C):=\bigoplus\limits_{l\geq p}H^{l,\,k-l}(X_t,\,\C):=F^p{\cal H}^k_t$$
\noindent that make up the Hodge filtration $F^0{\cal H}^k\supset\dots\supset F^p{\cal H}^k\supset F^{p+1}{\cal H}^k\supset\dots\supset F^k{\cal H}^k = {\cal H}^{k,\,0}$.
\end{Conc}


\subsection{The Iwasawa manifold}\label{subsection:Iwasawa}

Our main object of study in this paper will be the standard {\it Iwasawa manifold} $X=G/\Gamma$, defined as the quotient of the Heisenberg group
$$G:=\left\{\begin{pmatrix}1 & z_1 & z_3\\   
			  0 & 1 & z_2\\
		  0 & 0 & 1\end{pmatrix}\,\, ; \,\, z_1, z_2, z_3\in\C\right\}\subset GL_3(\C)$$
\noindent  by its discrete subgroup $\Gamma\subset G$ of matrices with entries $z_1, z_2, z_3\in\Z[i]$.

The map $(z_1,z_2,z_3)\mapsto (z_1,z_2)$ factors through the action of $\Gamma$ to a (holomorphically locally trivial) proper holomorphic submersion
$$\pi : X\to B,$$
where the base $B=\C^2/\Z[i]\oplus \Z[i]=\C/\Z[i]\times\C/\Z[i] $ is a two-dimensional Abelian variety (the product of two elliptic curves)
and where all the fibres are isomorphic to the Gauss elliptic curve $\C/\Z[i] $. 
This description displays the non-existence on $X$ of curves normalised by smooth rational curves, 
as any map from such a curve to any factor $\C/\Z[i]$ would be constant. 
(Indeed, thanks to the Riemann-Hurwitz formula, any non-constant map between two smooth curves is genus-decreasing.)

Since $G$ is a connected, simply connected, {\it nilpotent} complex Lie group, $X$ is a {\it nilmanifold}. 
Furthermore, $X$ is a {\it complex parallelisable} compact complex manifold (i.e. its holomorphic tangent bundle $T^{1,\,0}X$ is trivial) of complex dimension $3$. 
In particular, its canonical bundle $K_X$ is trivial, so $X$ is a Calabi-Yau manifold in our generalised sense.

It is well known that $X$ is not a $\partial\bar\partial$-manifold (in particular, it is not K\"ahler). 
In fact, its Fr\"olicher spectral sequence does not even degenerate at $E_1$, so there is no Hodge decomposition either canonical or non-canonical on $X$. 
(For a review of these and other facts, see e.g. [Pop14, $\S. 3.2$]).  

However, despite $X$ lacking the $\partial\bar\partial$ property, Nakamura [Nak75] showed that the Kuranishi family $(X_t)_{t\in\Delta}$ of $X=X_0$ is {\it unobstructed}, 
so its base $\Delta$ is smooth and can be identified with an open ball in $H^{0,\,1}(X,\,T^{1,\,0}X)\simeq H^{2,\,1}_{\bar\partial}(X,\,\C)$. 
It can be easily checked (see e.g. $\S.$\ref{subsection:2-1} below) that there is no Hodge decomposition of weight $3$ 
since the Dolbeault cohomology group $H^{2,\,1}_{\bar\partial}(X,\,\C)$ does not inject canonically into $H^3_{DR}(X,\,\C)$. 
In fact, $b_3=10$ while $h^{3,\,0} = h^{0,\,3} =1$ and $h^{2,\,1}=h^{1,\,2}=6$, 
so in a sense the vector space $H^{2,\,1}_{\bar\partial}(X,\,\C)$ is ``too large'' to fit into $H^3_{DR}(X,\,\C)$.

\vspace{2ex}

The main result of this paper can be loosely stated as follows (see Theorem \ref{The:mirror-map} for a precise statement).
\begin{The}\label{The:introd_Iwasawa-own-mirror} The Iwasawa manifold is its own mirror dual in the sense that its local universal family of essential deformations corresponds to its complexified Gauduchon cone.

\end{The}

The meaning of ``corresponds'' will be made precise by the end of the paper. Loosely speaking, it will mean that there exists a local biholomorphism between the local universal family of essential deformations and the complexified Gauduchon cone of the Iwasawa manifold and there exists an induced $C^\infty$ isomorphism of variations of Hodge structures (VHS) that exist on either side of the mirror. Moreover, this isomorphism is holomorphic at the level of the rank-$1$ components and anti-holomorphic at the level of the rank-$4$ components of these VHS.  

Here is a summary of the main steps and ideas.

\subsection{Outline of our approach}\label{subsection:outline}

 $(I)$\, On the {\bf complex-structure side of the mirror}, the starting point of our method is the observation that a natural Hodge decomposition of weight $3$ exists on the Iwasawa manifold $X$ 
if $H^{2,\,1}_{\bar\partial}(X,\,\C)$ is ``pared down'' to a $4$-dimensional vector subspace $H^{2,\,1}_{[\gamma]}(X,\,\C)\subset H^{2,\,1}_{\bar\partial}(X,\,\C)$ 
that injects canonically into $H^3_{DR}(X,\,\C)$ and parametrises what we call the {\it essential deformations} of $X$. 
Specifically, recalling that $\Delta\subset H^{2,\,1}_{\bar\partial}(X,\,\C)$ is an open ball, if we put
$$\Delta_{[\gamma]}:=\Delta\cap H^{2,\,1}_{[\gamma]}(X,\,\C),$$
\noindent we implicitly remove from the Kuranishi family $(X_t)_{t\in\Delta}$ the two dimensions corresponding to complex parallelisable deformations $X_t$ of $X$ 
(that have a similar geometry to that of $X$, so no geometric information is lost) 
and we are left with a family $(X_t)_{t\in\Delta_{[\gamma]}}$ of non-complex parallelisable deformations that we call {\it essential}. 
This description of the local deformations of $X$ is made possible by Nakamura's explicit calculations in [Nak75]. 
The holomorphic tangent space to $\Delta_{[\gamma]}$ at any of its points $t$ is isomorphic via the Kodaira-Spencer map 
to the analogue $H^{2,\,1}_{[\gamma]}(X_t,\,\C)$ at $t$ of $H^{2,\,1}_{[\gamma]}(X,\,\C) = H^{2,\,1}_{[\gamma]}(X_0,\,\C)$. 
We get a Hodge decomposition of weight $3$ for every $t\in\Delta_{[\gamma]}$ (cf. Proposition \ref{Prop:H21_gamma_Hodge}) in the following form.

\begin{Prop}\label{Prop:introd_H3_Hodge-decomp} There exist canonical isomorphisms
\begin{equation}\label{eqn:introd_H3_Hodge-decomp}H^3_{DR}(X,\,\C)\simeq H^{3,\,0}_{\bar\partial}(X_t,\,\C) \oplus H^{2,\,1}_{[\gamma]}(X_t,\,\C) 
\oplus H^{1,\,2}_{[\gamma]}(X_t,\,\C) \oplus H^{0,\,3}_{\bar\partial}(X_t,\,\C), \hspace{3ex} t\in\Delta_{[\gamma]},\end{equation}  
\noindent (where $H^{1,\,2}_{[\gamma]}(X_t,\,\C)\subset H^{1,\,2}_{\bar\partial}(X_t,\,\C)$ is defined by analogy with $H^{2,\,1}_{[\gamma]}(X_t,\,\C)$) and
\begin{equation}\label{eqn:introd_H3_Hodge-symm}H^{3,\,0}_{\bar\partial}(X_t,\,\C)\simeq\overline{H^{0,\,3}_{\bar\partial}(X_t,\,\C)} \hspace{2ex} 
\mbox{and} \hspace{2ex} H^{2,\,1}_{[\gamma]}(X,\,\C)\simeq\overline{H^{1,\,2}_{[\gamma]}(X,\,\C)}, \hspace{3ex} t\in\Delta_{[\gamma]}.\end{equation}
\end{Prop}

We go on to show that $\Delta_{[\gamma]}\ni t\mapsto H^{2,\,1}_{[\gamma]}(X_t,\,\C)$ is a $C^{\infty}$ vector bundle of rank $4$ (cf. Proposition \ref{Prop:H2-1_gamma_vbundle}) and that (\ref{eqn:introd_H3_Hodge-decomp}) and (\ref{eqn:introd_H3_Hodge-symm}) define a Hodge filtration 

$$F^2{\cal H}^3_{[\gamma]}\supset F^3{\cal H}^3$$ 

\noindent of {\it holomorphic} vector subbundles over $\Delta_{[\gamma]}$ of the constant bundle ${\cal H}^3$ of fibre $H^3_{DR}(X,\,\C)$. This induces a variation of Hodge structures (VHS) endowed with a Gauss-Manin connection satisfying the Griffiths transversality condition (cf. Theorem \ref{The:VHS_3_Delta}).

Thus, after restricting attention to the {\it essential deformations} of the non-$\partial\bar\partial$ Iwasawa manifold, 
we get a picture similar to the one described in Conclusion \ref{Conc:picture_dd-bar_def} for $\partial\bar\partial$-manifolds.  

Two further crucial observations cement the role played by the space $H^{2,\,1}_{[\gamma]}(X,\,\C)$ in this approach and its {\it canonical} nature. (By an isomorphism being canonical, we will mean that it is defined in an obvious way, not involving arbitrary choices, by the three standard holomorphic $1$-forms $\alpha, \beta, \gamma$ that generate the whole cohomology of the Iwasawa manifold and are induced by the canonical basis of $\C^3$ as recalled in $\S.$\ref{section:preliminaries}.)

The first observation (cf. Proposition \ref{Prop:H21_gamma_E2_21}, $(c)$) is the following

\begin{Prop}\label{Prop:introd_H21-gamma_E_2-21} There exists a canonical isomorphism
\begin{equation}\label{eqn:introd_H21_gamma_E_2-21} H^{2,\,1}_{[\gamma]}(X,\,\C)\simeq E_2^{2,\,1}(X,\,\C)\end{equation}
\noindent where $E_2^{2,\,1}(X,\,\C)$ is the space featuring at the second step of the Fr\"olicher spectral sequence of $X$ (known to degenerate at $E_2$ as do its counterparts for all the small deformations $X_t$). 
\end{Prop}

Moreover, the Hodge decomposition (\ref{eqn:introd_H3_Hodge-decomp}) reflects precisely this $E_2$ degeneration since there exist isomorphisms (cf. (\ref{eqn:E_2-splitting}))
\begin{equation}\label{eqn:introd_E_2_Hodge-decomp}H^3_{DR}(X,\,\C)\simeq E_2^{3,\,0}(X_t,\,\C) \oplus E_2^{2,\,1}(X_t,\,\C) \oplus E_2^{1,\,2}(X_t,\,\C) \oplus E_2^{0,\,3}(X_t,\,\C), \hspace{3ex} t\in\Delta_{[\gamma]},\end{equation}
\noindent in which each of the four spaces on the r.h.s. is isomorphic to the corresponding space on the r.h.s. of (\ref{eqn:introd_H3_Hodge-decomp}).

\vspace{2ex}

The second observation (cf. Observation \ref{Obs:H21_gamma-H22_isomorphism}) is the following 
\begin{Prop}\label{Prop:introd_H21-gamma_H22_A} There exists a canonical isomorphism 
\begin{equation}\label{eqn:introd_H21_gamma_H22_A}H^{2,\,1}_{[\gamma]}(X,\,\C)\simeq H^{2,\,2}_A(X,\,\C).\end{equation}
\end{Prop}

The isomorphism (\ref{eqn:introd_H21_gamma_H22_A}) justifies us in choosing the {\it essential deformations} of $X$ on the complex-structure side of the mirror 
and the {\it Gauduchon cone} of $X$ on the metric side of the mirror as the two main structures mirroring each other. 
Indeed, $H^{2,\,1}_{[\gamma]}(X,\,\C)$ is the tangent space to $\Delta_{[\gamma]}$ at $0$, while $H^{2,\,2}_A(X,\,\C)$ is the tangent space to the complexified Gauduchon cone 
(see Definition \ref{Def:G-cone-complexified}) at any of its points. 
The Aeppli-Gauduchon class $[\omega_0^2]\in{\cal G}_{X_0}$ of a natural Gauduchon metric $\omega_0$ induced on $X_0$ by the complex parallalisable structure of $X_0$ 
will be the privileged point chosen in the Gauduchon cone. 
It is the image of $0\in\Delta_{[\gamma]}$ under the mirror map that will be defined in Defintions \ref{Def:mirror-map_def} and \ref{Def:G-cone-complexified}. 
Isomorphism (\ref{eqn:introd_H21_gamma_H22_A}) is the single most powerful piece of initial motivating evidence in 
favour of the new mirror symmetry phenomenon that we highlight in this paper. 

\vspace{3ex}

 $(II)$\, On the {\bf metric side of the mirror}, we start off by constructing a $C^{\infty}$ family $(\omega_t)_{t\in\Delta_{[\gamma]}}$ of Gauduchon metrics 
on the fibres $(X_t)_{t\in\Delta_{[\gamma]}}$ (cf. Lemma \ref{Lem:omega_t_Gauduchon}) 
and a $C^{\infty}$ family $(\omega_t^{1,\,1})_{t\in\Delta}$ of Gauduchon metrics on $X_0$ (cf. Lemma \ref{Prop:omega_t_11}). 
The $\omega_t^{1,\,1}$'s are the $(1,\,1)$-components of the $\omega_t$'s w.r.t. the complex structure $J_0$ of $X_0$. 

Then we prove (cf. Corollary \ref{Cor:subbundle_H22_H4}) that the Aeppli cohomology groups of bidegree $(2,\,2)$ 
of the local essential deformations $X_t$ of the Iwasawa manifold $X=X_0$, namely the vector spaces
$$\Delta_{[\gamma]}\ni t\mapsto H^{2,\,2}_A(X_t,\,\C),$$
\noindent define a $C^{\infty}$ vector bundle ${\cal H}_A^{2,\,2}$ of rank $4$ that injects as a $C^{\infty}$ vector subbundle of the constant bundle 
${\cal H}^4\rightarrow\Delta_{[\gamma]}$ of fibre given by the De Rham cohomology group $H^4_{DR}(X_t,\,\C) = H^4(X,\,\C)$. 
This injection is proved by using in a crucial way the {\bf sGG property} (cf. [PU14]) of all the fibres $X_t$ 
and the family $(\omega_t)_{t\in\Delta_{[\gamma]}}$ of Gauduchon metrics thereon. 
Denoting by $\widetilde{H^{2,\,2}_{\omega_t}}$ the image of $H^{2,\,2}_A(X_t,\,\C)$ into $H^4(X,\,\C)$ under this $\omega_t$-induced injection, 
we get a $C^{\infty}$ vector bundle $\widetilde{{\cal H}^{2,\,2}_\omega}$ of rank $4$
$${\cal G}_{X_0}\ni[(\omega_t^{1,\,1})^2]_A\mapsto \widetilde{H^{2,\,2}_{\omega_t}}\subset H^4(X,\,\C)$$
\noindent after suitable identifications of certain spaces depending on $\omega_t$ with spaces depending on $\omega_t^{1,\,1}$ (cf. Conclusion \ref{Conc:VHS_two-families-metrics}).

 This produces a Hodge filtration

$$F_{{\cal G}}{\cal H}^4: = {\cal H}^{2,\,0}(B)\oplus\widetilde{{\cal H}^{2,\,2}_\omega}\supset F'_{{\cal G}}{\cal H}^4:={\cal H}^{2,\,0}(B)$$

\noindent of {\it holomorphic} vector bundles over the complexification $\widetilde{{\cal G}_0}$ of the subset ${\cal G}_0$ of the Gauduchon cone ${\cal G}_{X_0}$ consisting of the classes $[(\omega_t^{1,\,1})^2]_A$ with $t\in\Delta_{[\gamma]}$, where ${\cal H}^{2,\,0}(B)$ is a holomorphic line bundle over $\Delta_{[\gamma]}$ induced by the Albanese tori $B_t$ of the fibres $X_t$.

\vspace{3ex}

 $(III)$\, The {\bf link between the two sides of the mirror} is provided by the holomorphic family $(B_t)_{t\in\Delta}$ of $2$-dimensional complex Albanese tori $B_t = \mbox{Alb}(X_t)$ of the small deformations $X_t$ of the Iwasawa manifold $X=X_0$. Indeed, every small deformation $X_t$ of $X$ is a locally trivial holomorphic fibration $\pi_t:X_t\to B_t$ over its Albanese torus $B_t$. We get a holomorphic vector bundle of rank $5$

\begin{equation}\label{eqn:introd_subbundle_G_rank5_X0}
\widetilde{{\cal G}_0}\ni[(\omega_t^{1,\,1})^2]_A\mapsto H^{2,\,0}(B_t,\,\C)\oplus \widetilde{H^{2,\,2}_{\omega_t}}\subset H^3(X,\,\C) \oplus H^4(X,\,\C)\end{equation}
\noindent and a VHS parametrised by the complexification $\widetilde{{\cal G}_0}$ of the subset 
$${\cal G}_0:=\{[(\omega_t^{1,\,1})^2]_A\,\mid\,t\in\Delta_{[\gamma]}\}$$
\noindent of the Gauduchon cone ${\cal G}_{X_0}$ of $X_0$ (cf. Conclusion \ref{Conc:VHS_two-families-metrics}).

The VHS (\ref{eqn:introd_subbundle_G_rank5_X0}), constructed on the {\it metric side of the mirror}, is then proved to be $C^\infty$ {\it isomorphic} to the VHS induced by (\ref{eqn:introd_H3_Hodge-decomp}) and (\ref{eqn:introd_H3_Hodge-symm}) on the {\it complex-structure side of the mirror}. This $C^\infty$ isomorphism is actually {\it holomorphic} at the level of the $1$-dimensional parts of the two VHS's and {\it anti-holomorphic} at the level of the $4$-dimensional parts. This regularity meshes with the {\it sesquilinear} self-duality of the Iwasawa manifold highlighted in the next work [Pop17] of the author. This isomorphism will be obtained by proving (cf. Corollary \ref{Cor:holomorphic-bundle-isomorphisms}) that each of the two Hodge filtrations is $C^\infty$ {\it isomorphic} to the Hodge filtration $F^1{\cal H}^2(B)\supset F^2{\cal H}^2(B)$ of holomorphic vector bundles induced by the family of tori $(B_t)_{t\in\Delta_{[\gamma]}}$ over the moduli space $\Delta_{[\gamma]}$ of {\it essential deformations} of the Iwasawa manifold. 

 We also define explicitly (cf. Definition \ref{Def:G-cone-complexified}) a {\bf mirror map}
$$\widetilde{\cal M}:\Delta_{[\gamma]}\to\widetilde{\cal G}_X.$$
\noindent It has the property of taking the point $0\in\Delta_{[\gamma]}$ (i.e. the Iwasawa manifold $X=X_0$, the marked point in $\Delta_{[\gamma]}$) to the Aeppli cohomology class $[\omega_0^2]_A\in{\cal G}_X$ of the canonical Gauduchon metric $\omega_0$ on $X$ (the marked point in the Gauduchon cone ${\cal G}_X$). The mirror map $\widetilde{\cal M}$ is then proved in Theorem \ref{The:mirror-map} to be a {\it local biholomorphism} whose differential at $0\in\Delta_{[\gamma]}$ is the canonical isomorphism $H^{2,\,1}_{[\gamma]}(X,\,\C)\simeq H^{2,\,2}_A(X,\,\C)$ of Proposition \ref{Prop:introd_H21-gamma_H22_A}. The analogous statement holds at every $t\in\Delta_{[\gamma]}$ after we observe a canonical isomorphism $H^{2,\,1}_{[\gamma]}(X_t,\,\C)\simeq H^{2,\,2}_A(X_t,\,\C)$ (cf. Observation \ref{Obs:H21_gamma-H22_isomorphism}).

The mirror map is defined by ``complexification'' of what we call the {\bf positive mirror map} defined (cf. Definition \ref{Def:mirror-map_def}) by
$${\cal M}:\Delta_{[\gamma]}\to{\cal G}_X, \hspace{3ex} t\mapsto \bigg[(\omega_t^{1,\,1})^2\bigg]_A.$$

 We hope that these methods can be extended to other classes of compact complex manifolds. The ultimate goal is to get a general mirror symmetry theory asserting that every compact complex $n$-dimensional {\it sGG manifold} $X$ (possibly, but not necessarily, assumed to be $\partial\bar\partial$) whose canonical bundle $K_X$ is {\it trivial} and having some other familiar properties (e.g. {\it unobstructedness} of its Kuranishi family, degeneration at $E_2$ of its Fr\"olicher spectral sequence, etc) admits a {\it mirror dual} $\widetilde{X}$ such that the moduli space {\it EssDef}\,$(X)$ of {\it essential deformations of the complex structure} of $X$ (defined, e.g. using the space $E_2^{n-1,\,1}$ on the second page of the Fr\"olicher spectral sequence of $X$) corresponds via a local biholomorphism to the complexified Gauduchon cone $\widetilde{\cal G}_{\widetilde{X}}$ of $\widetilde{X}$ and vice versa. This local biholomorphism ought to induce an isomorphism of variations of Hodge structures parametrised respectively by {\it EssDef}\,$(X)$ and $\widetilde{\cal G}_{\widetilde{X}}$. This isomorphism may turn out to be {\it holomorphic} at the level of certain parts of the two VHS's and {\it anti-holomorphic} for the other parts. Certain non-linear PDEs (e.g. of the Monge-Amp\`ere or Hessian type) are expected to produce canonical metrics representing Aeppli cohomology classes in the Gauduchon cone. Some classes of nilmanifodls and solvmanifolds provide a fertile testing ground for this conjecture.

\vspace{2ex}

\noindent {\bf Acknowledgments.} Work on this paper started in Montr\'eal as a joint project with C. Mourougane when we were both visiting the UMI CNRS-CRM and the CIRGET at the UQ\`AM. The author is very grateful to C. Mourougane for contributing ideas, observations and expertise on various Hodge-theoretical topics, especially in the complex-structure part of the paper (sections $\S.$\ref{section:preliminaries} -- $\S.$\ref{section:coordinates_Delta}). Many thanks are also due to the CNRS for making the author's stay in Montr\'eal possible and to V. Apostolov, S. Lu and E. Giroux for stimulating discussions and their interest in this work.

\section{Preliminaries}\label{section:preliminaries}

We use the set up of the paragraph~\ref{subsection:Iwasawa}.
Recall that the $\C^3$-valued holomorphic $1$-form on $G$
$$G\ni M=\begin{pmatrix}1 & z_1 & z_3\\   
			  0 & 1 & z_2\\
		  0 & 0 & 1\end{pmatrix}  \mapsto M^{-1}\, dM = \begin{pmatrix}0 & dz_1 & dz_3-z_1\, dz_2\\   
			  0 & 0 & dz_2\\
		  0 & 0 & 0\end{pmatrix}  $$ 
\noindent is invariant under the action of $\Gamma$, hence descends to a holomorphic $1$-form on $X$ 
giving rise to three holomorphic $1$-forms $\alpha,\beta, \gamma$ on the Iwasawa manifold induced respectively by the forms $dz_1, dz_2, dz_3-z_1dz_2$ on $\C^3$. 
Thus, $\alpha, \beta, \gamma \in C^{\infty}_{1,\,0}(X,\,\C)$ and $$\bar\partial\alpha = \bar\partial\beta = \bar\partial\gamma = 0.$$ 
Since $dz_1, dz_2$ are closed and $d(dz_3-z_1dz_2) = - dz_1\wedge dz_2$, we get
\begin{equation}\label{eqn:alpha,beta,gamma_d}d\alpha = d\beta = 0 \hspace{2ex} 
\mbox{and}  \hspace{2ex} d\gamma = \partial\gamma = -\alpha\wedge\beta \neq 0 \hspace{3ex} \mbox{on}\hspace{1ex} X.\end{equation}
From the exact sequence
$$0\to\pi^\star\Omega^1_B\to\Omega^1_X\to\Omega^1_{X/B}\to 0,$$
as the map $H^1(\pi^\star\Omega^1_B)=H^1(\Oc_X)\otimes H^0(\pi^\star\Omega^1_B)\to H^1(\Oc_X)\otimes H^0(\Omega^1_X)=H^1(\Omega_X^1)$ is injective 
due to the triviality of $\Omega^1_B$ and $\Omega^1_X$, we get the simple presentation
$$0\to H^0(\pi^\star\Omega^1_B)\to H^0(\Omega^1_X)\to H^0(\Omega^1_{X/B})\to 0.$$
Thus, the form $\gamma$ is a representative of $H^0(\Omega^1_{X/B})$ in $H^0(\Omega^1_X)$. 
In other words, the forms $\alpha$ and $\beta$ are {\it horizontal} (i.e. coming from $B$), while $\gamma$ is {\it vertical} (i.e. lives on the fibres). 
It follows that the De Rham cohomology of $X$ reads
\begin{eqnarray}\label{eqn:DeRham}
\nonumber H^1(X,\C)&=&\bigg\langle \{\alpha\}_{DR},\, \{\beta\}_{DR} , \{\bar\alpha\}_{DR},\, \{\bar\beta\}_{DR}\bigg\rangle
=\pi^\star H^1(B,\C),\\
\nonumber\pi^\star H^2(B,\C)&=&\bigg\langle \{\alpha\wedge\bar\alpha\}_{DR},\, \{\alpha\wedge\bar\beta\}_{DR} , \{\beta\wedge\bar\alpha\}_{DR},\, 
\{\beta\wedge\bar\beta\}_{DR}\bigg\rangle 
\simeq  H^{1,\,1}_{BC}(X,\,\C) \simeq \pi^\star H^{1,\,1}(B,\C)\\
\nonumber H^2(X,\C)&=&\pi^\star H^2(B,\C)\oplus\bigg\langle \{\gamma\wedge\alpha\}_{DR},\, \{\gamma\wedge\beta\}_{DR} \bigg\rangle
\oplus\bigg\langle \{\bar\gamma\wedge\bar\alpha\}_{DR},\, \{\bar\gamma\wedge\bar\beta\}_{DR}\bigg\rangle,\\
\nonumber \pi^\star H^3(B,\C)&=&0,\\
\nonumber H^3(X,\C)&=&\bigg\langle\{\alpha\wedge\beta\wedge\gamma\}_{DR}\bigg\rangle\oplus
\{\gamma\wedge\pi^\star H^{1,\,1}(B,\C)\}_{DR}\oplus\{\bar\gamma\wedge\pi^\star H^{1,\,1}(B,\C)\}_{DR}\\
\nonumber & & \oplus \ \bigg\langle\{\bar\alpha\wedge\bar\beta\wedge\bar\gamma\}_{DR}\bigg\rangle\\
\nonumber H^4_{DR}(X,\C) & = & \bigg\langle \{\alpha\wedge\beta\wedge\gamma\wedge\bar\alpha\}_{DR},\, \{\alpha\wedge\beta\wedge\gamma\wedge\bar\beta\}_{DR}\bigg\rangle \\
\nonumber & &\oplus \ \bigg\langle \{\alpha\wedge\gamma\wedge\bar\alpha\wedge\bar\gamma\}_{DR},\, \{\alpha\wedge\gamma\wedge\bar\beta\wedge\bar\gamma\}_{DR},
\, \{\beta\wedge\gamma\wedge\bar\alpha\wedge\bar\gamma\}_{DR},\, \{\beta\wedge\gamma\wedge\bar\beta\wedge\bar\gamma\}_{DR} \bigg\rangle \\
  & &\oplus \ \bigg\langle \{\alpha\wedge\bar\alpha\wedge\bar\beta\wedge\bar\gamma\}_{DR}, \, \{\beta\wedge\bar\alpha\wedge\bar\beta\wedge\bar\gamma\}_{DR}\bigg\rangle.
\end{eqnarray}

Since the holomorphic $1$-form $\gamma$ is not closed, the Fr\"olicher spectral sequence of $X$ does not degenerate at $E_1$. 
Furthermore, thanks to (\ref{eqn:alpha,beta,gamma_d}), the triple Massey product of the De Rham cohomology classes $\{\alpha\}, \{\beta\}, \{\beta\}\in H^1_{DR}(X,\,\C)$ is
$$\langle\alpha,\beta,\beta\rangle=\{\beta\wedge\gamma\}_{DR}\in H^2(X,\C)\slash\{\alpha\}\cup H^1(X,\,\C) + \{\beta\}\cup H^1(X,\C).$$ 
\noindent Thus, $\langle\alpha,\beta,\beta\rangle\neq 0$ thanks to (\ref{eqn:DeRham}) and to $\{\alpha\}\cup H^1(X,\,\C) + \{\beta\}\cup H^1(X,\C)=\pi^\star H^2(B,\C)$. 
Therefore, no complex structure on the $C^{\infty}$ manifold underlying the Iwasawa manifold (in particular, no deformation of $X$) is K\"ahler or even $\partial\bar\partial$.
However, the Iwasawa manifold supports {\it balanced} metric structures, as is well known (cf. e.g. [AB91, Remark 3.1]).

It is known ([Nak75], [Sch07], [Ang11]) that the forms $\alpha, \beta, \gamma$ generate the entire cohomology of $X$. 
For example, we shall need the following descriptions in terms of generators of the following cohomology groups:
\begin{eqnarray}\label{eqn:cohomology_explicit}
\nonumber H^{1,\,0}_{\bar\partial}(X,\C) & = & \bigg\langle [\alpha]_{\bar\partial},\, [\beta]_{\bar\partial},\, [\gamma]_{\bar\partial}\bigg\rangle, \hspace{2ex} H^{0,\,1}_{\bar\partial}(X,\C) = \bigg\langle [\overline{\alpha}]_{\bar\partial},\, [\overline{\beta}]_{\bar\partial}\bigg\rangle=\pi^\star H^{0,\,1}_{\bar\partial}(B,\C),\\
\nonumber H^{1,\,1}_{\bar\partial}(X,\,\C) & = & \bigg\langle [\alpha\wedge\overline{\alpha}]_{\bar\partial},\, [\alpha\wedge\overline{\beta}]_{\bar\partial},\, [\beta\wedge\overline{\alpha}]_{\bar\partial},\, [\beta\wedge\overline{\beta}]_{\bar\partial},\,  [\gamma\wedge\overline{\alpha}]_{\bar\partial},\, [\gamma\wedge\overline{\beta}]_{\bar\partial}\bigg\rangle,\\
\nonumber H^{3,\,0}_{\bar\partial}(X,\C) & = & \bigg\langle [\alpha\wedge\beta\wedge\gamma]_{\bar\partial}\bigg\rangle, \hspace{2ex} H^{0,\,3}_{\bar\partial}(X,\C) = \bigg\langle [\overline{\alpha}\wedge\overline{\beta}\wedge\overline{\gamma}]_{\bar\partial}\bigg\rangle,  \\
\nonumber H^{2,\,1}_{\bar\partial}(X,\,\C) & = & \bigg\langle [\alpha\wedge\gamma\wedge\overline{\alpha}]_{\bar\partial},\, [\alpha\wedge\gamma\wedge\overline{\beta}]_{\bar\partial},\, [\beta\wedge\gamma\wedge\overline{\alpha}]_{\bar\partial},\, [\beta\wedge\gamma\wedge\overline{\beta}]_{\bar\partial}\bigg\rangle \oplus  \bigg\langle [\alpha\wedge\beta\wedge\overline{\alpha}]_{\bar\partial},\, [\alpha\wedge\beta\wedge\overline{\beta}]_{\bar\partial}\bigg\rangle \\
  & = &  [\gamma\wedge\pi^\star H^{1,\,1}_{\bar\partial}(B,\,\C)]_{\bar\partial}\oplus \pi^\star H^{2,\,1}_{\bar\partial}(B,\,\C),  \\
\nonumber H^{1,\,2}_{\bar\partial}(X,\,\C) & = & \bigg\langle [\alpha\wedge\overline{\alpha}\wedge\overline{\gamma}]_{\bar\partial},\, [\beta\wedge\overline{\alpha}\wedge\overline{\gamma}]_{\bar\partial},\, [\alpha\wedge\overline{\beta}\wedge\overline{\gamma}]_{\bar\partial},\, [\beta\wedge\overline{\beta}\wedge\overline{\gamma}]_{\bar\partial}\bigg\rangle \oplus  \bigg\langle [\gamma\wedge\overline{\alpha}\wedge\overline{\gamma}]_{\bar\partial},\, [\gamma\wedge\overline{\beta}\wedge\overline{\gamma}]_{\bar\partial}\bigg\rangle.
\end{eqnarray}

\section{Deformations}

\subsection{The Calabi-Yau isomorphism}\label{subsection:C-Y_isomorphism}

Since $T^{1,\,0}X$ is trivial, the Iwasawa manifold $X$ is, in particular, a Calabi-Yau manifold. 
Since its Kuranishi family $(X_t)_{t\in\Delta}$ is {\it unobstructed} by Nakamura [Nak75], 
its base $\Delta$ can be identified with an open ball in the Dolbeault cohomology group $H^{0,\,1}(X,\,T^{1,\,0}X)$ 
of classes of smooth $\bar\partial$-closed $(0,\,1)$-forms with values in the holomorphic tangent bundle $T^{1,\,0}X$. 
In particular, the holomorphic tangent space $T^{1,\,0}_0\Delta$ to $\Delta$ at $0$ is isomorphic, via the Kodaira-Spencer map $\rho$, to $H^{0,\,1}(X,\,T^{1,\,0}X)$.

On the other hand, the Calabi-Yau structure of $X$ is defined by any nowhere-vanishing holomorphic $(3,\,0)$-form $\Omega$ on $X$. 
All such forms are equal up to a multiplicative constant, so we may choose, for example, $\Omega:=\alpha\wedge\beta\wedge\gamma$. 
We get the following isomorphisms, the second of which will be called the {\it Calabi-Yau isomorphism}: 
\begin{eqnarray}\label{eqn:CY_isomorphism} 
T^{1,\,0}_0\Delta \underset{\simeq}{\overset{\rho}{\longrightarrow}} H^{0,\,1}(X,\,T^{1,\,0}X) & \underset{\simeq}{\overset{T_{\Omega}}{\longrightarrow}} 
& H^{2,\,1}_{\bar\partial}(X,\,\C),  \\
\nonumber [\theta]_{\bar\partial} & \longmapsto &  [\theta\lrcorner \Omega]_{\bar\partial}.
\end{eqnarray} 

The Calabi-Yau isomorphism can be described explicitly in the case of the Iwasawa manifold. 
Let $\xi_{\alpha}, \xi_{\beta}, \xi_{\gamma}\in H^0(X,\,T^{1,\,0}X)$ be the frame of holomorphic vector fields of type $(1,\,0)$ dual to the frame $\{\alpha,\,\beta,\,\gamma\}$. 
Thus, 
\begin{equation*}
\xi_{\alpha}=p_\star\bigg(\frac{\partial}{\partial z_1}\bigg), \xi_{\beta}=p_\star\bigg(\frac{\partial}{\partial z_2}+z_1\frac{\partial}{\partial z_3}\bigg)
\hspace{2ex} \mbox{and} \hspace{2ex} \xi_{\gamma}=p_\star\bigg(\frac{\partial}{\partial z_3}\bigg)
\end{equation*}
where $p_\star$ stands for the differential of the quotient map $p: G\to X$.

Now, $T^{1,\,0}X$ being trivial, $H^{0,\,1}(X,\,T^{1,\,0}X)=H^{0,\,1}(X,\,\C)\otimes H^0(X,\,T^{1,\,0}X)$
is generated (cf. [Nak75]) by the Dolbeault cohomology classes 
\begin{equation}\label{eqn:H^01_generators}H^{0,\,1}(X,\,T^{1,\,0}X)=\bigg\langle[\overline{\alpha}\otimes\xi_{\alpha}], [\overline{\alpha}\otimes\xi_{\beta}], 
[\overline{\alpha}\otimes\xi_{\gamma}], [\overline{\beta}\otimes\xi_{\alpha}], [\overline{\beta}\otimes\xi_{\beta}], 
[\overline{\beta}\otimes\xi_{\gamma}]\bigg\rangle.\end{equation}
\noindent  In particular, $\dim_{\C}H^{0,\,1}(X,\,T^{1,\,0}X) = 6$, so the Kuranishi family of $X$ is $6$-dimensional. 
The images under the Calabi-Yau isomorphism $T_{\Omega}$ of these generators 
are $[(\overline{\alpha}\otimes\xi_{\alpha})\lrcorner(\alpha\wedge\beta\wedge\gamma)]_{\bar\partial} 
= [\beta\wedge\gamma\wedge\overline{\alpha}]_{\bar\partial}$ 
and its analogues for the remaining five generators, hence the description of $H^{2,\,1}_{\bar\partial}(X,\,\C)$ in (\ref{eqn:cohomology_explicit}).

For future reference, we recall the following standard piece of notation. We let $\alpha_1=\alpha, \alpha_2=\beta, \xi_1=\xi_\alpha, \xi_2=\xi_\beta, \xi_3=\xi_\gamma$ and denote by $t_{i\lambda}$, with $1\leq\lambda\leq 2$ and $1\leq i\leq 3$, the coordinates induced on $H^{0,\,1}(X,T^{1,\,0}X)$ by the basis $([\overline{\alpha}_\lambda\otimes\xi_i])_{1\leq\lambda\leq 2\atop 1\leq i\leq 3}$ (cf. (\ref{eqn:H^01_generators})). Since $\Delta$ is an open ball about the origin in $H^{0,\,1}(X,T^{1,\,0}X)$, we can view $(t_{11},\,t_{12},\,t_{21},\,t_{22},\,t_{31},\,t_{32})$ as coordinates on $\Delta$. Thus, the points $t\in\Delta\subset H^{0,\,1}(X,T^{1,\,0}X)$ can be written uniquely as

$$t=\sum _{1\leq\lambda\leq 2\atop 1\leq i\leq 3} t_{i\lambda}\,\overline{\alpha}_\lambda\xi_i\in H^{0,\,1}(X,T^{1,\,0}X).$$

\subsection{The essential deformations}\label{subsection:essential-def}

The sequence of low-degree terms in the Leray spectral sequence induced by $\pi$ and $TX$ (the sheaf associated with the holomorphic tangent bundle $T^{1,\,0}X$) 
whose second page is given by $E_2^{p,q}=H^p(B,\Rc^q\pi_\star TX)$, together with the cohomologies of the short exact sequence 
$$0\to T\pi\to TX\to \pi^\star TB\to 0$$
\noindent defining the relative tangent bundle to the submersion $\pi$, reads~\footnote{The notation used here refers to sheaves. 
We shall often use in the sequel the vector-bundle notation. For example, $H^1(B,\,TB)$ (in sheaf notation) coincides with $H^{0,\,1}(B,\,T^{1,\,0}B)$ (in vector-bundle notation).}
$$
  \xymatrix{ 
0\ar[r]  &H^1(B,\pi_\star T\pi)\ar[d] \ar[r]           &     H^1(X,T\pi)\ar[d] \ar[r]                           & H^0(B,\Rc^1\pi_\star T\pi)\ar[d]            
  \\
0\ar[r]&H^1(B,\pi_\star TX)\ar[d]\ar[r]^{iso}&H^1(X,TX)\ar[d] \ar[r]&H^0(B,\Rc^1\pi_\star TX)\ar[d]\\
0\ar[r]  &H^1(B, TB)  \ar[r]^{iso}         &      H^1(X,\pi^\star TB) \ar[r]   \ar@/_2em/[u]_L        & H^0(B,\Rc^1\pi_\star\Oc_X\otimes TB).             
  }
$$

As $TX$ is trivial and as all $(0,1)$-Dolbeault cohomology classes on $X$ come from classes on $B$ 
(i.e. in terms of the Leray filtration, we have $H^1(X,\Oc_X)=\pi^\star H^1(B,\Oc_B)=F^1H^1(X,\Oc_X)$), 
the horizontal map $H^1(B,\pi_\star TX)\to H^1(X,TX)$ is an isomorphism. As $\bar\gamma\otimes\xi_{\cdot}$ is $\bar\partial_{\pi}$-closed 
(i.e. $\bar\partial(\bar\gamma\otimes\xi_{\cdot})=-\bar\alpha\wedge\bar\beta\otimes\xi_{\cdot}$ vanishes on the fibres of $\pi$), 
it defines an element in $H^0(B,\Rc^1\pi_\star T\pi)$, i.e. a deformation of the fibres of $\pi$. 
However, since $\bar\gamma\otimes\xi_{\cdot}$ is not $\bar\partial$-closed, this does not lift to a global deformation of $X$.

\noindent Now, consider the quotient map 
$$H^{0,\,1}(X,T^{1,\,0}X)=H^{0,\,1}(X)\otimes H^0(X,TX)\to H^{0,\,1}(X)\otimes H^0(X,\pi^\star TB)=H^{0,\,1}(X,\pi^\star T^{1,\,0}B)$$
given by the differential of the submersion $\pi$ 
and choose its lift $L :H^{0,\,1}(X,\pi^\star T^{1,\,0}B) \to H^{0,\,1}(X,T^{1,\,0}X)$ defined by
$$L : (\pi\circ p)_\star\bigg(\frac{\partial}{\partial z_1}\bigg)\mapsto \xi_{\alpha}, \hspace{3ex} (\pi\circ p)_\star\bigg(\frac{\partial}{\partial z_2}\bigg)\mapsto\xi_{\beta}.$$

\noindent Consider the subspace of $H^{0,\,1}(X,\,T^{1,\,0}X)$ defined by 
\begin{equation}\label{eqn:H^01_gamma_generators}
H^{0,\,1}_{[\gamma]}(X,\,T^{1,\,0}X):= L H^{0,\,1}(X,\,\pi^\star T^{1,\,0} B)=
\bigg\langle [\overline{\alpha}\otimes\xi_{\alpha}],\, [\overline{\alpha}\otimes\xi_{\beta}],\, [\overline{\beta}\otimes\xi_{\alpha}],\,
[\overline{\beta}\otimes\xi_{\beta}]\bigg\rangle
\end{equation}
This amounts to singling out, for every first-order deformation of $B$ 
(i.e. for every element of $H^{0,\,1}(B,T^{1,\,0}B)$), a suitable first-order automorphism in $H^1(B,\pi_\star T\pi)$ of the fibres of $\pi$.

\begin{Lem}\label{Lem:H21_gamma_contraction} 
The map $H^{0,\,1}(X,\,T^{1,\,0}X) \stackrel{\cdot\lrcorner[\gamma]_{\bar\partial}}{\longrightarrow} H^{0,\,1}_{\bar\partial}(X,\,\C), 
\hspace{2ex} [\theta]\mapsto[\theta\lrcorner\gamma]_{\bar\partial},$ is well defined and its kernel is precisely $H^{0,\,1}_{[\gamma]}(X,\,T^{1,\,0}X)$, i.e.
\begin{equation}\label{eqn:H21_gamma_contraction}H^{0,\,1}_{[\gamma]}(X,\,T^{1,\,0}X)=\bigg\{[\theta]\in H^{0,\,1}(X,\,T^{1,\,0}X)
\bigg\slash [\theta\lrcorner\gamma] = 0\in H^{0,\,1}_{\bar\partial}(X,\,\C)\bigg\}.\end{equation}
\end{Lem}

\begin{proof}

For every $[\theta]\in H^{0,\,1}(X,\,T^{1,\,0}X)$, we have $\bar\partial(\theta\lrcorner\gamma) = (\bar\partial\theta)\lrcorner\gamma + \theta\lrcorner(\bar\partial\gamma) = 0$ 
since $\bar\partial\theta = 0$ (where $\bar\partial$ is the canonical $(0,\,1)$-connection of the holomorphic vector bundle $T^{1,\,0}X$ 
and $\theta$ is viewed as a $\bar\partial$-closed $(0,\,1)$-form with values in this bundle) and $\bar\partial\gamma=0$. 
Thus, $\theta\lrcorner\gamma$ defines indeed a Dolbeault cohomology class of type $(0,\,1)$ 
which furthermore is independent of the choice of representative $\theta$ of the class $[\theta]\in H^{0,\,1}(X,\,T^{1,\,0}X)$. 
To see this last point, take two cohomologous $\theta_1, \theta_2$. Then, $\theta_1 - \theta_2 = \bar\partial\xi$ for some $\xi\in C^{\infty}(X,\,T^{1,\,0}X)$.
We have $\bar\partial(\xi\lrcorner\gamma) = (\bar\partial\xi)\lrcorner\gamma - \xi\lrcorner(\bar\partial\gamma) = (\bar\partial\xi)\lrcorner\gamma$. 
This proves the well-definedness of the map $\cdot\lrcorner[
\gamma]_{\bar\partial}$. Identity (\ref{eqn:H21_gamma_contraction}) follows at once from (\ref{eqn:H^01_generators}) and (\ref{eqn:H^01_gamma_generators}).  

\end{proof}

\begin{Def}\label{Def:Delta_gamma} Bearing in mind that $\Delta\subset H^{0,\,1}(X,\,T^{1,\,0}X)$ is an open subset, let 
\begin{equation*}
\Delta_{[\gamma]}:=\Delta\cap H^{0,\,1}_{[\gamma]}(X,\,T^{1,\,0}X).\end{equation*} 

\end{Def}

So formally, thanks to (\ref{eqn:H21_gamma_contraction}) and  by analogy with polarising $(1,\,1)$-classes
\footnote{Recall that in the standard case of a K\"ahler class $[\omega]$ on $X_0$, the fibres $X_t$ {\it polarised} by $[\omega]$,
i.e. the fibres $X_t$ for which $[\omega]$ remains of $J_t$-type $(1,\,1)$, are precisely those corresponding to $[\theta]\in H^{0,\,1}(X_0,\,T^{1,\,0}X_0)$ 
satisfying the condition $[\theta\lrcorner\omega]=0$ in $H^{0,\,2}(X_0,\,\C)$.}, 
the family of deformations $(X_t)_{t\in\Delta_{[\gamma]}}$ is ``polarised'' by the $(1,\,0)$-class $[\gamma]_{\bar\partial}\in H^{1,\,0}_{\bar\partial}(X,\,\C)$.

It follows from Nakamura's description of the Kuranishi family of the Iwasawa manifold ([Nak75, p. 96]) 
that the manifolds $X_t$ with $t\in\Delta_{[\gamma]}\setminus\{0\}$ are contained in the union of Nakamura's classes $(ii)$ and $(iii)$. 
They are not complex parallelisable. 
Meanwhile, the removed deformations $X_t$ with $t\in\Delta\setminus\{0\}$ corresponding to 
$[\theta\lrcorner\Omega]\in\langle[\alpha\wedge\beta\wedge\overline{\alpha}]_{\bar\partial},\, 
[\alpha\wedge\beta\wedge\overline{\beta}]_{\bar\partial}\rangle \subset H^{2,\,1}_{\bar\partial}(X,\,\C)$ make up Nakamura's class $(i)$. 
They are all complex parallelisable (and, in a sense, have the {\it same} geometry as the Iwasawa manifold $X=X_0$). 
So, no geometric information is lost by these removals.
For this reason, we call $(X_t)_{t\in\Delta_{[\gamma]}}$ the local universal family of {\bf essential deformations} of $X$.

In terms of coordinates, we see that $(t_{11},\,t_{12},\,t_{21},\,t_{22})$ define coordinates on $\Delta_{[\gamma]}$. Consequently, the points $t\in\Delta_{[\gamma]}\subset H^{0,\,1}_{[\gamma]}(X,T^{1,\,0}X)$ can be written uniquely as

$$t=\sum _{1\leq\lambda\leq 2\atop 1\leq i\leq 2} t_{i\lambda}\,\overline{\alpha}_\lambda\xi_i\in H^{0,\,1}_{[\gamma]}(X,T^{1,\,0}X).$$

\section{Weight-three Hodge decomposition}\label{section:Hodge-decomp}

\subsection{The $(3,0)$-part}\label{subsection:3-0}

We start with a simple general observation.
\begin{Lem}\label{Lem:Hn0_injection} Let $Y$ be an arbitrary compact complex manifold with $\mbox{dim}_{\C}Y=n$. 
Then, there is a  canonical injection $H^{n,\,0}_{\bar\partial}(Y,\,\C)\hookrightarrow H^n_{DR}(Y,\,\C)$.
\end{Lem}


\begin{proof}
It is clear that $H^{n,\,0}_{\bar\partial}(Y,\,\C) = C^{\infty}_{n,\,0}(Y,\,\C)\cap\ker\bar\partial$ since every Dolbeault cohomology class $[u]_{\bar\partial}$ 
of bidegree $(n,\,0)$ has a unique representative $u$. Indeed, zero is the only $\bar\partial$-exact $(n,\,0)$-form. 
Moreover, every such $(n,0)$-form $u$ is $d$-closed since $\partial u=0$ for bidegree reasons. Therefore, the following map is well defined~:
\begin{equation*}H^{n,\,0}_{\bar\partial}(Y,\,\C)\longrightarrow H^n_{DR}(Y,\,\C), \hspace{3ex} [u]_{\bar\partial}\longmapsto\{u\}_{DR}.\end{equation*}
\noindent It remains to prove that this map is injective, i.e. that $u=0$ whenever $u$ is $d$-exact. Suppose that for a $\bar\partial$-closed $(n,\,0)$-form $u$, 
we have $u=dv$. Then $u = \partial v$ for bidegree reasons. Hence 
$$0\leq\int_Yi^{n^2}\,u\wedge\bar{u} = \int_Yi^{n^2}\,u \wedge \bar\partial\overline{v} 
= (-1)^n\,i^{n^2}\,\int_Y\bar\partial(u\wedge\overline{v})
= (-1)^n\,i^{n^2}\,\int_Y d (u\wedge\overline{v})=0,$$ 
\noindent where the last identity follows from Stokes's theorem.

Since the smooth $(n,\,n)$-form $i^{n^2}\,u\wedge\bar{u}$ is non-negative at every point, this can only happen if $i^{n^2}\,u\wedge\bar{u} = 0$ at every point. 
We get $u=0$ on $X$. Indeed, writing $u=f\,dz_1\wedge\dots\wedge dz_n$ in local coordinates, 
we see that $i^{n^2}\,u\wedge\bar{u} = |f|^2\,idz_1\wedge d\bar{z}_1\wedge\dots\wedge idz_n\wedge d\bar{z}_n$, hence $f=0$ in our situation.  

\end{proof}

\subsection{The $(2,1)$-part: definition of $H^{2,\,1}_{[\gamma]}(X,\,\C)$}\label{subsection:2-1}

The space $H^{2,\,1}_{\bar\partial}(X,\,\C)$ does not inject canonically into $H^3_{DR}(X,\,\C)$ as can be seen from (\ref{eqn:DeRham}) and (\ref{eqn:cohomology_explicit}), 
so there is no standard Hodge decomposition for $H^3_{DR}(X,\,\C)$ on the Iwasawa manifold $X$. 
This can also be seen by a simple dimension count: $b_3=10$, while $h^{3,\,0} + h^{2,\,1} + h^{1,\,2} + h^{0,\,3} = 1 + 6 + 6 + 1 =14 >10$. 
However, we shall shrink the Dolbeault cohomology group of bidegree $(2,\,1)$ in order to make it {\it fit} into $H^3_{DR}(X,\,\C)$ 
and shall thus obtain a corresponding Hodge decomposition of weight $3$ that will be seen to have a precise geometric meaning in terms of the essential deformations of $X$ 
defined in $\S.$\ref{subsection:essential-def}.

\begin{Def}
The $4$-dimensional subspace $H^{2,\,1}_{[\gamma]}(X,\,\C)$ of $H^{2,\,1}_{\bar\partial}(X,\,\C)$ is defined as the image 
of $H^{0,\,1}_{[\gamma]}(X,\,T^{1,\,0}X)= L H^{0,\,1}(X,\,\pi^\star T^{1,\,0} B)$ 
under the Calabi-Yau isomorphism $T_{\Omega}: H^{0,\,1}(X,\,T^{1,\,0}X)\to H^{2,\,1}_{\bar\partial}(X,\,\C)$.

\end{Def}

Thus, by (\ref{eqn:H^01_gamma_generators}), we get
\begin{equation}\label{eqn:H21_gamma_def} 
H^{2,\,1}_{[\gamma]}(X,\,\C)= \bigg\langle [\alpha\wedge\gamma\wedge\overline{\alpha}]_{\bar\partial},\, [\alpha\wedge\gamma\wedge\overline{\beta}]_{\bar\partial},\, 
[\beta\wedge\gamma\wedge\overline{\alpha}]_{\bar\partial},\, [\beta\wedge\gamma\wedge\overline{\beta}]_{\bar\partial}\bigg\rangle
=[\gamma\wedge\pi^\star H^{1,1}(B,\C)]_{\bar\partial}.\end{equation}
\noindent We see from~\eqref{eqn:DeRham} that $H^{2,\,1}_{[\gamma]}(X,\,\C)$ injects into $H^3_{DR}(X,\C)$. Note that, since $[\xi_{\gamma}\lrcorner \Omega]=[\alpha\wedge\beta]=[-d\gamma]=0\in H^2_{DR}(X,\C)$ while $\bar\alpha$ and $\bar\beta$ are closed, 
the image of  $H^{2,\,1}_{[\gamma]}(X,\,\C)$ in $H^3_{DR}(X,\C)$ does depend neither on the choice of the lift $L$~\footnote{It is also the image of the map
$H^0(X,\Omega^1_{X/B})\times \pi^\star H^{1,1}_{\bar\partial}(B,\C)\to H^3_{DR}(X,\C), ([\gamma],[u]_{\bar\partial})
\mapsto\{u\wedge\gamma\}_{DR}$} nor on the choice of $[\gamma]$ in $\frac{H^0(X,\Omega^1_X)}{H^0(\pi^\star\Omega^1_B)}$.

 We get isomorphisms 
\begin{equation*}
T^{1,\,0}_0\Delta_{[\gamma]} \underset{\simeq}{\overset{\rho}{\longrightarrow}} H^{0,\,1}_{[\gamma]}(X,\,T^{1,\,0}X) \underset{\simeq}{\overset{T_{\Omega}}{\longrightarrow}} 
T_{\Omega}(H^{0,\,1}_{[\gamma]}(X,\,T^{1,\,0}X))=:H^{2,\,1}_{[\gamma]}(X,\,\C).
\end{equation*}

\subsection{The $(1,2)$-part}\label{subsection:1-2}

Recall now that if a Hermitian metric $\omega$ has been fixed on an arbitrary compact complex $n$-dimensional manifold $Y$, 
the corresponding Hodge star operator $\star$ (defined by $u\wedge\star\bar{v}=\langle u,v\rangle_\omega\, dV_\omega$) 
leads to the following  isomorphisms for every bidegree $(p,\,q)$~:
\begin{equation*}
\iota : H^{p,\,q}_{\bar\partial}(Y,\,\C)\stackrel{\simeq}{\longrightarrow} H^{n-q,\,n-p}_{\partial}(Y,\,\C)\stackrel{\simeq}{\longrightarrow} {H^{n-p,\,n-q}_{\bar\partial}(Y,\,\C)}.
\end{equation*}

\noindent Indeed, the first  isomorphism is given by $\star$ since $\star\,\Delta'' = \Delta'\,\star$, while the second one, which is $\C$-anti-linear, is defined by conjugation.

In our case, $n=3$ and the Iwasawa manifold $X_0$ is endowed with the canonical metric 
\begin{eqnarray}\label{eqn:Omega_def}
 \omega = \omega_0:=i\alpha\wedge\bar\alpha+i\beta\wedge\bar\beta+i\gamma\wedge\bar\gamma,
\end{eqnarray}
  so we get
\begin{equation*} 
H^{3,\,0}_{\bar\partial}(X,\,\C)\stackrel{\simeq}{\longrightarrow} {H^{0,\,3}_{\bar\partial}(X,\,\C)} \hspace{2ex} 
\mbox{and} \hspace{2ex} H^{2,\,1}_{\bar\partial}(X,\,\C)\stackrel{\simeq}{\longrightarrow} {H^{1,\,2}_{\bar\partial}(X,\,\C)}.
\end{equation*}

Accordingly, we define   \begin{eqnarray}\label{eqn:H12_gamma_def}\nonumber 
H^{1,\,2}_{[\gamma]}(X,\,\C): & = &\iota H^{2,\,1}_{[\gamma]}(X,\,\C) = \bigg\langle  [\star(\beta\wedge\overline{\beta}\wedge\overline{\gamma})]_{\bar\partial},
\, [\star(\alpha\wedge\overline{\beta}\wedge\overline{\gamma})]_{\bar\partial},
\, [\star(\beta\wedge\overline{\alpha}\wedge\overline{\gamma})]_{\bar\partial},
\, [\star(\alpha\wedge\overline{\alpha}\wedge\overline{\gamma})]_{\bar\partial}
\bigg\rangle \\
   & = & \bigg\langle  [\beta\wedge\overline{\beta}\wedge\overline{\gamma}]_{\bar\partial},
\, [\alpha\wedge\overline{\beta}\wedge\overline{\gamma}]_{\bar\partial},
\, [\beta\wedge\overline{\alpha}\wedge\overline{\gamma}]_{\bar\partial},
\, [\alpha\wedge\overline{\alpha}\wedge\overline{\gamma}]_{\bar\partial}
\bigg\rangle    \subset H^{3}_{DR}(X,\,\C),\end{eqnarray}

\noindent where $\star = \star_\omega$ is the Hodge star operator associated with $\omega$. The fact that $\star$ can be dropped from the above definition of $H^{1,\,2}_{[\gamma]}(X,\,\C)$ to give the expression on the second line follows from Lemma \ref{Lem:star_generators_21_gamma} below.

\begin{Prop}\label{Prop:H21_gamma_Hodge} 
Let $X$ be the Iwasawa manifold. 

\noindent There are canonical injections $H^{2,\,1}_{[\gamma]}(X,\,\C)\hookrightarrow H^3_{DR}(X,\,\C)$ and $H^{1,\,2}_{[\gamma]}(X,\,\C)\hookrightarrow H^3_{DR}(X,\,\C)$ 
giving rise to a canonical isomorphism
\begin{equation}\label{eqn:H3_Hodge-decomp}H^3_{DR}(X,\,\C)\simeq H^{3,\,0}_{\bar\partial}(X,\,\C) 
\oplus H^{2,\,1}_{[\gamma]}(X,\,\C) \oplus H^{1,\,2}_{[\gamma]}(X,\,\C) \oplus H^{0,\,3}_{\bar\partial}(X,\,\C)\end{equation}
\noindent that will be called the {\bf essential weight-three Hodge decomposition} of the Iwasawa manifold. 
Moreover, there are canonical isomorphisms given by conjugation
\begin{equation}\label{eqn:H3_Hodge-symm}
H^{3,\,0}_{\bar\partial}(X,\,\C)\simeq\overline{H^{0,\,3}_{\bar\partial}(X,\,\C)} \hspace{2ex} 
\mbox{and} \hspace{2ex} H^{2,\,1}_{[\gamma]}(X,\,\C)\simeq\overline{H^{1,\,2}_{[\gamma]}(X,\,\C)}
~\footnote{\textrm{Note from  the explicit descriptions in~\eqref{eqn:cohomology_explicit}
that this isomorphism does not hold for the full Dolbeault cohomology}}
\end{equation}    
\noindent that will be called the {\bf essential weight-three Hodge symmetry} of the Iwasawa manifold.     
\end{Prop}

\begin{proof} The canonical injections follow obviously from the descriptions (\ref{eqn:H21_gamma_def}), (\ref{eqn:H12_gamma_def}) 
and (\ref{eqn:cohomology_explicit}) of $H^{2,\,1}_{[\gamma]}(X,\,\C)$, $H^{1,\,2}_{[\gamma]}(X,\,\C)$ and resp. $H^3_{DR}(X,\,\C)$. 
On the other hand, $H^{3,\,0}_{\bar\partial}(X,\,\C)$ injects canonically into $H^3_{DR}(X,\,\C)$ by Lemma \ref{Lem:Hn0_injection}, 
while $H^{0,\,3}_{\bar\partial}(X,\,\C)$ injects canonically thanks to its explicit description in (\ref{eqn:cohomology_explicit}). 
Since the images in $H^3_{DR}(X,\,\C)$ of $H^{3,\,0}_{\bar\partial}(X,\,\C), H^{2,\,1}_{[\gamma]}(X,\,\C), H^{1,\,2}_{[\gamma]}(X,\,\C), H^{0,\,3}_{\bar\partial}(X,\,\C)$ 
are mutually transversal by the explicit description of the injections 
and since $10= \mbox{dim}\, H^3 = \mbox{dim}\, H^{3,\,0} + \mbox{dim}\, H^{2,\,1}_{[\gamma]} + \mbox{dim}\, H^{1,\,2}_{[\gamma]} + \mbox{dim}\, H^{0,\,3} = 1+4+4+1$, 
we get the isomorphism (\ref{eqn:H3_Hodge-decomp}). The isomorphisms (\ref{eqn:H3_Hodge-symm}) 
follow from (\ref{eqn:cohomology_explicit}), (\ref{eqn:H21_gamma_def}) and (\ref{eqn:H12_gamma_def}). 
\end{proof}

\subsection{Hodge decomposition for small essential deformations of $X$}

 Recall that $\Delta_{[\gamma]} = \{t\in\Delta\,\mid\, t_{31}=t_{32}=0\}$, so $(t_{11},\,t_{12},\,t_{21},\,t_{22})$ are coordinates on $\Delta_{[\gamma]}$.

\begin{Prop}\label{Prop:H2-1_gamma_vbundle} 
Let $(X_t)_{t\in\Delta}$ be the Kuranishi family of the Iwasawa manifold $X=X_0$. 
Then the space $H^{2,\,1}_{[\gamma]}(X,\,\C) = H^{2,\,1}_{[\gamma]}(X_0,\,\C)$ described in (\ref{eqn:H21_gamma_def}) is the fibre over $t=0$ of a $C^{\infty}$ vector bundle $\Delta_{[\gamma]}\ni t\mapsto H^{2,\,1}_{[\gamma]}(X_t,\,\C)$ of rank $4$ on $\Delta_{[\gamma]}$ 
that will be denoted by ${\cal H}^{2,\,1}_{[\gamma]}$. 
\end{Prop}

\begin{proof}

Recall that by [Nak75, p. 95], for $t=\sum _{1\leq\lambda\leq 2\atop 1\leq i\leq 3}
t_{i\lambda}\,\overline{\alpha}_\lambda\xi_i\in H^{0,\,1}(X_0,T^{1,\,0}X_0)$~\footnote{where $\alpha_1=\alpha, \alpha_2=\beta, \xi_1=\xi_\alpha, \xi_2=\xi_\beta, \xi_3=\xi_\gamma$},
a system of local holomorphic coordinates $(\zeta_1(t),\, \zeta_2(t),\,\zeta_3(t))$ on $X_t=\C^3/\Gamma_t$ 
is given in terms of a system of local holomorphic coordinates $(z_1, \,z_2,\,z_3)$ on $X=X_0$ by the formulae 
\begin{equation}\label{eqn:coordinates_t}\zeta_1(t) = z_1 + \sum\limits_{\lambda=1}^2t_{1\lambda}\,\bar{z}_{\lambda},
\hspace{2ex} \zeta_2(t) = z_2 + \sum\limits_{\lambda=1}^2t_{2\lambda}\,\bar{z}_{\lambda},\hspace{2ex} \zeta_3(t) = z_3 + 
\sum\limits_{\lambda=1}^2(t_{3\lambda} + t_{2\lambda}\,z_1)\,\bar{z}_{\lambda} + A_t(\bar{z})  - D(t)\,\bar{z}_3,\end{equation}
\noindent where \begin{eqnarray*}
                 A_t(\bar{z})&:=& \frac{1}{2}\,[t_{11}\,t_{21}\,\bar{z}_1^2 + 2\,t_{11}\,t_{22}\,\bar{z}_1\,\bar{z}_2 + t_{12}\,t_{22}\,\bar{z}_2^2] \\
\mbox{and} \hspace{6ex} D(t)&:=&t_{11}\,t_{22} - t_{12}\,t_{21}.
                \end{eqnarray*}
Note that the $\zeta_j(t)$'s depend holomorphically on $t$. The projection map given in coordinates by 

$$(\zeta_1(t),\, \zeta_2(t),\,\zeta_3(t))\stackrel{\pi_t}{\mapsto}(\zeta_1(t),\, \zeta_2(t))$$

\noindent displays $X_t$ as fibred over an {\it Abelian surface} $B_t = \mbox{Alb}(X_t)$, the Albanse torus of $X_t$.
These coordinates induce ([Ang11, $\S.4.3$]), for every $t\in\Delta$ close to $0$, the co-frame
\begin{equation}\label{eqn:alpha-beta-gamma_t_def}\alpha_t:=d\zeta_1(t), \hspace{2ex} \beta_t:=d\zeta_2(t), 
\hspace{2ex} \gamma_t:= d\zeta_3(t) - z_1\,d\zeta_2(t) - (t_{21}\,\bar{z}_1 + t_{22}\,\bar{z}_2)\,d\zeta_1(t)
\end{equation}
\noindent of $(1,\,0)$-forms on $X_t$ (i.e. a $\Gamma_t$-invariant co-frame of $(1,\,0)$-forms on $\C^3$) varying in a holomorphic way with $t$. Note that $\alpha_t, \beta_t, \gamma_t$ are linearly independent at every point of $X_t$ if $t$ is sufficiently close to $0$ by mere continuity of their dependence on $t$ 
since $\alpha_0=\alpha$, $\beta_0=\beta$ and $\gamma_0=\gamma$ are linearly independent at every point of $X_0$. 
Also note that $\gamma_t$ need not be $\bar\partial_t$-closed when $t\neq 0$. 
Actually, the complex structure of $X_t$ is complex parallelisable iff $\bar\partial_t\gamma_t = 0$ iff $X_t$ is in Nakamura's class $(i)$ (see [Nak75, p. 94-96]). 

Moreover, for $t$ in one of Nakamura's classes $(ii)$ or $(iii)$ (in particular, for $t\in\Delta_{[\gamma]}$), the {\bf structure equations} for $\gamma_t$ (cf. [Ang11, $\S.4.3$]) read 
\begin{eqnarray}\label{eq:structure}
\nonumber
\bar\partial_t\gamma_t&=& \sigma_{1\bar{1}}(t)\,\alpha_t\wedge\bar\alpha_t + \sigma_{1\bar{2}}(t)\,\alpha_t\wedge\bar\beta_t + \sigma_{2\bar{1}}(t)\,\beta_t\wedge\bar\alpha_t + \sigma_{2\bar{2}}(t)\,\beta_t\wedge\bar\beta_t,\\
\partial_t\gamma_t&=& \sigma_{12}(t)\,\alpha_t\wedge\beta_t,  
\end{eqnarray}
where $\sigma_{12}$ and $\sigma_{i\bar{j}}$ are $C^{\infty}$ functions of $t\in\Delta_{[\gamma]}$ that depend only on $t$ (so $\sigma_{12}(t)$ and $\sigma_{i\bar{j}}(t)$ are complex numbers for every fixed $t\in\Delta_{[\gamma]}$) and satisfy $\sigma_{12}(0) = -1$ and $\sigma_{i\bar{j}}(0) = 0$ for all $i, j$.

Now, for every $t\in\Delta$ close to $0$, the $J_t$-$(1,\,1)$-form 
\begin{equation}\label{eqn:omega_t_def}\omega_t:=i\alpha_t\wedge\overline{\alpha}_t + i\beta_t\wedge\overline{\beta}_t + i\gamma_t\wedge\overline{\gamma}_t\end{equation} 
\noindent is positive definite, hence it defines a Hermitian metric on $X_t$ that varies in a $C^{\infty}$ way with $t$. 
Note that $\omega_0$ is {\it canonically} induced by the complex parallelisable structure of the Iwasawa manifold $X_0$. This feature will play a key role further down.

Let $\Delta''_t=\bar\partial_t\bar\partial_t^{\star} + \bar\partial_t^{\star}\bar\partial_t$ be the $\bar\partial$-Laplacian on $X_t$ defined by $\omega_t$. 
According to [Ang14, p. 80], for every $t$ in one of Nakamura's classes $(ii)$ or $(iii)$ (in particular, for every $t\in\Delta_{[\gamma]}$), 
the following $J_t$-$(2,\,1)$-forms
\begin{eqnarray}\label{eqn:Gamma_j-forms_t}\nonumber \Gamma_1(t):=\alpha_t\wedge\gamma_t\wedge\bar{\alpha}_t - \frac{\sigma_{2\bar{2}}(t)}{\overline{\sigma}_{12}(t)}\,
\alpha_t\wedge\beta_t\wedge\bar{\gamma}_t, & & \Gamma_2(t):=\alpha_t\wedge\gamma_t\wedge\bar{\beta}_t - \frac{\sigma_{2\bar{1}}(t)}{\overline{\sigma}_{12}(t)}\,
\alpha_t\wedge\beta_t\wedge\bar{\gamma}_t, \\
    \Gamma_3(t):=\beta_t\wedge\gamma_t\wedge\bar{\alpha}_t - \frac{\sigma_{1\bar{2}}(t)}{\overline{\sigma}_{12}(t)}\,\alpha_t\wedge\beta_t\wedge\bar{\gamma}_t, 
    & & \Gamma_4(t):=\beta_t\wedge\gamma_t\wedge\bar{\beta}_t - \frac{\sigma_{1\bar{1}}(t)}{\overline{\sigma}_{12}(t)}\,\alpha_t\wedge\beta_t\wedge\bar{\gamma}_t\end{eqnarray}
\noindent are linearly independent $\Delta''_t$-harmonic forms. So, their Dolbeault classes are linearly independent.

\begin{Def}\label{Def:H21_gamma_t} We define
 
\begin{eqnarray}\label{eqn:H21_gamma_t}
H^{2,\,1}_{[\gamma]}(X_t,\,\C):= \bigg\langle [\Gamma_1(t)]_{\bar\partial},\, [\Gamma_2(t)]_{\bar\partial},\, [\Gamma_3(t)]_{\bar\partial},
\, [\Gamma_4(t)]_{\bar\partial}\bigg\rangle
\subset H^{2,\,1}_{\bar\partial}(X_t,\,\C)  \hspace{3ex} \mbox{for every} \hspace{1ex} t\in\Delta_{[\gamma]}.
\end{eqnarray}

\end{Def}

The families $(\Gamma_k(t))_{t\in\Delta_{[\gamma]}}$ are $C^{\infty}$ families of $\Delta''_t$-harmonic $(2,\,1)$-forms (inducing $C^\infty$ families $([\Gamma_k(t)]_{\bar\partial})_{t\in\Delta_{[\gamma]}}$ of $\bar\partial$-cohomology classes) on the fibres of $(X_t)_{t\in\Delta_{[\gamma]}}$ such that $\Gamma_1(0) = \alpha\wedge\gamma\wedge\bar{\alpha}, \hspace{2ex} \Gamma_2(0) = \alpha\wedge\gamma\wedge\bar{\beta}, \hspace{2ex} \Gamma_3(0)
= \beta\wedge\gamma\wedge\bar{\alpha}, \hspace{2ex} \Gamma_4(0) = \beta\wedge\gamma\wedge\bar{\beta}.$ Note that the $\Gamma_k(t)$'s do not depend holomorphically on $t$.

Therefore, we get a $C^{\infty}$ vector bundle ${\cal H}^{2,\,1}_{[\gamma]}\longrightarrow\Delta_{[\gamma]}$ of rank $4$,
$\Delta_{[\gamma]}\ni t\mapsto H^{2,\,1}_{[\gamma]}(X_t,\,\C) = {\cal H}^{2,\,1}_{[\gamma],\,t}$
whose fibre above $t=0$ is $H^{2,\,1}_{[\gamma]}(X,\,\C)$ defined in (\ref{eqn:H21_gamma_def})~\footnote{Alternatively, we could have displayed ${\cal H}^{2,\,1}_{[\gamma],\,t}$ as the bundle of kernels of a smooth family of elliptic differential operators involving a $\mbox{zero}^{th}$-order perturbation by the $\gamma_t$.}.   
\end{proof}

\begin{Rem} {\rm By analogy with $\S.$\ref{subsection:C-Y_isomorphism}, for every $t\in\Delta_{[\gamma]}$ we consider the $J_t$-$(3,\,0)$-form $$\Omega_t:=\alpha_t\wedge\beta_t\wedge\gamma_t.$$ Then $\Omega_t$ depends holomorphically on $t$, hence (by continuity) it is non-vanishing on $X_t$ for all $t$ sufficiently close to zero since $\Omega_0$ is non-vanishing. Moreover, $\Omega_t$ is holomorphic since $\bar\partial_t\Omega_t = \alpha_t\wedge\beta_t\wedge\bar\partial_t\gamma_t =0$, the last identity being a consequence of the special shape of the structure equations~\eqref{eq:structure} (displaying the form $\bar\partial_t\gamma_t$ 
as lying in $\pi_t^\star\mathcal{C}^\infty_{1,\,1}(B_t,\C)$). This shows again that the canonical bundle of $X_t$ is trivial. By analogy with (\ref{eqn:CY_isomorphism}), for every $t\in\Delta_{[\gamma]}$ we define the {\it Calabi-Yau isomorphism} of $X_t$ by

\begin{equation}\label{eqn:CY_isomorphism_t} H^{0,\,1}(X_t,\,T^{1,\,0}X_t)  \underset{\simeq}{\overset{T_{\Omega_t}}{\longrightarrow}}  H^{2,\,1}_{\bar\partial}(X_t,\,\C),  \hspace{3ex} [\theta]  \longmapsto  [\theta\lrcorner \Omega_t],\end{equation}

\noindent and finally, using the subspace $H^{2,\,1}_{[\gamma]}(X_t,\,\C) \subset H^{2,\,1}_{\bar\partial}(X_t,\,\C)$ introduced in Definition \ref{Def:H21_gamma_t}, we put

\begin{equation}\label{eqn:H01_gamma_t}H^{0,\,1}_{[\gamma]}(X_t,\,T^{1,\,0}X_t) := T_{\Omega_t}^{-1}(H^{2,\,1}_{[\gamma]}(X_t,\,\C))\subset H^{0,\,1}(X_t,\,T^{1,\,0}X_t), \hspace{3ex} t\in\Delta_{[\gamma]}.\end{equation}

\noindent In particular, the family $(T_{\Omega_t})_{t\in\Delta_{[\gamma]}}$ of Calabi-Yau isomorphisms is holomorphic and $T^{1,\,0}_t\Delta_{[\gamma]}\simeq H^{0,\,1}_{[\gamma]}(X_t,\,T^{1,\,0}X_t)$ for all $t\in\Delta_{[\gamma]}$.}

\end{Rem}

\vspace{3ex}

The following statement follows from definitions~\eqref{eqn:Gamma_j-forms_t} and the structure equations~\eqref{eq:structure}. For $t=0$, it overlaps with Lemma \ref{Lem:del-star-closedness}.

\begin{Lem}\label{Lem:H21_gamma_t_injection} Let $(X_t)_{t\in\Delta}$ be the Kuranishi family of the Iwasawa manifold $X=X_0$. 
Then, for every $t\in\Delta_{[\gamma]}$, the $J_t$-$(2,\,1)$-forms $\Gamma_1(t), \Gamma_2(t), \Gamma_3(t), \Gamma_4(t)$ of (\ref{eqn:Gamma_j-forms_t}) are $d$-closed and $\bar\partial_t^\star$-closed, where $\bar\partial_t^\star$ is the formal adjoint of $\bar\partial_t$ w.r.t. the metric $\omega_t$ defined in (\ref{eqn:omega_t_def}). When $t=0$, they are also $\partial_0^\star$-closed.
\end{Lem}

\begin{proof}
Thanks to (\ref{eqn:alpha-beta-gamma_t_def}), we have $d\alpha_t = d\beta_t = 0$. Meanwhile, $\partial_t\gamma_t = \sigma_{12}(t)\,\alpha_t\wedge\beta_t$ comes from a form of type $(2,0)$ on $B_t$ by (\ref{eq:structure}). Hence, 
\begin{eqnarray*}\nonumber\partial_t(\alpha_t\wedge\gamma_t\wedge\bar{\alpha}_t) & = & -\alpha_t\wedge\partial_t\gamma_t\wedge\bar{\alpha}_t =0
\end{eqnarray*}
and also 
$$ \partial_t(\beta_t\wedge\gamma_t\wedge\bar{\alpha}_t)=\partial_t(\alpha_t\wedge\gamma_t\wedge\bar{\beta}_t)=\partial_t(\beta_t\wedge\gamma_t\wedge\bar{\beta}_t)=0.$$
\noindent From (\ref{eq:structure}), we get 
$$\partial_t\bar\gamma_t = \overline{\bar\partial_t\gamma_t} 
= \overline{\sigma_{1\bar{1}}(t)}\,\bar\alpha_t\wedge\alpha_t + \overline{\sigma_{1\bar{2}}(t)}\,\bar\alpha_t\wedge\beta_t 
+ \overline{\sigma_{2\bar{1}}(t)}\,\bar\beta_t\wedge\alpha_t + \overline{\sigma_{2\bar{2}}(t)}\,\bar\beta_t\wedge\beta_t,$$ 
\noindent hence
$$\partial_t(\alpha_t\wedge\beta_t\wedge\bar{\gamma}_t) = \alpha_t\wedge\beta_t\wedge\partial_t\bar{\gamma}_t = 0$$
\noindent since all the terms in the resulting sum contain a product $\alpha_t\wedge\alpha_t =0$ or $\beta_t\wedge\beta_t =0$. These identities, together with (\ref{eqn:Gamma_j-forms_t}), prove that $\partial_t\Gamma_j(t) = 0$ for all $t\in\Delta_{[\gamma]}$ and all $j=1,2,3,4$. 

 On the other hand, $\bar\partial_t\Gamma_j(t) = 0$ and $\bar\partial_t^\star\Gamma_j(t) = 0$ since the forms $\Gamma_j(t)$ are $\Delta''_t$-harmonic ([Ang14, p. 80]). Therefore, they are all $d$-closed and $\bar\partial_t^\star$-closed.   

Thanks to (\ref{eqn:Gamma_j-forms_t}), checking whether or not the forms $\Gamma_j(t)$ lie in the kernel of $\partial_t^\star$ involves computing the quantities $\langle\langle\partial_t^\star(\alpha_t\wedge\gamma_t\wedge\bar{\alpha}_t),\,u\rangle\rangle, \,\, \langle\langle\partial_t^\star(\beta_t\wedge\gamma_t\wedge\bar{\alpha}_t),\,u\rangle\rangle, \,\, \langle\langle\partial_t^\star(\alpha_t\wedge\gamma_t\wedge\bar{\beta}_t),\,u\rangle\rangle,$
$ \langle\langle\partial_t^\star(\beta_t\wedge\gamma_t\wedge\bar{\beta}_t),\,u\rangle\rangle, \,\, \langle\langle\partial_t^\star(\alpha_t\wedge\beta_t\wedge\bar{\gamma}_t),\,u\rangle\rangle$ for all forms $u\in C^\infty_{1,\,1}(X_t,\,\C)$ in a system of generators. Now, among the generators $\alpha_t\wedge\bar\alpha_t$, $\alpha_t\wedge\bar\beta_t$, $\alpha_t\wedge\bar\gamma_t$, $\beta_t\wedge\bar\alpha_t$, $\beta_t\wedge\bar\beta_t$, $\beta_t\wedge\bar\gamma_t$, $\gamma_t\wedge\bar\alpha_t$, $\gamma_t\wedge\bar\beta_t$, $\gamma_t\wedge\bar\gamma_t$ of the space $C^\infty_{1,\,1}(X_t,\,\C)$, only those containing $\gamma_t$ or $\bar\gamma_t$ are not $\partial_t$-closed. Moreover, when $u$ is one of these except $\gamma_t\wedge\bar\gamma_t$, $\partial_t u$ is a sum of factors none of which is either $\gamma_t$ or $\bar\gamma_t$, so the above $L^2_{\omega_t}$ inner products vanish. 

 Indeed, for example, if $u = \alpha_t\wedge\bar\gamma_t$, then 

$$\partial_t u = -\alpha_t\wedge\partial_t\bar\gamma_t = -\alpha_t\wedge[\overline{\sigma_{1\bar{2}}(t)}\,\bar\alpha_t\wedge\beta_t + \overline{\sigma_{2\bar{2}}(t)}\,\bar\beta_t\wedge\beta_t],$$

\noindent where the last identity follows from (\ref{eq:structure}). We get

$$\langle\langle\partial_t^\star(\alpha_t\wedge\gamma_t\wedge\bar{\alpha}_t),\,u\rangle\rangle = \langle\langle\alpha_t\wedge\gamma_t\wedge\bar{\alpha}_t,\,\partial_t u\rangle\rangle = 0$$

\noindent since $\alpha_t\wedge\gamma_t\wedge\bar{\alpha}_t$ is $L^2_{\omega_t}$-orthogonal onto $\alpha_t\wedge\bar\alpha_t\wedge\beta_t$ and onto $\alpha_t\wedge\bar\beta_t\wedge\beta_t$. This orthogonality follows from the basis of $(1,\,0)$-forms $\alpha_t, \beta_t, \gamma_t$ being $L^2_{\omega_t}$-orthonormal.

However, when $u=\gamma_t\wedge\bar\gamma_t$, we get \begin{eqnarray}\nonumber\partial_t u = \partial_t\gamma_t\wedge\bar\gamma_t - \gamma_t\wedge\partial_t\bar\gamma_t & = & \sigma_{12}(t)\,\alpha_t\wedge\beta_t\wedge\bar\gamma_t + \overline{\sigma_{1\bar{1}}(t)}\,\gamma_t\wedge\alpha_t\wedge\bar\alpha_t + \overline{\sigma_{1\bar{2}}(t)}\,\gamma_t\wedge\beta_t\wedge\bar\alpha_t \\
  \nonumber & + & \overline{\sigma_{2\bar{1}}(t)}\,\gamma_t\wedge\alpha_t\wedge\bar\beta_t + \overline{\sigma_{2\bar{2}}(t)}\,\gamma_t\wedge\beta_t\wedge\bar\beta_t.\end{eqnarray}

\noindent Hence $\langle\langle\alpha_t\wedge\gamma_t\wedge\bar{\alpha}_t,\,\partial_t u\rangle\rangle = - \sigma_{1\bar{1}}(t)$ and $\langle\langle\frac{\sigma_{2\bar{2}}(t)}{\overline{\sigma_{12}(t)}}\,\alpha_t\wedge\beta_t\wedge\bar{\gamma}_t,\,\partial_t u\rangle\rangle = \sigma_{2\bar{2}}(t)$, so $\partial_t^\star\Gamma_1(t)=0$ if and only if $\sigma_{2\bar{2}}(t) = -\sigma_{1\bar{1}}(t)$. There is no reason for this to happen when $t\neq 0$, but it does happen at $t=0$ since $\sigma_{i\bar{j}}(0)=0$ for all $i,j$.

The forms $\Gamma_2(t), \Gamma_3(t), \Gamma_4(t)$ can be treated in a similar way.   \end{proof}

\begin{Cor}\label{Cor:H21_gamma_t_injection} For every $t\in\Delta_{[\gamma]}$ sufficiently close to $0$, we have a linear {\bf injection}

\begin{equation}\label{eqn::H21_injection_DeRham} H^{2,\,1}_{[\gamma]}(X_t,\,\C)\longrightarrow H^3_{DR}(X,\,\C),  \hspace{6ex} [\Gamma_j(t)]_{\bar\partial}\mapsto\{\Gamma_j(t)\}_{DR} \hspace{3ex} \mbox{for} \hspace{1ex} j=1,\dots , 4,\end{equation}

\noindent where $X$ is the $C^\infty$ manifold underlying the fibres $X_t$. 

\end{Cor}

\noindent {\it Proof.} The $\Delta_0$-harmonicity of the linearly independent forms $\Gamma_1(0), \Gamma_2(0), \Gamma_3(0), \Gamma_4(0)$ implies that the De Rham classes they define are linearly independent in $H^3_{DR}(X,\,\C)$. Thus, the linear map defined in (\ref{eqn::H21_injection_DeRham}) is an injection when $t=0$. Then, by continuity, it remains an injection for $t\in\Delta_{[\gamma]}$ sufficiently close to $0$.  \hfill $\Box$

\vspace{3ex}

As earlier on, we define
\begin{eqnarray}\label{def:h12_gammat}
H^{1,\,2}_{[\gamma]}(X_t,\,\C):=\iota_t(H^{2,\,1}_{[\gamma]}(X_t,\,\C))= \bigg\langle [\star_t\overline\Gamma_1(t)]_{\bar\partial},\, [\star_t\overline\Gamma_2(t)]_{\bar\partial},\, [\star_t\overline\Gamma_3(t)]_{\bar\partial},\, [\star_t\overline\Gamma_4(t)]_{\bar\partial}\bigg\rangle \subset H^{1,\,2}_{\bar\partial}(X_t,\,\C),
\end{eqnarray}

\noindent where $\star_t:=\star_{\omega_t}$ is the Hodge star operator associated with the metric $\omega_t$ defined in (\ref{eqn:omega_t_def}) on $X_t$. 

\subsection{Identification of $H^{2,\,1}_{[\gamma]}(X_t,\,\C)$ with $E_2^{2,\, 1}(X_t)$}\label{subsection:H21_gamma_E2_21}

We shall now give a cohomological interpretation of the spaces $H^{2,\,1}_{[\gamma]}(X_t,\,\C)$ in terms of the groups $E_2^{2,\, 1}(X_t,\,\C)$ featuring at the second step of the Fr\"olicher spectral sequence of each small deformation $X_t$ of the Iwasawa manifold $X=X_0$. At least the first conclusion of the following statement was observed in [COUV16].

\begin{Prop}\label{Prop:H21_gamma_E2_21} Let $(X_t)_{t\in\Delta}$ be the Kuranishi family of the Iwasawa manifold $X=X_0$. Then
\begin{itemize}
\item[$(a)$]\, the Fr\"olicher spectral sequence of $X_t$ degenerates at $E_2$ for every $t\in\Delta$ sufficiently close to $0$;
\item[$(b)$]\, at the second step of the Fr\"olicher spectral sequence, we have $\dim E_2^{2,\,1}(X_t,\,\C)=4$ for $t=0$ 
and for every $X_t$ in any of Nakamura's classes $(ii)$ and $(iii)$ (in particular, for every $t\in\Delta_{[\gamma]}$);
\item[$(c)$]\, there is a canonical isomorphism $E_2^{2,\,1}(X_t,\,\C)\simeq H^{2,\,1}_{[\gamma]}(X_t,\,\C)$ for $t=0$ 
and for every $X_t$ in any of Nakamura's classes $(ii)$ and $(iii)$ (in particular, for every $t\in\Delta_{[\gamma]}$). 
\end{itemize}
\end{Prop}

\begin{proof}

\begin{itemize}
\item[$(a)$]\, This follows from Theorem 5.6 in [COUV16]. Indeed, the $X_t$'s are nilmanifolds of real dimension $6$ endowed 
with invariant complex structures and admitting sG metrics. This last property follows from the Iwasawa manifold $X_0$ being balanced, 
hence sG, and from the sG property being deformation open ([Pop14, Theorem 3.1]). 
\item[ $(b)$ and $(c )$]\, For $X=X_0$, the part of the $E_1$ page of the Fr\"olicher spectral sequence relevant to us is
$$ \cdots \stackrel{\partial}{\longrightarrow} H^{1,\,1}_{\bar\partial}(X,\,\C) \stackrel{\partial}{\longrightarrow} H^{2,\,1}_{\bar\partial}(X,\,\C)
= H^{2,\,1}_{[\gamma]}(X,\,\C) \oplus \bigg\langle [\alpha\wedge\beta\wedge\bar\alpha]_{\bar\partial},
\, [\alpha\wedge\beta\wedge\bar\beta]_{\bar\partial}\bigg\rangle  \stackrel{\partial}{\longrightarrow} H^{3,\,1}_{\bar\partial}(X,\,\C) \stackrel{\partial}{\longrightarrow} 0,$$
\noindent where $\partial$ is defined in cohomology by $\partial([u]_{\bar\partial}) = [\partial u]_{\bar\partial}$ 
and the direct-sum splitting follows from (\ref{eqn:cohomology_explicit}) and (\ref{eqn:H21_gamma_def}). 
Now, we see that much like $\alpha\wedge\beta\wedge\bar\alpha$ and $\alpha\wedge\beta\wedge\bar\beta$, the representatives $\alpha\wedge\gamma\wedge\bar\alpha$, $\alpha\wedge\gamma\wedge\bar\beta$, $\beta\wedge\gamma\wedge\bar\alpha$ 
and $\beta\wedge\gamma\wedge\bar\beta$ of the four $(2,\,1)$-classes generating $H^{2,\,1}_{[\gamma]}(X,\,\C)$ (cf. (\ref{eqn:H21_gamma_def})) are $\partial$-closed. Indeed, for example, $\partial(\alpha\wedge\gamma\wedge\bar\alpha) = -\alpha\wedge\partial\gamma\wedge\bar\alpha = \alpha\wedge(\alpha\wedge\beta)\wedge\bar\alpha = 0$ (cf. (\ref{eqn:alpha,beta,gamma_d})). Hence, the whole of $H^{2,\,1}_{\bar\partial}(X,\,\C)$ is contained in the kernel of $\partial$. Using the explicit description~\eqref{eqn:cohomology_explicit} of $H^{1,\,1}_{\bar\partial}(X,\,\C)$ 
and the structure equation $\partial\gamma = -\alpha\wedge\beta$ of (\ref{eqn:alpha,beta,gamma_d}),
we infer that the image of the map $\partial : H^{1,\,1}_{\bar\partial}(X,\,\C) \longrightarrow H^{2,\,1}_{\bar\partial}(X,\,\C)$
is $\langle[\alpha\wedge\beta\wedge\bar\alpha]_{\bar\partial},\, [\alpha\wedge\beta\wedge\bar\beta]_{\bar\partial}\rangle$. This proves that
\begin{eqnarray*}
E_2^{2,\,1}(X)&=&\ker\bigg(\partial : H^{2,\,1}_{\bar\partial}(X,\,\C) \longrightarrow H^{3,\,1}_{\bar\partial}(X,\,\C)\bigg)\bigg
/\mbox{Im}\,\bigg(\partial : H^{1,\,1}_{\bar\partial}(X,\,\C) \longrightarrow H^{2,\,1}_{\bar\partial}(X,\,\C)\bigg)\\
& =& H^{2,\,1}_{\bar\partial}(X,\,C)\bigg /\bigg\langle [\alpha\wedge\beta\wedge\bar\alpha]_{\bar\partial},
\, [\alpha\wedge\beta\wedge\bar\beta]_{\bar\partial}\bigg\rangle \simeq H^{2,\,1}_{[\gamma]}(X,\,\C)
\end{eqnarray*}
which is $(c)$ for $t=0$. In particular, $\mbox{dim}\, E_2^{2,\,1}(X) =4$ since $H^{2,\,1}_{[\gamma]}(X,\,\C)$ has dimension $4$ by construction.

We now analyse the case when $X_t$ is in Nakamura's class $(iii)$ and show that the Fr\"olicher spectral sequence degenerates even at $E_1$. 
Indeed, the Betti numbers (deformation invariant) and the Hodge numbers of any such $X_t$ computed in [Nak75] read
$$b_1 = 4 = 2 + 2 = h^{1,\,0}(t) + h^{0,\,1}(t), \hspace{2ex} b_2 = 8 = 1 + 5 + 2 = h^{2,\,0}(t) + h^{1,\,1}(t) + h^{0,\,2}(t), $$
$$b_3 = 10 = 1 + 4 + 4 +1 = h^{3,\,0}(t) + h^{2,\,1}(t) + h^{1,\,2}(t) + h^{0,\,3}(t).  $$

\noindent By Poincar\'e and Serre duality, we also get $b_4 = 8 = 2 + 5 +1 = h^{3,\,1}(t) + h^{2,\,2}(t) + h^{1,\,3}(t)$ 
and $b_5 = 4 = 2 + 2 = h^{3,\,2}(t) + h^{2,\,3}(t)$. These identities amount to $E_1(X_t) = E_{\infty}(X_t)$ for every $X_t$ in Nakamura's class $(iii)$. 
In particular, $E_2^{2,\,1}(X_t) = E_1^{2,\,1}(X_t) = H^{2,\,1}_{\bar\partial}(X_t,\,\C)$ whose dimension is $h^{2,\,1}(t)=4$. 
Since the vector subspace $H^{2,\,1}_{[\gamma]}(X_t,\,\C)\subset H^{2,\,1}_{\bar\partial}(X_t,\,\C)$ has the same dimension $4$ (cf. (\ref{eqn:H21_gamma_def})), 
we get $E_2^{2,\,1}(X_t) = H^{2,\,1}_{\bar\partial}(X_t,\,\C) = H^{2,\,1}_{[\gamma]}(X_t,\,\C)$. This proves $(b)$ and $(c)$ for $X_t$ in Nakamura's class $(iii)$.

Suppose now that $X_t$ is in Nakamura's class $(ii)$. Using the description (cf. [Ang11, Appendix A])
\begin{equation*}\label{eqn:H21_t}
\noindent H^{2,\,1}_{\bar\partial}(X_t,\,\C) = H^{2,\,1}_{[\gamma]}(X_t,\,\C)\oplus\bigg\langle[\alpha_t\wedge\beta_t\wedge\bar\alpha_t]_{\bar\partial},
\,[\alpha_t\wedge\beta_t\wedge\bar\beta_t]_{\bar\partial}\bigg\rangle,\end{equation*}
where  $\mbox{dim}\,\langle[\alpha_t\wedge\beta_t\wedge\bar\alpha_t]_{\bar\partial},\,[\alpha_t\wedge\beta_t\wedge\bar\beta_t]_{\bar\partial}\rangle =1$,
and Lemma~\ref{Lem:H21_gamma_t_injection},
we find that the map $\partial_t: H^{2,\,1}_{\bar\partial}(X_t,\,\C)\longrightarrow H^{3,\,1}_{\bar\partial}(X_t,\,\C)$ is identically zero.

Recall that, thanks to [Ang11], we have the splitting
\begin{equation*}
H^{1,\,1}_{\bar\partial}(X_t,\,\C) = \pi_t^\star H^{1,\,1}(B_t,\,\C)\oplus H^{1,\,1}_{vert}(X_t,\,\C)\end{equation*}
in which $H^{1,\,1}_{vert}(X_t,\,\C)$ is of dimension $2$ and is generated by classes represented by forms (containing the vertical form $\gamma$) of the shape 
$E\,\alpha_t\wedge\bar\gamma_t + F\,\beta_t\wedge\bar\gamma_t + G\,\gamma_t\wedge\bar\alpha_t + H\,\gamma_t\wedge\bar\beta_t$, where $E, F, G, H$ are constants. Since $d\alpha_t = d\beta_t = 0$, $\partial_t(\pi_t^\star H^{1,\,1}(B_t,\,\C))=0$. Meanwhile, immediate computations and the use of \eqref{eq:structure} give
\begin{eqnarray*}\partial_t(\alpha_t\wedge\bar\gamma_t) & = & -\alpha_t\wedge\overline{\bar\partial_t\gamma_t} 
= -\alpha_t\wedge(\overline{\sigma_{1\bar{2}}(t)}\,\bar\alpha_t\wedge\beta_t + \overline{\sigma_{2\bar{2}}(t)}\,\bar\beta_t\wedge\beta_t),  \\
\partial_t(\beta_t\wedge\bar\gamma_t) & = & -\beta_t\wedge\overline{\bar\partial_t\gamma_t} 
= -\beta_t\wedge(\overline{\sigma_{1\bar{1}}(t)}\,\bar\alpha_t\wedge\alpha_t + \overline{\sigma_{2\bar{1}}(t)}\,\bar\beta_t\wedge\alpha_t), \\
\partial_t(\gamma_t\wedge\bar\alpha_t) & = & \partial_t\gamma_t\wedge\bar\alpha_t 
= \sigma_{12}(t)\,\alpha_t\wedge\beta_t\wedge\bar\alpha_t, \hspace{2ex}  \partial_t(\gamma_t\wedge\bar\beta_t) 
= \partial_t\gamma_t\wedge\bar\beta_t = \sigma_{12}(t)\,\alpha_t\wedge\beta_t\wedge\bar\beta_t.\end{eqnarray*}
Thus, $\partial_t(H^{1,\,1}_{\bar\partial}(X_t,\,\C)) = \langle[\alpha_t\wedge\beta_t\wedge\bar\alpha_t]_{\bar\partial},\,[\alpha_t\wedge\beta_t\wedge\bar\beta_t]_{\bar\partial}\rangle$.
This settles the case of Nakamura's class~(ii).
\end{itemize}
\end{proof}

 The conclusion of these considerations is summed up in the following

\begin{The}\label{The:VHS_3_Delta} Let $(X_t)_{t\in\Delta}$ be the Kuranishi family of the Iwasawa manifold $X=X_0$.

\vspace{1ex}

$(i)$\, There exists over $\Delta_{[\gamma]}$ a {\bf variation of Hodge structures (VHS)} of weight $3$
\begin{equation}\label{eqn:VHS_3_Delta}{\cal H}^3 = {\cal H}^{3,\, 0} \oplus{\cal H}^{2,\,1}_{[\gamma]} \oplus{\cal H}^{1,\,2}_{[\gamma]} \oplus{\cal H}^{0,\,3},\end{equation}
\noindent where ${\cal H}^3$ is the local system of fibre $H^3_{DR}(X,\,\C)$, ${\cal H}^{3,\, 0}$ is the holomorphic line bundle $\Delta_{[\gamma]}\ni t\mapsto H^{3,\,0}_{\bar\partial}(X_t,\,\C)$, ${\cal H}^{2,\,1}_{[\gamma]}$ is the $C^{\infty}$ vector bundle $\Delta_{[\gamma]}\ni t\mapsto H^{2,\,1}_{[\gamma]}(X_t,\,\C)\simeq E_2^{2,\,1}(X_t,\,\C)$ of rank $4$, while ${\cal H}^{1,\,2}_{[\gamma]} \simeq \overline{{\cal H}^{2,\,1}_{[\gamma]}}$ and ${\cal H}^{0,\,3} = \overline{{\cal H}^{3,\, 0}}$.

\vspace{1ex}

$(ii)$\, The vector subbundles $F^3{\cal H}^3:={\cal H}^{3,\, 0}\subset{\cal H}^3$ and $F^2{\cal H}^3_{[\gamma]}:= {\cal H}^{3,\, 0} \oplus{\cal H}^{2,\,1}_{[\gamma]}\subset{\cal H}^3$ are {\bf holomorphic}. 

The $C^{\infty}$ vector subbundle $F^1{\cal H}^3_{[\gamma]}:= {\cal H}^{3,\, 0} \oplus{\cal H}^{2,\,1}_{[\gamma]} \oplus{\cal H}^{1,\,2}_{[\gamma]}\subset{\cal H}^3$ is {\bf not holomorphic}. This is one of two possible deviations from the behaviour of a standard Hodge filtration. 

\vspace{1ex}

$(iii)$\, As in the standard case, there is a flat connection $\nabla : {\cal H}^3 \longrightarrow {\cal H}^3\otimes\Omega_{\Delta_{[\gamma]}}$ (the {\bf Gauss-Manin connection}) satisfying the {\bf Griffiths transversality condition}
\begin{equation}\label{eqn:transversality_Kuranishi}\nabla F^3{\cal H}^3 \subset F^2{\cal H}^3_{[\gamma]}\otimes\Omega_{\Delta_{[\gamma]}}.\end{equation}

Moreover, in the case of $F^1{\cal H}^3_{[\gamma]}$, the {\bf orthogonality relations} derived from a possible transversality statement remain true:
\begin{eqnarray}\label{ortho}
  Q(\nabla F^1{\cal H}^3_{[\gamma]}, F^0{\cal H}^3_{[\gamma]})=0.
\end{eqnarray}

\end{The}

 It is unclear whether the transversality condition $\nabla F^p{\cal H}^3_{[\gamma]} \subset F^{p-1}{\cal H}^3_{[\gamma]}\otimes\Omega_{\Delta_{[\gamma]}}$ holds for $p =2$ or $p=1$ (the second possible deviation from the behaviour of a standard Hodge filtration).

\begin{proof} $(i)$\, The injection ${\cal H}^{3,\,0}\hookrightarrow{\cal H}^3$ is a consequence of Lemma \ref{Lem:Hn0_injection}, while the injection ${\cal H}^{2,\,1}_{[\gamma]}\hookrightarrow{\cal H}^3$ follows from Corollary \ref{Cor:H21_gamma_t_injection}.
 
Moreover, the property $E_2(X_t) = E_{\infty}(X_t)$ (cf. $(a)$ of Proposition \ref{Prop:H21_gamma_E2_21}) gives an isomorphism 
\begin{equation}\label{eqn:E_2-splitting}H^3_{DR}(X,\,\C)\simeq E_2^{3,\,0}(X_t,\,\C) \oplus E_2^{2,\,1}(X_t,\,\C) \oplus E_2^{1,\,2}(X_t,\,\C) \oplus E_2^{0,\,3}(X_t,\,\C)
\hspace{2ex} \mbox{for every} \hspace{1ex} t\in\Delta.\end{equation}

\noindent We have (cf. $(c)$ of Proposition \ref{Prop:H21_gamma_E2_21}) a canonical isomorphism $E_2^{2,\,1}(X_t,\,\C)\simeq H^{2,\,1}_{[\gamma]}(X_t,\,\C)$, 
while it is easy to prove that $E_2^{3,\,0}(X_t,\,\C) = H^{3,\,0}_{\bar\partial}(X_t,\,\C)$ for every $t\in\Delta_{[\gamma]}$.
Indeed, to see this last point, recall that
\begin{equation}\label{eqn:E_2_30_def}E_2^{3,\,0}(X_t,\,\C) 
= \ker\bigg(\partial_t: H^{3,\,0}_{\bar\partial}(X_t,\,\C)\longrightarrow 0 \bigg)
\bigg\slash \mbox{Im}\,\bigg(\partial_t: H^{2,\,0}_{\bar\partial}(X_t,\,\C)\longrightarrow H^{3,\,0}_{\bar\partial}(X_t,\,\C)\bigg).\end{equation}
\noindent The map $\partial_t$ acting on $H^{3,\,0}_{\bar\partial}(X_t,\,\C)$ arrives in $H^{4,\,0}_{\bar\partial}(X_t,\,\C) = 0$, 
while $H^{2,\,0}_{\bar\partial}(X_t,\,\C)$ is generated by $[\alpha_t\wedge\beta_t]_{\bar\partial}$ when $X_t$ is in Nakamura's class $(iii)$ 
and by $[\alpha_t\wedge\beta_t]_{\bar\partial}$ and either $[\alpha_t\wedge\gamma_t]_{\bar\partial}$ or $[\beta_t\wedge\gamma_t]_{\bar\partial}$ 
when $X_t$ is in Nakamura's class $(ii)$. Now, all the three forms $\alpha_t\wedge\beta_t, \alpha_t\wedge\gamma_t, \beta_t\wedge\gamma_t$ are $\partial_t$-closed 
since $\alpha_t$ and $\beta_t$ are $\partial_t$-closed and $\partial_t\gamma_t$ is a multiple of $\alpha_t\wedge\beta_t$. 
Therefore, $\partial_t(H^{2,\,0}_{\bar\partial}(X_t,\,\C)) = 0$. Thus, we get from (\ref{eqn:E_2_30_def}) that $E_2^{3,\,0}(X_t,\,\C) = H^{3,\,0}_{\bar\partial}(X_t,\,\C)$, as stated.

 It can then be proved from this that $E_2^{2,\,1}(X_t,\,\C)\stackrel{\simeq}{\longrightarrow} \overline{E_2^{1,\,2}(X_t,\,\C)}$ for every $t\in\Delta_{[\gamma]}$.
Now, (\ref{eqn:VHS_3_Delta}) follows by combining these facts with Proposition \ref{Prop:H2-1_gamma_vbundle}.

\vspace{2ex}

$(ii)$\, In the first statement, only the fact that the $C^{\infty}$ vector subbundle $F^2{\cal H}^3_{[\gamma]}\subset{\cal H}^3$ is actually holomorphic still needs a proof. We have to show that the holomorphic structure of $F^2{\cal H}^3_{[\gamma]}$ is the restriction of the holomorphic structure of ${\cal H}^3$. In other words, we have to show that for any $C^\infty$ section $s$ of $F^2{\cal H}^3_{[\gamma]}$, the a priori ${\cal H}^3$-valued $(0,\,1)$-form $D''s$ is actually $F^2{\cal H}^3_{[\gamma]}$-valued, where $D''$ is the canonical $(0,\,1)$-connection of the constant bundle ${\cal H}^3$. We are thus reduced to showing that all the anti-holomorphic first-order derivatives of the $[\Gamma_j(t)]_{\bar\partial}$'s lie in $F^2H^3_{[\gamma]}(X_t,\,\C)$, i.e. that

\begin{equation}\label{eqn:holomorphicity_F2}\frac{\partial[\Gamma_j]_{\bar\partial}}{\partial\bar{t}_{i\lambda}}(t)\in H^{3,\,0}(X_t,\,\C)\oplus H^{2,\,1}_{[\gamma]}(X_t,\,\C) = F^2H^3_{[\gamma]}(X_t,\,\C)  \hspace{3ex} \mbox{for all} \hspace{1ex} t\in\Delta_{[\gamma]} \hspace{1ex} \mbox{all} \hspace{1ex} i,\lambda.\end{equation}

\noindent By way of example, we will show this for the derivatives at $t=0$.

To this end, we will make use of the explicit formula for $\Gamma_1(t)$ and its analogues for $\Gamma_2(t), \Gamma_3(t), \Gamma_4(t)$ obtained in Lemma \ref{Lem:Gamma_1_t_explicit} and also of Lemma \ref{Lem:sigma_22bar_vanishing-anti-hol-deriv} (cf. Appendix). Only the terms on the r.h.s. of that formula that are linear in the $\bar{t}_{i\lambda}$'s give a non-trivial contribution to $(\partial\Gamma_1(t)/\partial\bar{t}_{i\lambda})(0)$. Now, in each of the formulae for $\Gamma_1(t), \Gamma_2(t), \Gamma_3(t), \Gamma_4(t)$, the only such term featuring on the r.h.s. is, respectively, 

$$-\bar{t}_{12}\,\alpha\wedge\beta\wedge\gamma, \hspace{2ex} -\bar{t}_{22}\,\alpha\wedge\beta\wedge\gamma, \hspace{2ex} \bar{t}_{11}\,\alpha\wedge\beta\wedge\gamma, \hspace{2ex} \bar{t}_{21}\,\alpha\wedge\beta\wedge\gamma,$$

\noindent whose derivative in the $\bar{t}_{12}$-direction (respectively the $\bar{t}_{22}$-, $\bar{t}_{11}$-, $\bar{t}_{21}$-direction) is obviously $-\alpha\wedge\beta\wedge\gamma$ (respectively $-\alpha\wedge\beta\wedge\gamma$, $\alpha\wedge\beta\wedge\gamma$, $\alpha\wedge\beta\wedge\gamma$). Thus, for $j\in\{1, 2, 3, 4\}$, the only non-vanishing first-order anti-holomorphic derivatives of the $[\Gamma_j]_{\bar\partial}$'s at $0$ are

$$\frac{\partial[\Gamma_j]_{\bar\partial}}{\partial\bar{t}_{i\lambda}}(0) = \pm\,[\alpha\wedge\beta\wedge\gamma]_{\bar\partial} \in H^{3,\,0}(X_0,\,\C) \subset H^{3,\,0}(X_0,\,\C)\oplus H^{2,\,1}_{[\gamma]}(X_0,\,\C) = F^2H^3_{[\gamma]}(X_0,\,\C).$$

\noindent This proves the contention. Note that this also shows that the $C^\infty$ vector subbundle ${\cal H}^{2,\,1}_{[\gamma]}$ of ${\cal H}^3$ is not a holomorphic subbundle, so the analogy with the standard, K\"ahler, case is preserved. 

 The second statement under $(ii)$ is proved under $(B)$ in the comments that follow the end of this proof. 

\vspace{2ex}

$(iii)$\, The transversality statement is an immediate consequence of the fact that the $(-1,\, +1)$-component of the connection $\nabla_{[\theta]}$ coincides at any point $[\theta]\in T^{1,\,0}_t\Delta_{[\gamma]}\simeq H^{0,\,1}_{[\gamma]}(X_t,\,T^{1,\,0}X_t)$ (for $t\in\Delta_{[\gamma]}$) with the contraction operator $[\theta]\lrcorner\cdot$ (see (\ref{eqn:CY_isomorphism_t}) and (\ref{eqn:H01_gamma_t})). Note that the relation $[\bar\alpha\wedge\bar\beta]_{\bar\partial} =[-\bar\partial\bar\gamma]_{\bar\partial}=0$ implies that the contraction of the forms  of \eqref{eqn:H21_gamma_def} by the elements of \eqref{eqn:H^01_gamma_generators} vanishes, hence we get transversality at $0$:
for all $[\theta]\in T_0^{1,\,0}\Delta_{[\gamma]}\simeq H^{0,\,1}_{[\gamma]}(X,\,T^{1,\,0}X)$,
\begin{eqnarray*}\label{eqn:transversality_0}
\nabla_{[\theta]}H^{3,\,0}(X,\,\C)&\subset& H^{3,\,0}(X,\,\C)\oplus H^{2,\,1}_{[\gamma]}(X,\,\C),
\\
\nabla_{[\theta]}H^{2,\,1}_{[\gamma]}(X,\,\C)&\subset& H^{2,\,1}_{[\gamma]}(X,\,\C)\subset H^{2,\,1}_{[\gamma]}(X,\,\C)\oplus H^{1,\,2}_{[\gamma]}(X,\,\C).
\end{eqnarray*}
\end{proof}

\vspace{3ex}

 We end this discussion with further comments about the Hodge filtration of Theorem \ref{The:VHS_3_Delta}. We notice (cf. Corollary \ref{Cor:holomorphic-bundle-isomorphisms}) that the Hodge filtration $F^2{\cal H}^3_{[\gamma]}\supset F^3{\cal H}^3$ of holomorphic vector bundles over $\Delta_{[\gamma]}$ constructed on the {\it complex-structure side of the mirror} is $C^\infty$ {\it isomorphic} to the Hodge filtration $F^1{\cal H}^2(B)\supset F^2{\cal H}^2(B)$ of holomorphic vector bundles over $\Delta_{[\gamma]}$ determined by the holomorphic family $(B_t)_{t\in\Delta_{[\gamma]}}$ of Albanese tori $B_t = \mbox{Alb}(X_t)$ of the fibres $X_t$. The latter Hodge filtration will be proved to be $C^\infty$ isomorphic to a Hodge filtration that we shall construct on the {\it metric side of the mirror} in section \ref{section:metric-side}, providing thus the link between the two sides.

\vspace{2ex}

$(A)$\, Recall that the fibres $X_t$ are locally trivial holomorphic fibrations $\pi_t:X_t\to B_t$ over complex tori $B_t$ (the Albanese tori of the $X_t$'s) of dimension $2$ varying in a holomorphic family $(B_t)_{t\in\Delta}$. Implicit in the definition of $H^{2,\,1}_{[\gamma]}(X_t,\,\C) \subset H^{2,\,1}_{\bar\partial}(X_t,\,\C)$ (cf. Definition \ref{Def:H21_gamma_t}) are the isomorphisms of complex vector spaces

\begin{equation}\label{eqn:H30_H21_gamma_descriptions}H^{3,\,0}(X_t,\,\C)\simeq [\gamma_t\wedge\pi_t^{\star}H^{2,\,0}(B_t,\,\C)]_{\bar\partial}  \hspace{2ex} \mbox{and}  \hspace{2ex}   H^{2,\,1}_{[\gamma]}(X_t,\,\C)\simeq[\gamma_t\wedge\pi_t^{\star}H^{1,\,1}(B_t,\,\C)]_{\bar\partial}, \hspace{2ex} t\in\Delta_{[\gamma]},\end{equation}

\noindent defined by the descriptions $H^{2,\,0}(B_t,\,\C) = \C\,[\alpha_t\wedge\beta_t]_{\bar\partial}$ and $H^{1,\,1}(B_t,\,\C) = \langle[\alpha_t\wedge\bar\alpha_t]_{\bar\partial},\, [\alpha_t\wedge\bar\beta_t]_{\bar\partial},\,[\beta_t\wedge\bar\alpha_t]_{\bar\partial},\,[\beta_t\wedge\bar\beta_t]_{\bar\partial}\rangle$ of these vector spaces. 

\begin{Cor}\label{Cor:holomorphic-bundle-isomorphisms} The vector space isomorphisms (\ref{eqn:H30_H21_gamma_descriptions}) induce $C^\infty$ isomorphisms of vector bundles over $\Delta_{[\gamma]}$

\begin{equation}\label{eqn:X-B-vb-isom}F^3{\cal H}^3\simeq F^2{\cal H}^2(B)   \hspace{2ex}   \mbox{and}  \hspace{2ex}   F^2{\cal H}^3_{[\gamma]}\simeq F^1{\cal H}^2(B),\end{equation}

\noindent where $F^2{\cal H}^2(B)$ stands for the vector bundle $\Delta_{[\gamma]}\ni t\mapsto H^{2,\,0}(B_t,\,\C)$ and $F^1{\cal H}^2(B)$ stands for the vector bundle $\Delta_{[\gamma]}\ni t\mapsto H^{2,\,0}(B_t,\,\C) \oplus H^{1,\,1}(B_t,\,\C)$.

\end{Cor}

Although the first isomorphism in (\ref{eqn:X-B-vb-isom}) is holomorphic (because $\gamma_t$ and $\pi_t$ depend holomorphically on $t$), it is unclear whether the second one is holomorphic since the pullback under $\pi_t$ and the subsequent exterior multiplication by $\gamma_t$ are followed by the subtraction of a multiple of $\alpha_t\wedge\beta\wedge\bar\gamma_t$ in the definition (\ref{eqn:Gamma_j-forms_t}) of the $\Gamma_j(t)$'s that need not depend holomorphically on $t$.

 Now, $(B_t)_{t\in\Delta}$ is a holomorphic family of compact K\"ahler manifolds, so its Hodge filtration $F^p{\cal H}^2(B)$ consists of {\it holomorphic} subbundles of the constant bundle $\Delta_{[\gamma]}\ni t\mapsto H^2(B_t)$ (denoted henceforth by ${\cal H}^2(B)$). On the other hand, we know from the conclusion $(ii)$ of Theorem \ref{The:VHS_3_Delta} that the subbundles $F^3{\cal H}^3\longrightarrow\Delta_{[\gamma]}$ and $F^2{\cal H}^3_{[\gamma]}\longrightarrow\Delta_{[\gamma]}$ of the constant bundle ${\cal H}^3\longrightarrow\Delta_{[\gamma]}$ are {\it holomorphic}.

\vspace{3ex}

 $(B)$\, We now prove the last statement in part $(ii)$ of Theorem \ref{The:VHS_3_Delta}. We know from (\ref{def:h12_gammat}) that the vector bundle ${\cal H}^{1,\,2}_{[\gamma]}$ is trivialised in a neighbourhood of $0\in\Delta_{[\gamma]}$ by the Dolbeault cohomology classes of the forms $\star_t\,\overline{\Gamma}_j(t)$ with $j=1,\dots , 4$. 

It will be seen in Lemma \ref{Lem:star_generators_21_gamma} that
 $\star(\alpha\wedge\beta\wedge\overline{\gamma}) = i\,\alpha\wedge\beta\wedge\overline{\gamma}$. This also applies at an arbitrary $t$ as do all the identities in Lemma \ref{Lem:star_generators_21_gamma}, so $\star_t(\alpha_t\wedge\beta_t\wedge\overline{\gamma}_t) = i\,\alpha_t\wedge\beta_t\wedge\overline{\gamma}_t$ for all $t\in\Delta$. Therefore, using (\ref{eqn:Gamma_j-forms_t}) for the first line below and (\ref{eqn:forms_t-forms_0}) for the second line, we get for all $t\in\Delta$
\begin{eqnarray}\label{eqn:star_t_Gamma_1_t}\nonumber\star_t\Gamma_1(t) & = & -i\,\beta_t\wedge\gamma_t\wedge\bar\beta_t - i\,\frac{\sigma_{2\bar{2}}(t)}{\bar\sigma_{12}(t)}\,\alpha_t\wedge\beta_t\wedge\bar\gamma_t \\
\nonumber & = & -i\,(\beta + t_{21}\,\bar\alpha + t_{22}\,\bar\beta)\wedge(\gamma + t_{31}\,\bar\alpha + t_{32}\,\bar\beta - D(t)\,\bar\gamma) \wedge(\bar\beta + \bar{t}_{21}\,\alpha + \bar{t}_{22}\,\beta)\\
\nonumber & - & i\,\frac{\sigma_{2\bar{2}}(t)}{\bar\sigma_{12}(t)}\,(\alpha + t_{11}\,\bar\alpha + t_{12}\,\bar\beta)\wedge (\beta + t_{21}\,\bar\alpha + t_{22}\,\bar\beta)\wedge(\bar\gamma + \bar{t}_{31}\,\alpha + \bar{t}_{32}\,\beta -\overline{D(t)}\,\gamma).\end{eqnarray}

\noindent Thus, the terms of $\overline{\star_t\Gamma_1(t)}$ that are linear in the $\bar{t}_{i\lambda}$'s are contained in 

$$i\,\bar{t}_{21}\,\alpha\wedge\bar\gamma\wedge\beta + i\,\bar{t}_{31}\,\bar\beta\wedge\alpha\wedge\beta + i\, \frac{\overline{\sigma_{2\bar{2}}(t)}}{\sigma_{12}(t)}\,\bar\alpha\wedge\bar\beta\wedge\gamma.$$ 

\noindent Deriving at $t=0$, we get

$$\frac{\partial\,\overline{\star_t\Gamma_1(t)}}{\partial\bar{t}_{21}}_{|t=0} = i\,\alpha\wedge\bar\gamma\wedge\beta + i\,\frac{\partial}{\partial\bar{t}_{21}}\bigg(\frac{\overline{\sigma_{2\bar{2}}(t)}}{\sigma_{12}(t)}\bigg)_{|t=0}\,\bar\alpha\wedge\bar\beta\wedge\gamma.$$

\noindent However, although the form $\bar\alpha\wedge\bar\beta\wedge\gamma$ is $\bar\partial_0$-closed, the form $\alpha\wedge\bar\gamma\wedge\beta$ is not (since $\bar\partial_0(\alpha\wedge\bar\gamma\wedge\beta) = -\alpha\wedge\bar\alpha\wedge\beta\wedge\bar\beta \neq 0$), so the form $(\partial\,\overline{\star_t\Gamma_1(t)})/(\partial\bar{t}_{21})_{|t=0}$ defines no Dolbeault cohomology class for $\bar\partial_0$. In particular, the $C^\infty$ section

$$\Delta_{[\gamma]}\ni t\mapsto [\star_t\,\overline\Gamma_1(t)]_{\bar\partial}$$

\noindent of the $C^\infty$ vector subbundle $\Delta_{[\gamma]}\ni t\mapsto H^{1,\,2}_{[\gamma]}(X_t,\,\C)$ of ${\cal H}^3\longrightarrow\Delta_{[\gamma]}$ does not remain a section of this bundle after derivation in the direction $\bar{t}_{21}$.

 We conclude that $\Delta_{[\gamma]}\ni t\mapsto H^{1,\,2}_{[\gamma]}(X_t,\,\C)$ is not a holomorphic subbundle of ${\cal H}^3\longrightarrow\Delta_{[\gamma]}$.

%
%
%
%
%
%
%
%
%
%
%
%
%
%
%
%
%
%
%
%
%

\section{Coordinates on the base $\Delta_{[\gamma]}$ of essential deformations}\label{section:coordinates_Delta}

\subsection{Signature of the intersection form on $F^2_{[\gamma]}H^3(X,\,\C)$}\label{subsection:signature}

Let $(X_t)_{t\in\Delta}$ be the Kuranishi family of the Iwasawa manifold $X=X_0$. 
Recall that the Hodge-Riemann bilinear intersection form $Q$ can always be canonically defined on $H^n_{DR}(X,\,\C)$ 
for any compact complex $n$-dimensional manifold $X$. It is {\it non-degenerate} and depends only on the differential structure of $X$. 
When $\mbox{dim}_{\C}X=3$, $Q$ is {\it alternating} and reads \begin{equation}\label{eqn:Q-def} Q : H^3_{DR}(X,\,\C)\times H^3_{DR}(X,\,\C)\longrightarrow\C,\hspace{3ex} 
(\{u\}, \{v\})\longmapsto -\,\int\limits_X u\wedge v 
~\footnote{In dimension $n$, the coefficient of the integral is $(-1)^{\frac{n(n-1)}{2}}$.}.\end{equation}

\noindent The associated sesquilinear form  \begin{equation}\label{eqn:H-def} 
H : H^3_{DR}(X,\,\C)\times H^3_{DR}(X,\,\C)\longrightarrow\C,\hspace{3ex}  
(\{u\}, \{v\})\longmapsto -i\,\int\limits_X u\wedge\bar v = i\,Q(\{u\},\,\{\bar v\})
~\footnote{\textrm{For arbitrary $n$, the coefficient of the integral is $(-1)^{\frac{n(n+1)}{2}}\,i^n$.}}\end{equation}\noindent is non-degenerate.

Also recall that if a Hermitian metric $\omega$ has been fixed on an arbitrary compact complex $n$-dimensional manifold $Y$, the corresponding Hodge star operator $\star$ maps $\Delta$-harmonic $n$-forms to $\Delta$-harmonic $n$-forms (where $\Delta:=dd^{\star} + d^{\star}d$ is the usual $d$-Laplacian), hence defines in conjunction with the Hodge isomorphism $H^n_{DR}(Y,\,\C)\simeq\ker(\Delta: C^\infty_n(Y,\,\C)\to C^\infty_n(Y,\,\C))$ a linear map $\star: H^n_{DR}(Y,\,\C)\longrightarrow H^n_{DR}(Y,\,\C)$ satisfying $\star^2 = (-1)^n$. 
When $n=3$, the eigenvalues of the operator $\star$ are $-i, i$ and we get a decomposition
\begin{equation}\label{eqn:H3_sign_decomp}H^3_{DR}(X,\,\C) = H^3_{+}(X,\,\C) \oplus H^3_{-}(X,\,\C),\end{equation} 
\noindent where $H^3_{\pm}(X,\,\C)$ are the eigenspaces of $\star$ corresponding to the eigenvalues $+i$, resp. $-i$.

Suppose now that $\mbox{dim}_{\C}X=3$. It was shown in [Pop13b, Lemmas 5.1 and 5.2] that for any Hermitian metric $\omega$ on $X$, 
$H(\cdot,\,\cdot)$ is positive definite on $H^3_{+}(X,\,\C)$, negative definite on $H^3_{-}(X,\,\C)$ and $H^3_{+}(X,\,\C)$ is $H$-orthogonal to $H^3_{-}(X,\,\C)$. Moreover, 
\begin{equation}\label{eqn:H30_sign}H^{3,\,0}(X,\,\C)\subset H^3_{-}(X,\,\C).\end{equation}
\noindent Similar statements hold in arbitrary dimension $n$ after adjusting for the parity of $n$.

Finally, recall that any compact {\it complex parallelisable} manifold $X$ has a natural inner product defined on its space $C^{\infty}_{p,\,q}(X,\,\C)$ 
of smooth differential forms of any bidegree $(p,\,q)$ (cf. [Nak75, $\S.4$] for a construction going back to Kodaira). 
Indeed, if $n=\mbox{dim}_{\C}X$, the hypothesis on $X$ amounts to the existence of $n$ holomorphic $1$-forms $\varphi_1,\dots , \varphi_n\in C^{\infty}_{1,\,0}(X,\,\C)$ that are linearly independent at every point in $X$. If $\xi_1,\dots , \xi_n\in H^0(X,\,T^{1,\,0}X)$ form the dual basis of holomorphic vector fields, 
every form $\varphi\in C^{\infty}_{0,\,1}(X,\,\C)$ can be written uniquely as $\varphi=\sum\limits_{\lambda=1}^nf_{\lambda}\,\overline{\varphi}_{\lambda}$, 
where the $f_{\lambda}$'s are smooth functions globally defined on $X$. One defines the $L^2$ {\bf inner product} on $C^{\infty}_{0,\,1}(X,\,\C)$ by
\begin{equation}\label{eqn:L2_inner-product_def}\langle\langle\varphi,\,\psi\rangle\rangle:= \int\limits_X\bigg(\sum\limits_{\lambda=1}^nf_{\lambda}\,
\overline{g}_{\lambda}\bigg)\, i^{n^2}\, \varphi_1\wedge\dots\wedge\varphi_n\wedge\overline{\varphi}_1\wedge\dots\wedge\overline{\varphi}_n\end{equation}
\noindent for any smooth $(0,\,1)$-forms $\varphi=\sum\limits_{\lambda=1}^nf_{\lambda}\,\overline{\varphi}_{\lambda}$ 
and $\psi=\sum\limits_{\lambda=1}^ng_{\lambda}\,\overline{\varphi}_{\lambda}$. 
Note that $dV:= i^{n^2}\, \varphi_1\wedge\dots\wedge\varphi_n\wedge\overline{\varphi}_1\wedge\dots\wedge\overline{\varphi}_n >0$ is a $C^{\infty}$ positive $(n,\,n)$-form on $X$ that is used as volume form in (\ref{eqn:L2_inner-product_def}). This means that $\langle\langle\varphi,\,\psi\rangle\rangle = \int_X\langle\varphi,\,\psi\rangle\,dV$, where the pointwise inner product $\langle\varphi,\,\psi\rangle$ on $(0,\,1)$-forms is defined by
\begin{equation*}
\langle\overline{\varphi}_{\lambda},\,\overline{\varphi}_{\mu}\rangle = \delta_{\lambda\,\mu}  \hspace{3ex} \mbox{for all}\hspace{1ex} \lambda, \mu.\end{equation*} 

\noindent This induces a pointwise inner product on $C^{\infty}_{p,\,q}(X,\,\C)$ for every $p,q$.

Now suppose that $X$ is the Iwasawa manifold. Thus, $n=3$ and $X$ is complex parallelisable, so with the notation of $\S.$\ref{section:preliminaries} we can choose 
$$\varphi_1=\alpha, \, \varphi_2=\beta, \, \varphi_3=\gamma \hspace{2ex} \mbox{and} \hspace{2ex} \xi_1=\xi_{\alpha}, \, \xi_2=\xi_{\beta}, \, \xi_3=\xi_{\gamma}.$$
\noindent The inner product defined above, induced by the complex parallelisable structure of $X$, coincides with the inner product induced by the canonical metric $\omega_0$ on $X$ defined in (\ref{eqn:Omega_def}).

We can easily check that the $(2,\,1)$-forms $\alpha\wedge\gamma\wedge\overline{\alpha},\, \alpha\wedge\gamma\wedge\overline{\beta},\, 
\beta\wedge\gamma\wedge\overline{\alpha},\, \beta\wedge\gamma\wedge\overline{\beta}$ representing the Dolbeault cohomology classes that generate $H^{2,\,1}_{[\gamma]}(X,\,\C)$ (cf. (\ref{eqn:H21_gamma_def})) are $\Delta$-harmonic. Indeed, they are $\bar\partial$-closed since they are products of $\bar\partial$-closed forms. They are also $\partial$-closed (even if $\gamma$ isn't), as can easily be checked. For example, using (\ref{eqn:alpha,beta,gamma_d}), 
we get $\partial(\alpha\wedge\gamma\wedge\overline{\alpha}) = -\alpha\wedge\partial\gamma\wedge\overline{\alpha} = \alpha\wedge(\alpha\wedge\beta)\wedge\overline{\alpha} =0$ since $\alpha\wedge\alpha = 0$. Thus, all these forms are $d$-closed. They are also both $\partial^{\star}$-closed and $\bar\partial^{\star}$-closed as shown in the next statement (cf. also Lemma \ref{Lem:H21_gamma_t_injection}).

\begin{Lem}\label{Lem:del-star-closedness} The forms $\alpha\wedge\gamma\wedge\overline{\alpha},\, \alpha\wedge\gamma\wedge\overline{\beta},
\, \beta\wedge\gamma\wedge\overline{\alpha},\, \beta\wedge\gamma\wedge\overline{\beta}$ are all $\partial^{\star}$-closed and $\bar\partial^{\star}$-closed.

Note furthermore that the forms $\alpha\wedge\beta\wedge\overline{\alpha},\, \alpha\wedge\beta\wedge\overline{\beta}$ are $\bar\partial^{\star}$-closed but not $\partial^{\star}$-closed.
\end{Lem}

\begin{proof} The identity $\partial^{\star}(\alpha\wedge\gamma\wedge\overline{\alpha})=0$ is equivalent to 
\begin{equation}\label{eqn:alpha-gamma-alpha_bar}\langle\langle\partial^{\star}(\alpha\wedge\gamma\wedge\overline{\alpha}),\,u\rangle\rangle = 0, 
\hspace{2ex} \mbox{i.e. to} \hspace{2ex} \langle\langle\alpha\wedge\gamma\wedge\overline{\alpha},\,\partial u\rangle\rangle = 0\end{equation}
\noindent for every $(1,\,1)$-form $u$ on $X$. Now, the space $C^{\infty}_{1,\,1}(X,\,\C)$ of smooth $(1,\,1)$-forms on $X$ 
is generated by $\alpha\wedge\bar\alpha,\, \alpha\wedge\bar\beta,\, \alpha\wedge\bar\gamma,\, \beta\wedge\bar\alpha,\, \beta\wedge\bar\beta,\, 
\beta\wedge\bar\gamma,\, \gamma\wedge\bar\alpha,\, \gamma\wedge\bar\beta,\, \gamma\wedge\bar\gamma$. Since $d\alpha = d\beta = 0$ and $\bar\partial\gamma=0$, the only generators that are not $d$-closed are those containing $\gamma$. For them, since $\partial\gamma = -\alpha\wedge\beta$, we get:

\begin{enumerate}
 \item if $u=\gamma\wedge\bar\alpha$, then $\partial u = -\alpha\wedge\beta\wedge\bar\alpha$, 
hence $\langle\langle\alpha\wedge\gamma\wedge\overline{\alpha},\,\partial u\rangle\rangle 
= -\langle\langle\alpha\wedge\gamma\wedge\overline{\alpha},\,\alpha\wedge\beta\wedge\bar\alpha\rangle\rangle = 0$;
\item if $u=\gamma\wedge\bar\beta$, then $\partial u = -\alpha\wedge\beta\wedge\bar\beta$, 
hence $\langle\langle\alpha\wedge\gamma\wedge\overline{\alpha},\,\partial u\rangle\rangle 
= -\langle\langle\alpha\wedge\gamma\wedge\overline{\alpha},\,\alpha\wedge\beta\wedge\bar\beta\rangle\rangle = 0$;
\item if $u=\gamma\wedge\bar\gamma$, then $\partial u = -\alpha\wedge\beta\wedge\bar\gamma$, 
hence $\langle\langle\alpha\wedge\gamma\wedge\overline{\alpha},\,\partial u\rangle\rangle 
= -\langle\langle\alpha\wedge\gamma\wedge\overline{\alpha},\,\alpha\wedge\beta\wedge\bar\gamma\rangle\rangle = 0$.
\end{enumerate}
\noindent The three inner products above vanish since the forms $\alpha, \beta, \gamma$ are $\omega_0$-orthonormal. We have thus proved identity (\ref{eqn:alpha-gamma-alpha_bar}). The identities $\partial^{\star}(\alpha\wedge\gamma\wedge\overline{\beta})= \partial^{\star}(\beta\wedge\gamma\wedge\overline{\alpha})= \partial^{\star}(\beta\wedge\gamma\wedge\overline{\beta})=0$ are proved in the same way: all the resulting inner products involve the pairing of a form containing $\gamma$ with a form that does not contain $\gamma$, hence they vanish. 

 This argument does not hold for the forms $\alpha\wedge\beta\wedge\overline{\alpha}$ and $\alpha\wedge\beta\wedge\overline{\beta}$ 
since $\langle\langle\alpha\wedge\beta\wedge\overline{\alpha},\,\partial u\rangle\rangle \neq 0$ when $u=\gamma\wedge\bar\alpha$ 
and $\langle\langle\alpha\wedge\beta\wedge\overline{\beta},\,\partial u\rangle\rangle \neq 0$ when $u=\gamma\wedge\bar\beta$.

 To prove the identities $\bar\partial^{\star}(\alpha\wedge\gamma\wedge\overline{\alpha}) = \bar\partial^{\star}(\alpha\wedge\gamma\wedge\overline{\beta})
= \bar\partial^{\star}(\beta\wedge\gamma\wedge\overline{\alpha})= \bar\partial^{\star}(\beta\wedge\gamma\wedge\overline{\beta})=0$, 
we have to prove that for any form $v\in\{\alpha\wedge\gamma\wedge\overline{\alpha},\, \alpha\wedge\gamma\wedge\overline{\beta},
\, \beta\wedge\gamma\wedge\overline{\alpha},\, \beta\wedge\gamma\wedge\overline{\beta}\}$ and any $w\in C^{\infty}_{2,\,0}(X,\,\C)$, 
we have $\langle\langle v,\,\bar\partial w\rangle\rangle =0$. This is obvious since $C^{\infty}_{2,\,0}(X,\,\C)$ is generated by the $\bar\partial$-closed forms $\alpha\wedge\beta$, $\alpha\wedge\gamma$ and $\beta\wedge\gamma$. The same argument applies to yield the $\bar\partial^{\star}$-closedness of the forms $\alpha\wedge\beta\wedge\overline{\alpha}$ and $\alpha\wedge\beta\wedge\overline{\beta}$.    
\end{proof}


We now compute the Hodge star operator $\star$ induced by the pointwise inner product $\langle\cdot\,,\cdot\rangle$ 
defined by the complex parallelisable structure of $X$ on the $\Delta''$-harmonic representatives of the classes generating $H^{2,\,1}_{[\gamma]}(X,\,\C)$.

\begin{Lem}\label{Lem:star_generators_21_gamma} On the Iwasawa manifold $X$, the following identities hold
\begin{eqnarray*}
\star(\alpha\wedge\gamma\wedge\overline{\alpha}) = -i\,\beta\wedge\gamma\wedge\overline{\beta},
&&
\star(\beta\wedge\gamma\wedge\overline{\beta}) = -i\,\alpha\wedge\gamma\wedge\overline{\alpha},\\
\star(\alpha\wedge\gamma\wedge\overline{\beta}) = i\,\alpha\wedge\gamma\wedge\overline{\beta},
&&
\star(\beta\wedge\gamma\wedge\overline{\alpha}) = i\,\beta\wedge\gamma\wedge\overline{\alpha},\\
\star(\alpha\wedge\beta\wedge\gamma)=-i\alpha\wedge\beta\wedge\gamma, &&
\star(\alpha\wedge\beta\wedge\overline{\gamma}) = i\,\alpha\wedge\beta\wedge\overline{\gamma}.
\end{eqnarray*}

\noindent Consequently, we get
\begin{eqnarray*}
\star(\alpha\wedge\gamma\wedge\overline{\alpha} + \beta\wedge\gamma\wedge\overline{\beta}) & = & -i\,(\alpha\wedge\gamma\wedge\overline{\alpha} + \beta\wedge\gamma\wedge\overline{\beta}), \\
\nonumber \star(\alpha\wedge\gamma\wedge\overline{\alpha} - \beta\wedge\gamma\wedge\overline{\beta}) & = & i\,(\alpha\wedge\gamma\wedge\overline{\alpha} - \beta\wedge\gamma\wedge\overline{\beta}).
\end{eqnarray*}

\end{Lem}

\begin{proof}
From the definition of the Hodge star operator we know that
$$u\wedge\overline{\star(\alpha\wedge\gamma\wedge\overline{\alpha})} = \langle u,\,\alpha\wedge\gamma\wedge\overline{\alpha}\rangle\, dV$$
\noindent for every $(2,\,1)$-form $u$. Both sides of this identity vanish if $u$ is the product of three forms chosen 
from $\alpha,\beta,\gamma,\overline{\alpha}, \overline{\beta}, \overline{\gamma}$, except if $u= \alpha\wedge\gamma\wedge\overline{\alpha}$. In this case, we get
$$(\alpha\wedge\gamma\wedge\overline{\alpha})\wedge\overline{\star(\alpha\wedge\gamma\wedge\overline{\alpha})}
= \langle\alpha\wedge\gamma\wedge\overline{\alpha},\,\alpha\wedge\gamma\wedge\overline{\alpha}\rangle\, dV
= i\,\alpha\wedge\overline{\alpha}\wedge i\,\beta\wedge\overline{\beta}\wedge i\,\gamma\wedge\overline{\gamma}
=i\alpha\wedge\gamma\wedge\overline{\alpha}\wedge\beta\wedge\overline{\beta}\wedge\overline{\gamma}$$
\noindent hence $\overline{\star(\alpha\wedge\gamma\wedge\overline{\alpha})}$ must be the form complementary to $\alpha\wedge\gamma\wedge\overline{\alpha}$, 
i.e. $i\beta\wedge\overline{\beta}\wedge\overline{\gamma}$. We get $\star(\alpha\wedge\gamma\wedge\overline{\alpha}) = -i\,\beta\wedge\gamma\wedge\overline{\beta}$.
The remaining identities are proved in a similar way.  \end{proof}

We can now infer from these computations the signature of the sesquilinear intersection form $H$ on $F^2_{[\gamma]}H^3(X,\,\C)$. 
It is different from the one in the standard case of compact K\"ahler Calabi-Yau $3$-folds with $h^{p,\,0}=0$ for $p=1,2$ 
(where the signature of $H$ on the standard $F^2H^3$ is $(-,\,+,\dots , +)$ due to all classes in $H^3$ being primitive thanks to the assumption $h^{0,\,1}=0$
which implies $h^{3,\,2}=0$ by Serre duality). In our non-K\"ahler case of the Iwasawa manifold, primitivity is meaningless for classes in $H^3$ while $h^{0,\,1} = 2\neq 0$. 
The different signature of $H$ is a key feature of our situation compared to the standard one.

\begin{Cor}\label{Cor:H_signature} If $X$ is the Iwasawa threefold, then $\{\alpha\wedge\gamma\wedge\overline{\alpha} + \beta\wedge\gamma\wedge\overline{\beta}\}_{DR}\in H^3_{-}(X,\,\C),$ while

$$\{\alpha\wedge\gamma\wedge\overline{\alpha} - \beta\wedge\gamma\wedge\overline{\beta}\}_{DR}, \, \{\alpha\wedge\gamma\wedge\overline{\beta}\}_{DR},
\, \{\beta\wedge\gamma\wedge\overline{\alpha}\}_{DR}\in H^3_{+}(X,\,\C).$$
\noindent Hence the signature of $H(\cdot\,,\,\cdot)$ on $H^{2,\,1}_{[\gamma]}(X,\,\C)$ is $(-,\,+,\,+,\,+)$, while the signature of $H(\cdot\,,\,\cdot)$ on $F^2_{[\gamma]}H^3(X,\,\C)$ is $(-,\,-,\,+,\,+,\,+)$.
\end{Cor}

\begin{proof}
We have argued above (cf. Lemma \ref{Lem:del-star-closedness}) that the forms $\alpha\wedge\gamma\wedge\overline{\alpha} + \beta\wedge\gamma\wedge\overline{\beta}$,
$\alpha\wedge\gamma\wedge\overline{\alpha} - \beta\wedge\gamma\wedge\overline{\beta}$, $\alpha\wedge\gamma\wedge\overline{\beta}$ and $\beta\wedge\gamma\wedge\overline{\alpha}$ 
are all $\Delta$-harmonic. Since the splitting (\ref{eqn:H3_sign_decomp}) was defined by the analogous splitting of the space of $\Delta$-harmonic $3$-forms, 
the first statement follows from Lemma \ref{Lem:star_generators_21_gamma}.

The second statement follows from (\ref{eqn:H30_sign}), from the properties of $H(\cdot\,,\,\cdot)$ spelt out above (\ref{eqn:H30_sign}) and from the fact that
$\{[\alpha\wedge\gamma\wedge\overline{\alpha} + \beta\wedge\gamma\wedge\overline{\beta}], \, [\alpha\wedge\gamma\wedge\overline{\alpha} 
- \beta\wedge\gamma\wedge\overline{\beta}], \, [\alpha\wedge\gamma\wedge\overline{\beta}], \, [\beta\wedge\gamma\wedge\overline{\alpha}]\}$ 
is a basis of $H^{2,\,1}_{[\gamma]}(X,\,\C)$.  
\end{proof}

\subsection{Construction of coordinates on $\Delta_{[\gamma]}$}\label{subsection:construction_coord}

\subsubsection{Abstract construction}\label{subsubsection:abstract-construction}

  Let $(X_t)_{t\in\Delta}$ be the Kuranishi family of the Iwasawa manifold $X = X_0$. We know from [Nak75, table on p. 96] that $h^{3,\,0}_{\bar\partial}(X_t)=1$ 
  for all $t\in\Delta$. This implies that $\Delta\ni t\mapsto H^{3,\,0}_{\bar\partial}(X_t,\,\C)$ is a $C^{\infty}$ line bundle by [KS60].
  It is even holomorphic and denoted, as usual, by ${\mathcal H}^{3,\,0}$. Moreover, since $K_{X_0}$ is trivial, the constancy of $h^{3,\,0}(X_t)$ also implies that $K_{X_t}$ is trivial for all $t\in\Delta$. Let us fix, after possibly shrinking $\Delta$ about $0$, a holomorphic section $u=(u_t)_{t\in\Delta}$ of the Hodge bundle ${\mathcal H}^{3,\,0}$ (i.e. a holomorphic family of holomorphic $(3,\,0)$-forms $u_t$ on $X_t$) 
 such that the form $u_t$ is non-vanishing on $X_t$ for every $t\in\Delta$.

  Put, for simplicity, $H^3(X,\,\C):=H^3_{DR}(X,\,\C)$, where by $X$ we mean the $C^{\infty}$ manifold underlying the fibres $X_t$. We know from Lemma \ref{Lem:Hn0_injection} that every space $H^{3,\,0}(X_t,\,\C)$ injects canonically into $H^3(X,\,\C)$, 
  so $u$ can be viewed as a holomorphic function $\Delta\ni t\longmapsto u_t\in H^3(X,\,\C)$.

Meanwhile, $(H^3(X,\,\C),\,Q(\cdot\,,\,\cdot))$ is a symplectic vector space (cf. (\ref{eqn:Q-def})). We shall adapt to our context the presentation in [Voi96, lemme 3.1] to prove that a well-chosen symplectic basis $\{\eta_0,\eta_1,\dots , \eta_4,\,\nu_0, \nu_1,\dots , \nu_4\}$ 
(i.e. such that $Q(\eta_j,\,\eta_k) = Q(\nu_j,\,\nu_k) =0$ and $Q(\eta_j,\,\nu_k) = \delta_{jk}$ for all $j,k$) of $H^3(X,\,\C)$ produces holomorphic coordinates $z_1,\dots , z_4$
near $0$ on $\Delta_{[\gamma]}$. We shall choose all the classes $\eta_j$ and $\nu_k$ to be {\it real}, i.e. $\eta_j=\overline{\eta}_j$ and $\nu_k=\overline{\nu}_k$ for all $j,k$. Consider the following

\vspace{3ex}

\noindent {\bf Setup.} {\it Let $(X_t)_{t\in\Delta}$ be the Kuranishi family of the Iwasawa manifold $X = X_0$ 
on which we have fixed a non-vanishing holomorphic section $u=(u_t)_{t\in\Delta}$ of ${\mathcal H}^{3,\,0}$. 
Let $\eta_0 = \eta_0^{3,\,0} + \eta_0^{2,\,1} + \overline{\eta_0^{2,\,1}} + \overline{\eta_0^{3,\,0}}\in H^3(X,\,\C)$ be a real class 
with $\eta_0^{3,\,0}\in H^{3,\,0}(X_0,\,\C)$, $\eta_0^{2,\,1}\in H^{2,\,1}_{[\gamma]}(X_0,\,\C)$ such that
\begin{equation}\label{eqn:eta_0_assumption} (i)\, Q(u_0,\,\eta_0)\neq 0 \hspace{2ex} \mbox{and} \hspace{2ex} (ii)\, H(\eta_0^{2,\,1},\,\eta_0^{2,\,1})<0.\end{equation}  
\noindent Complete $\eta_0$ to a symplectic basis $\{\eta_0,\eta_1,\dots , \eta_4,\,\nu_0, \nu_1,\dots , \nu_4\}$ of  $(H^3(X,\,\R),\,Q(\cdot\,,\,\cdot))$. 
By continuity, we have $Q(u_t,\,\eta_0)\neq 0$ for all $t$ in a neighbourhood of $0\in\Delta$, so after replacing $u_t$ by $u'_t:=u_t\slash Q(u_t,\,\eta_0)$ we may assume that
\begin{equation*}
  Q(u_t,\,\eta_0) = 1 \hspace{2ex} \mbox{for all}\hspace{1ex} t\in\Delta \hspace{2ex} \mbox{sufficiently close to}\,\,\, 0.\end{equation*}}

We can now state the main result of this subsection.

\begin{Prop}\label{Prop:coordinates} In the setup described above, the functions
\begin{equation}\label{eqn:coordinates_def}z_i(t):= Q(u_t,\,\eta_i) \hspace{2ex} \mbox{for}\hspace{1ex} t\in\Delta_{[\gamma]} \hspace{2ex} 
\mbox{and}\hspace{1ex}  i\in\{1,\dots , 4\}\end{equation}
\noindent define {\bf holomorphic coordinates} on $\Delta_{[\gamma]}$ in a neighbourhood of\, $0$.
\end{Prop}

\begin{proof}
Classes $\eta_0\in H^3(X,\,\C)$ satisfying (\ref{eqn:eta_0_assumption}) do exist. Indeed, for every $3$-class $\eta_0$, $Q(u_0,\,\eta_0) = Q(u_0,\,\eta_0^{0,\,3})$
for bidegree reasons since $u_0$ is of type $(3,\,0)$, so it suffices to choose a class $\eta_0^{0,\,3}\in H^{0,\,3}(X_0,\,\C)$ 
such that $Q(u_0,\,\eta_0^{0,\,3})\neq 0$ for $(i)$ to be satisfied. This is possible since $u_0\neq 0$. Classes $\eta_0^{2,\,1}\in H^{2,\,1}_{[\gamma]}(X_0,\,\C)$ satisfying $(ii)$ exist thanks to the signature of $H$ on $H^{2,\,1}_{[\gamma]}(X_0,\,\C)$ being $(-,\,+,\,+,\, +)$ (cf. Corollary \ref{Cor:H_signature}). We can then put $\eta_0:=\overline{\eta_0^{0,\,3}} + \eta_0^{2,\,1} + \overline{\eta_0^{2,\,1}} + \eta_0^{0,\,3}$
to obtain a {\it real} class $\eta_0$ satisfying (\ref{eqn:eta_0_assumption}). Every class $\eta_0\in H^3(X,\,\C)$ automatically satisfies $Q(\eta_0,\,\eta_0) = 0$ 
since $Q(\eta_0,\,\eta_0) = Q(\eta_0^{3,\,0},\,\eta_0^{0,\,3}) + Q(\eta_0^{2,\,1},\,\eta_0^{1,\,2}) + Q(\eta_0^{1,\,2},\,\eta_0^{2,\,1}) + Q(
\eta_0^{0,\,3},\,\eta_0^{3,\,0})$ while $Q(\eta_0^{3,\,0},\,\eta_0^{0,\,3}) = - Q(\eta_0^{0,\,3},\,\eta_0^{3,\,0})$ and $Q(\eta_0^{2,\,1},\,\eta_0^{1,\,2})
= -Q(\eta_0^{1,\,2},\,\eta_0^{2,\,1})$ since $Q$ is alternating. So $\eta_0$ can be completed to a symplectic basis.

We have to prove that the holomorphic map
\begin{equation*}
\Phi\,:\,\Delta_{[\gamma]} \longrightarrow \C^4, \hspace{2ex} \Phi(t):=(z_1(t),\dots , z_4(t)),\end{equation*}
\noindent is a local diffeomorphism at $0$. Since $\Phi$ is the composition of the maps
$$\Delta_{[\gamma]} \stackrel{u}{\longrightarrow} H^3(X,\,\C) \stackrel{Q_{\eta_1,\dots , \eta_4}}{\longrightarrow} \C^4, \hspace{2ex}
\mbox{where} \hspace{1ex} Q_{\eta_1,\dots , \eta_4}(\cdot):= \bigg(Q(\cdot,\,\eta_1),\dots , Q(\cdot,\,\eta_4)\bigg),$$
\noindent its differential map $d\Phi_0$ at $0$ is the composition of the maps
\begin{equation*}
T^{1,\,0}_0\Delta_{[\gamma]} \underset{\simeq}{\overset{\rho}{\longrightarrow}} H^{0,\,1}_{[\gamma]}(X_0,\,T^{1,\,0}X_0) 
\underset{\simeq}{\overset{\cdot\lrcorner u_0}{\longrightarrow}} H^{2,\,1}_{[\gamma]}(X_0,\,\C) \hookrightarrow H^3(X,\,\C)
\stackrel{Q_{\eta_1,\dots , \eta_4}}{\longrightarrow} \C^4, 
\end{equation*}
\noindent where $\rho$ is the restriction to $\Delta_{[\gamma]}$ of the Kodaira-Spencer map classifying the infinitesimal deformations of $X_0$ 
and the composition of the first three maps is the differential map $du_0\,:\,T^{1,\,0}_0\Delta_{[\gamma]} \longrightarrow H^3(X,\,\C)$ by [Gri68] 
and Proposition \ref{Prop:H21_gamma_Hodge}. Since $T^{1,\,0}_0\Delta_{[\gamma]}$ and $\C^4$ have equal dimensions, it suffices to prove that $d\Phi_0$ is injective.

Reasoning by contradiction, suppose that $d\Phi_0$ is not injective. Then, there exists $\mu\in H^{2,\,1}_{[\gamma]}(X_0,\,\C)$ 
such that $Q(\mu,\,\eta_1) = \dots = Q(\mu,\,\eta_4) = 0$. Since $Q(u_t,\,\eta_0)=1$ for all $t\in\Delta_{[\gamma]}$ close to $0$, $Q(du_0(\xi),\,\eta_0)=0$ 
for every $\xi\in T^{1,\,0}_0\Delta_{[\gamma]}$. Hence $Q(\mu,\,\eta_0) = 0$ because $\mu\in H^{2,\,1}_{[\gamma]}(X_0,\,\C) = (du_0)(T^{1,\,0}_0\Delta_{[\gamma]})$. 
Therefore, $Q(\mu,\,\eta_0) = \dots = Q(\mu,\,\eta_4) = 0$, so $\mu\in\langle\eta_0,\eta_1,\dots , \eta_4\rangle$ 
since the basis $\{\eta_0,\eta_1,\dots , \eta_4,\,\nu_0, \nu_1,\dots , \nu_4\}$ is symplectic. 
This implies that $H(\mu,\,\mu) = 0$ since the subspace $\langle\eta_0,\eta_1,\dots , \eta_4\rangle \subset H^3(X,\,\C)$ is {\it real} and {\it totally $Q$-isotropic}.

On the other hand, $H(\mu,\,\eta_0) = 0$ because $Q(\mu,\,\eta_0) = 0$ and $\eta_0 = \overline{\eta_0}$. 
Thus, $0 = H(\mu,\,\eta_0) = H(\mu,\,\eta_0^{3,\,0}) + H(\mu,\,\eta_0^{2,\,1}) + H(\mu,\,\eta_0^{1,\,2}) + H(\mu,\,\eta_0^{0,\,3}) = H(\mu,\,\eta_0^{2,\,1})$,
where the last identity holds trivially for bidegree reasons since $\mu$ is of type $(2,\,1)$.

Summing up, we have the classes $\mu, \eta_0^{2,\,1}\in H^{2,\,1}_{[\gamma]}(X_0,\,\C)$ with the properties $H(\mu,\,\mu) = 0$ and $H(\mu,\,\eta_0^{2,\,1}) = 0$. 
On the other hand, we know from Corollary \ref{Cor:H_signature} that the restriction of $H$ to $H^{2,\,1}_{[\gamma]}(X_0,\,\C)$ 
is non-degenerate of signature $(-,\,+,\,+,\,+)$, i.e. $H(\cdot\,,\,\cdot): H^{2,\,1}_{[\gamma]}\times H^{2,\,1}_{[\gamma]} \longrightarrow\C$ is a Lorentzian sesquilinear form.

Let $\rho_{\varepsilon}\in H^{2,\,1}_{[\gamma]}(X_0,\,\C)$ such that $H(\rho_{\varepsilon},\,\rho_{\varepsilon})<0$ for every $\varepsilon>0$ 
and $\rho_{\varepsilon}\rightarrow\mu$ as $\varepsilon\rightarrow 0$ (i.e. $(\rho_{\varepsilon})_{\varepsilon>0}$ is an approximation of $\mu$, 
an element in the {\it lightlike cone} of $H$, by elements $\rho_{\varepsilon}$ in the {\it timelike cone} of $H$). 
Let $\eta_{0,\,\varepsilon}^{2,\,1}\rightarrow \eta_0^{2,\,1}$ be an approximation of $\eta_0^{2,\,1}$ such that $H(\rho_{\varepsilon},\,\eta_{0,\,\varepsilon}^{2,\,1}) = 0$ 
for every $\varepsilon$. Since $\rho_{\varepsilon}$ is timelike and the signature of $H$ on $H^{2,\,1}_{[\gamma]}$ is $(-,\,+,\,+,\,+)$, 
the $H$-orthogonal complement $\langle\rho_{\varepsilon}\rangle^{\perp}$ in $H^{2,\,1}_{[\gamma]}$ of the line generated by $\rho_{\varepsilon}$ 
is contained in the subspace $\{\zeta\in H^{2,\,1}_{[\gamma]}\,/\, H(\zeta,\,\zeta)\geq 0\}$. 
(This can be trivially checked by completing $\rho_{\varepsilon}/\sqrt{
|H(\rho_{\varepsilon},\,\rho_{\varepsilon})|}$ to an orthonormal basis of $(H^{2,\,1}_{[\gamma]},\,H)$.) 
Thus, $H(\eta_{0,\,\varepsilon}^{2,\,1},\,\eta_{0,\,\varepsilon}^{2,\,1})\geq 0$ for every $\varepsilon>0$, 
hence for its limit as $\varepsilon\rightarrow 0$ we get $H(\eta_0^{2,\,1},\,\eta_0^{2,\,1})\geq 0$. This contradicts the assumption $(ii)$ of (\ref{eqn:eta_0_assumption}). 
Therefore, $d\Phi_0$ must be injective.  
\end{proof}


\subsubsection{Explicit computations}
The construction of $\S.$\ref{subsubsection:abstract-construction} can be made explicit by choosing 
\begin{eqnarray*}
u_t&=&\alpha_t\wedge\beta_t\wedge\gamma_t,\\
\eta_0^{3,0}&=&\alpha\wedge\beta\wedge\gamma, \ \ \ \ \ \ \ \  \ \ \ \ \ \ \ \ \ \ \ \ \ \ \ \ \ \ \ \ \ \ \ \ \ \ \ \ 
\eta_0^{2,1} = i\,(\alpha\wedge\overline{\alpha} + \beta\wedge\overline{\beta})\wedge\gamma,\\
\eta_0&=&\alpha\wedge\beta\wedge\gamma + i\,(\alpha\wedge\bar\alpha + \beta\wedge\bar\beta)\wedge (\gamma+\bar\gamma) + \bar\alpha\wedge\bar\beta\wedge\bar\gamma, \\
\eta_1&=&\alpha\wedge\beta\wedge\gamma + \bar\alpha\wedge\bar\beta\wedge\bar\gamma, \ \ \ \ \ \ \ \  \ \ \ \ \ \ \ \
\eta_2\ =\ i(\alpha\wedge\bar\alpha-\beta\wedge\bar\beta)\wedge (\gamma+\bar\gamma),\\
\eta_3&=&\alpha\wedge\bar\beta\wedge\gamma+\bar\alpha\wedge\beta\wedge\bar\gamma, \ \ \ \ \ \ \ \  \ \ \ \ \ \ \ \
\eta_4\ =\ \bar\alpha\wedge\beta\wedge\gamma+\alpha\wedge\bar\beta\wedge\bar\gamma.
\end{eqnarray*}

\noindent These forms satisfy condition \eqref{eqn:eta_0_assumption}. Indeed, for example, we have

$$H(\eta_0^{2,1},\,\eta_0^{2,1}) = -\int\limits_X(i\,\alpha\wedge\bar\alpha + i\,\beta\wedge\bar\beta)^2\wedge i\,\gamma\wedge\bar\gamma <0.$$

\noindent The forms $\alpha_t, \beta_t, \gamma_t$ can be computed in terms of $\alpha, \beta, \gamma$ using relations (\ref{eqn:coordinates_t}) and (\ref{eqn:alpha-beta-gamma_t_def}). After recalling the notation $D(t):=t_{11}\,t_{22} - t_{12}\,t_{21}$, we get the following identities for all $t\in\Delta$:
\begin{eqnarray}\label{eqn:forms_t-forms_0}
\nonumber \alpha_t & = & d\zeta_1(t) = dz_1 + (t_{11}\,\bar\alpha + t_{12}\,\bar\beta) = \alpha + t_{11}\,\bar\alpha + t_{12}\,\bar\beta\\
\nonumber \beta_t & = & d\zeta_2(t) = dz_2 + (t_{21}\,\bar\alpha + t_{22}\,\bar\beta) = \beta + t_{21}\,\bar\alpha + t_{22}\,\bar\beta \\
\nonumber \gamma_t & = & d\zeta_3(t) - z_1\,d\zeta_2(t) - (t_{21}\,\bar{z}_1 + t_{22}\,\bar{z}_2)\,d\zeta_1(t) \\
\nonumber & = & [dz_3 + t_{21}\,dz_1\,\bar{z}_1 + (t_{31} + t_{21}\,z_1)\,d\bar{z}_1 + t_{22}\,dz_1\,\bar{z}_2 + (t_{32} + t_{22}\,z_1)
\,d\bar{z}_2 + t_{11}\,t_{21}\,\bar{z}_1\,d\bar{z}_1 + \\
\nonumber & & + t_{11}\,t_{22}\,(\bar{z}_1\,d\bar{z}_2 + \bar{z}_2\,d\bar{z}_1) + t_{12}\,t_{22}\,\bar{z}_2\,d\bar{z}_2-D(t)d\bar{z}_3] \\
 & & -z_1\,(\beta + t_{21}\,\bar\alpha + t_{22}\,\bar{\beta}) - (t_{21}\,\bar{z}_1 + t_{22}\,\bar{z}_2)\,(\alpha + t_{11}\,\bar\alpha + t_{12}\,\bar\beta) \\
\nonumber & \stackrel{(i)}{=} & (\gamma + z_1\,\beta) + t_{21}\,\bar{z}_1\,\alpha + (t_{31} + t_{21}\,z_1)\,\bar\alpha + t_{22}\,\bar{z}_2\,\alpha + (t_{32} + t_{22}\,z_1)
\,\bar\beta + t_{11}\,t_{21}\,\bar{z}_1\bar\alpha \\
\nonumber & & +t_{11}\,t_{22}\,(\bar{z}_1\,\bar\beta + \bar{z}_2\,\bar\alpha) + t_{12}\,t_{22}\,\bar{z}_2\,\bar\beta -D(t)d\bar{z}_3\\
\nonumber&&- z_1\,(\beta + t_{21}\,\bar\alpha + t_{22}\,\bar\beta)-(t_{21}\,\bar{z}_1 + t_{22}\,\bar{z}_2)\,(\alpha + t_{11}\,\bar\alpha + t_{12}\,\bar\beta) \\
 \nonumber   & = & \gamma + t_{31}\,\bar\alpha + t_{32}\,\bar\beta-D(t)d\bar{z}_3 + D(t)\,\bar{z}_1\,\bar\beta\\
  \nonumber   & = & \gamma + t_{31}\,\bar\alpha + t_{32}\,\bar\beta-D(t)\bar\gamma, \hspace{55ex} t\in\Delta,
    \end{eqnarray}

\noindent where $(i)$ followed from $dz_3 = \gamma + z_1\,\beta$. 

 Consequently, we get \begin{eqnarray*}
u_t&=&\left(\alpha + t_{11}\,\bar\alpha + t_{12}\,\bar\beta\right)
\wedge \left(\beta + t_{21}\,\bar\alpha + t_{22}\,\bar\beta\right)\wedge
\left(\gamma + t_{31}\,\bar\alpha + t_{32}\,\bar\beta-D(t)\bar\gamma\right)\\
&=&\alpha\wedge\beta\wedge(\gamma-D(t)\bar\gamma)+D(t)\bar\alpha\wedge\bar\beta\wedge(\gamma-D(t)\bar\gamma)\\
&&+(t_{21}\alpha\wedge\bar\alpha-t_{12}\beta\wedge\bar\beta)\wedge(\gamma-D(t)\bar\gamma)
+(t_{22}\alpha\wedge\bar\beta+t_{11}\bar\alpha\wedge\beta)\wedge(\gamma-D(t)\bar\gamma)
\\
&&+\alpha\wedge\beta\wedge(t_{31}\bar\alpha+t_{32}\bar\beta)
+(t_{21}t_{32}-t_{31}t_{22})\alpha\wedge\bar\alpha\wedge\bar\beta
+(t_{11}t_{32}-t_{12}t_{31})\bar\alpha\wedge\beta\wedge\bar\beta.
 \end{eqnarray*}
Note that the terms are displayed according to their degree and type on the base $B$ of $\pi:X\to B$. The part coming from the base (i.e. the terms on the last line, those containing neither $\gamma$ nor $\bar\gamma$) vanishes on $\Delta_{[\gamma]}$ since $t_{31}=t_{32}=0$ there.

We can now compute the resulting coordinates on $\Delta_{[\gamma]}$. We get for $t\in\Delta_{[\gamma]}$:
\begin{eqnarray*}
 Q(u_t,\eta_0)&=&-\int_X u_t\wedge \eta_0
 =-\int_X u_t\wedge \bigg(\alpha\wedge\beta\wedge\gamma+\bar\alpha\wedge\bar\beta\wedge\bar\gamma+ i(\alpha\wedge\bar\alpha+\beta\wedge\bar\beta)\wedge (\gamma+\bar\gamma)\bigg)\\
 &=&\bigg(i(1+D(t)^2)+(t_{21}-t_{12})(1+D(t))\bigg)\,\int_X i\alpha\wedge\bar\alpha\wedge i\beta\wedge\bar\beta\wedge i\gamma\wedge\bar\gamma = 1, 
\end{eqnarray*}
where the last identity is the normalisation adopted in Proposition \ref{Prop:coordinates}. We also get for $t\in\Delta_{[\gamma]}$:

\begin{eqnarray*}
 Q(u_t,\eta_1)&=&\frac{i(1+D(t)^2)}{i(1+D(t)^2)+(t_{21}-t_{12})(1+D(t))}, \ 
 Q(u_t,\eta_2)=-\frac{(t_{12} + t_{21})(1+D(t))}{i(1+D(t)^2)+(t_{21}-t_{12})(1+D(t))},\\
 Q(u_t,\eta_3)&=&-i\,\frac{t_{11}\,D(t) + t_{22}}{i(1+D(t)^2)+(t_{21}-t_{12})(1+D(t))},\  
 Q(u_t,\eta_4)=-i\,\frac{t_{22}\,D(t) + t_{11}}{i(1+D(t)^2)+(t_{21}-t_{12})(1+D(t))}.
\end{eqnarray*}

\subsection{The $B$-Yukawa coupling}
\begin{Def}

Suppose we have fixed a non-vanishing holomorphic $(3,\,0)$-form $u$ on the Iwasawa manifold $X$. 
It identifies with the class $[u]\in H^{3,\,0}_{\bar\partial}(X,\,\C)\simeq H^0(X,\,K_X)\simeq\C$. The {\bf Yukawa coupling} associated with $u$ is standardly defined as
\begin{eqnarray}\nonumber
Y_2^{(u)}\,:\,H^{0,\,1}(X_0,\,T^{1,\,0}X_0) \times H^{0,\,1}(X_0,\,T^{1,\,0}X_0) \times H^{0,\,1}(X_0,\,T^{1,\,0}X_0) &\longrightarrow& \C\\
\nonumber
([\theta_1],\, [\theta_2],\,[\theta_3])&\mapsto &\bigg\langle u^2,\, [\theta_1]\cdot[\theta_2]\cdot[\theta_3]\bigg\rangle
\end{eqnarray}
\noindent where $u^2$ is viewed as a section $u^2\in H^0(X,\,K_X^{\otimes 2})\simeq H^{3,\,0}(X,\,K_X)$, 
the cup product $[\theta_1]\cdot[\theta_2]\cdot[\theta_3]\in H^{0,\,3}(X,\,\Lambda^3T^{1,\,0}X) = H^{0,\,3}(X,\,K_X^{-1})$ 
and $\langle\cdot,\,\cdot\rangle \,:\,H^{3,\,0}(X,\,K_X)\times H^{0,\,3}(X,\,K_X^{-1}) \longrightarrow \C$ is the Serre duality.
\end{Def}

We can now use the symplectic basis and the coordinates constructed in Proposition \ref{Prop:coordinates} to show, 
by the same method as in the standard K\"ahler case ([BG83]), that the Yukawa couplings $Y_2$ on $T^{1,\,0}_0\Delta_{[\gamma]} \simeq H^{0,\,1}_{[\gamma]}(X_0,\,T^{1,\,0}X_0)$ 
are defined by a potential.

\begin{Prop}\label{Prop:Yukawa2-potential} Let $(X_t)_{t\in\Delta}$ be the Kuranishi family of the Iwasawa manifold $X = X_0$ 
on which we have fixed a non-vanishing holomorphic section $u=(u_t)_{t\in\Delta}$ of ${\mathcal H}^{3,\,0}$ normalised by the choice of a symplectic basis 
as in Proposition \ref{Prop:coordinates}. Let $z_1,\dots , z_4$ be the induced holomorphic coordinates near $0$ on $\Delta_{[\gamma]}$ 
constructed in Proposition \ref{Prop:coordinates}.
Then, there exists a $C^{\infty}$ function $F = F(z_1,\dots , z_4)\,:\,\Delta_{[\gamma]}\longrightarrow\C$ such that
\begin{equation}\label{eqn:Yukawa2-potential}Y_2^{(u)}\bigg(\frac{\partial}{\partial z_i},\,\frac{\partial}{\partial z_j},\,\frac{\partial}{\partial z_k}\bigg)
= \frac{\partial^3F}{\partial z_i\partial z_j\partial z_k}\end{equation}
\noindent for all $\frac{\partial}{\partial z_i},\,\frac{\partial}{\partial z_j},\,\frac{\partial}{\partial z_k}
\in T^{1,\,0}_0\Delta_{[\gamma]}\simeq H^{0,\,1}_{[\gamma]}(X_0,\,T^{1,\,0}X_0)$.
\end{Prop}

\begin{proof}
The arguments are standard (see e.g. [Voi96, $\S.3.1.2$]), but we spell them out for the reader's convenience and to show that they adapt to our non-standard situation.

\vspace{1ex}

{\it Step $1$.} For all $i\in\{1,\dots, 4\}$, put $\Psi_i:\Delta_{[\gamma]}\longrightarrow\C$, $\Psi_i(z_1(t),\dots , z_4(t)):=Q(u_t,\,\nu_i)$. Prove that
\begin{equation}\label{eqn:Psi_i_derivatives}\frac{\partial\Psi_i}{\partial z_j} = \frac{\partial\Psi_j}{\partial z_i} \hspace{2ex} 
\mbox{for all}\hspace{1ex} i,j\in\{1,\dots , 4\}.\end{equation}

\noindent This is proved by writing $u_t = a_0\nu_0 + \sum\limits_{j=1}^4a_j\nu_j + \sum\limits_{j=1}^4b_j\eta_j + b_0\eta_0$ 
and computing the coefficients $a_j, b_j$ by using the relation $Q(u_t,\,\eta_0) = 1$ 
and the symplectic property of the basis $\{\eta_0,\eta_1,\dots , \eta_4,\,\nu_0,\nu_1,\dots , \nu_4\}$. 
We get $u_t = \nu_0 + \sum\limits_{j=1}^4z_j\nu_j - \sum\limits_{j=1}^4\Psi_j\eta_j - Q(u_t,\,\nu_0)\,\eta_0$. 
Taking the derivative $\partial/\partial z_i$, we get
\begin{equation*}\label{eqn:del_u_del_z_i}\frac{\partial u}{\partial z_i}
= \nu_i - \sum\limits_{j=1}^4\frac{\partial\Psi_j}{\partial z_i}\,\eta_j - \frac{\partial Q(u,\,\nu_0)}{\partial z_i}\,\eta_0.\end{equation*}
From this and the symplectic property of the basis $\eta_j,\nu_k$, we infer
$$ Q\bigg(\frac{\partial u}{\partial z_i},\, \frac{\partial u}{\partial z_j}\bigg) = -\frac{\partial\Psi_i}{\partial z_j}  + \frac{\partial\Psi_j}{\partial z_i}.$$ 
\noindent On the other hand, $\frac{\partial u}{\partial z_i} = \rho(\frac{\partial}{\partial z_i})\lrcorner u \in F^2_{[\gamma]}H^3(X,\,\C)$ 
for all $\frac{\partial}{\partial z_i}\in T^{1,\,0}_0\Delta_{[\gamma]} \stackrel{\rho}{\simeq} H^{0,\,1}_{[\gamma]}(X_0,\,T^{1,\,0}X_0)$ 
by Griffiths's transversality [Gri68] (see (\ref{eqn:transversality_Kuranishi}), our version of it), so for bidegree reasons we get:
$0 = Q\bigg(\frac{\partial u}{\partial z_i},\, \frac{\partial u}{\partial z_j}\bigg).$
 This proves (\ref{eqn:Psi_i_derivatives}).

It follows from (\ref{eqn:Psi_i_derivatives}) that there exists a $C^{\infty}$ function $F = F(z_1,\dots , z_4)\,:\,\Delta_{[\gamma]}\longrightarrow\C$ such that
\begin{equation*}
\frac{\partial F}{\partial z_i} = \Psi_i \hspace{2ex} \mbox{for all} \hspace{1ex} i\in\{1,\dots , 4\}.\end{equation*}

\vspace{1ex}

{\it Step $2$.} Prove (\ref{eqn:Yukawa2-potential}) for this choice of $F$.

By the orthogonality relations ~\eqref{ortho}, we have
$$\bigg\langle u,\,\frac{\partial^2 u}{\partial z_i\partial z_j}\bigg\rangle = 0$$
\noindent since $u_t\in H^{3,\,0}(X_t,\,\C)$ for all $t$. Applying $\partial/\partial z_k$, we get
\begin{equation*}\label{eqn:3rd_derivatives_u1}
\bigg\langle \frac{\partial u}{\partial z_k},\,\frac{\partial^2 u}{\partial z_i\partial z_j}\bigg\rangle 
+ \bigg\langle  u,\,\frac{\partial^3 u}{\partial z_i\partial z_j\partial z_k}\bigg\rangle = 0, \hspace{2ex} 
\mbox{hence} \hspace{1ex} Y_2^{(u)}\bigg(\frac{\partial}{\partial z_i},\,\frac{\partial}{\partial z_j},\,\frac{\partial}{\partial z_k}\bigg) 
= - \bigg\langle \frac{\partial u}{\partial z_k},\,\frac{\partial^2 u}{\partial z_i\partial z_j}\bigg\rangle.\end{equation*}
\noindent On the other hand, from the identities $u_t = \nu_0 + \sum\limits_{l=1}^4z_l\,\nu_l - \sum\limits_{l=1}^4\Psi_l\,\eta_l - Q(u_t,\,\nu_0)\,\eta_0$ 
 seen at Step $1$, we compute
\begin{eqnarray*}\label{eqn:3rd_derivatives_u2}\nonumber\bigg\langle \frac{\partial u}{\partial z_k},\,\frac{\partial^2 u}{\partial z_i\partial z_j}\bigg\rangle 
& = & - \bigg\langle \nu_k - \sum\limits_{l=1}^4\frac{\partial\Psi_l}{\partial z_k}\,\eta_l - \frac{\partial Q(u,\,\nu_0)}{\partial z_k}\,\eta_0,\, 
\sum\limits_{l=1}^4\frac{\partial^2\Psi_l}{\partial z_i\partial z_j}\,\eta_l + \frac{\partial^2 Q(u,\,\nu_0)}{\partial z_i\partial z_j}\,\eta_0\bigg\rangle \\
  & = & -\frac{\partial^2\Psi_k}{\partial z_i\partial z_j} = -  \frac{\partial^3F}{\partial z_i\partial z_j\partial z_k}.\end{eqnarray*}
  \noindent The last two main identities combined prove (\ref{eqn:Yukawa2-potential}).  \end{proof}

\section{The metric side of the mirror}\label{section:metric-side}

As usual, we let $(X_t)_{t\in\Delta_{[\gamma]}}$ stand for the Kuranishi family of the Iwasawa manifold $X=X_0$.

\subsection{Constructing Gauduchon metrics}
\subsubsection{A smooth family $(\omega_t)_{t\in\Delta_{[\gamma]}}$ of Gauduchon metrics on $(X_t)_{t\in\Delta_{[\gamma]}}$}\label{subsubsection:G_Xt}

Recall that (\ref{eqn:omega_t_def}) provides us with a $C^{\infty}$ family of canonical Hermitian metrics $(\omega_t)_{t\in\Delta}$ on the fibres $(X_t)_{t\in\Delta}$ 
after possibly shrinking $\Delta$ about $0$. 
Simple calculations enable us to prove the following.
\begin{Lem}\label{Lem:omega_t_Gauduchon} Let $(X_t)_{t\in\Delta}$ be the Kuranishi family of the Iwasawa manifold $X=X_0$. 
Then, for every $t\in\Delta_{[\gamma]}$, the metric $\omega_t= i\alpha_t\wedge\bar\alpha_t + i\beta_t\wedge\bar\beta_t + i\gamma_t\wedge\bar\gamma_t$ 
is a Gauduchon metric on $X_t$, hence $[\omega_t^2]_A$ defines an element in the Gauduchon cone ${\cal G}_{X_t}$ of $X_t$.
\end{Lem}

\begin{proof}
Since $\dim_{\C}X_t = 3$, we have to show that $\partial_t\bar\partial_t\omega_t^2 = 0$ for $t\in\Delta_{[\gamma]}$. For all $t\in\Delta$, 
$$\omega_t^2 = -2\,\alpha_t\wedge\bar\alpha_t\wedge\beta_t\wedge\bar\beta_t - 2\,\alpha_t\wedge\bar\alpha_t\wedge\gamma_t\wedge\bar\gamma_t - 2\,
\beta_t\wedge\bar\beta_t\wedge\gamma_t\wedge\bar\gamma_t.$$
It now follows from lemma~\ref{lem:pluri-closed} that
$\partial_t\bar\partial_t\omega_t^2 = 0$ for all $t\in\Delta_{[\gamma]}$.  \end{proof}

\subsubsection{A smooth family $(\omega_t^{1,\,1})_{t\in\Delta_{[\gamma]}}$ of Gauduchon metrics on $X_0$}\label{subsubsection:G_X0}

 We will implicitly construct a smooth family of Aeppli-Gauduchon classes in ${\cal G}_{X_0}$ naturally induced by the structure of the family $(X_t)_{t\in\Delta}$. Each Hermitian metric $\omega_t=i\alpha_t\wedge\bar\alpha_t + i\beta_t\wedge\bar\beta_t + i\gamma_t\wedge\bar\gamma_t$ (proved in Lemma \ref{Lem:omega_t_Gauduchon} to be even a Gauduchon metric on $X_t$ for $t\in\Delta_{[\gamma]}$) can be viewed as a real $2$-form on the $C^{\infty}$ manifold $X$ underlying the fibres $X_t$. As such, $\omega_t$ has a component of bidegree $(1,\,1)$ w.r.t. the complex structure $J_0$ of $X_0$. We denote it by $\omega_t^{1,\,1}\in C^{\infty}_{1,\,1}(X_0,\,\R)$.

\begin{Prop}\label{Prop:omega_t_11} Let $(X_t)_{t\in\Delta}$ be the Kuranishi family of the Iwasawa manifold $X_0$. 
Then, the $J_0$-$(1,\,1)$-form $\omega_t^{1,\,1}$ is a Gauduchon metric on $X_0$ for every $t\in\Delta$ sufficiently close to $0$. 
Moreover, $\omega_0^{1,\,1} = \omega_0$ and $\omega_t^{1,\,1}$ varies in a $C^{\infty}$ way with $t$. 
\end{Prop}

\begin{proof}

Recall that $\omega_t = i\,\alpha_t\wedge\bar\alpha_t + i\,\beta_t\wedge\bar\beta_t + i\,\gamma_t\wedge\bar\gamma_t$. Hence, using the identities~\eqref{eqn:forms_t-forms_0}, 
we get
\begin{eqnarray}\nonumber\omega_t & = & i\,(\alpha + t_{11}\,\bar\alpha + t_{12}\,\bar\beta)\wedge(\bar\alpha + \bar{t}_{11}\,\alpha + \bar{t}_{12}\,\beta) + i\,( \beta + t_{21}
\,\bar\alpha + t_{22}\,\bar\beta)\wedge(\bar\beta + \bar{t}_{21}\,\alpha + \bar{t}_{22}\,\beta) \\
\nonumber & + & i\,[\gamma + t_{31}\,\bar\alpha + t_{32}\,\bar\beta -D(t)\,\bar\gamma]\wedge[\bar\gamma + \bar{t}_{31}\,\alpha + \bar{t}_{32}\,\beta 
- \overline{D(t)}\,\gamma].\end{eqnarray} 
\noindent Hence, the $J_0$-type $(1,\,1)$-component of $\omega_t$ is
\begin{equation}\label{eqn:omega_t_11_J0}\omega_t^{1,\,1} = (1+ c_1(t))\,i\alpha\wedge\bar\alpha + (1+c_2(t))\,i\beta\wedge\bar\beta +(1+c_3(t))\,i\gamma\wedge\bar\gamma
+ d(t)\,i\alpha\wedge\bar\beta 
+ \overline{d(t)}\,i\beta\wedge\bar\alpha,\end{equation}
\noindent where 
\begin{eqnarray}\label{eqn:c_j-d_t}
\nonumber c_1(t) & = & -(|t_{11}|^2 + |t_{21}|^2 + |t_{31}|^2) \\ 
\nonumber c_2(t) & = & -(|t_{12}|^2 + |t_{22}|^2 + |t_{32}|^2), \\
\nonumber c_3(t) & =& - |D(t)|^2 = -|t_{11}\,t_{22} - t_{12}\,t_{21}|^2, \\
        d(t) & = & -(t_{12}\,\bar{t}_{11} + t_{22}\,\bar{t}_{21} + t_{32}\,\bar{t}_{31}).
\end{eqnarray}

We see that $\omega_t^{1,\,1}$ varies in a $C^{\infty}$ way with $t$ and that $\omega_0^{1,\,1} = \omega_0$. In particular, since $\omega>0$, by continuity we get $\omega_t^{1,\,1} > 0$ for all $t$ sufficiently close to $0$, so $(\omega_t^{1,\,1})_{t\in\Delta}$ 
is a $C^{\infty}$ family of Hermitian metrics on $X_0$ after possibly shrinking $\Delta$ about $0$.


It remains to show that $\partial\bar\partial(\omega_t^{1,\,1})^2 = 0$, where $\partial=\partial_0$ and $\bar\partial = \bar\partial_0$, i.e. that each $\omega_t^{1,\,1}$ is a Gauduchon metric on $X_0$. Taking squares in (\ref{eqn:omega_t_11_J0}), we get
\begin{eqnarray}\nonumber (\omega_t^{1,\,1})^2 & = & \omega_0^2 + 2c_1(t)\, \omega_0\wedge i\,\alpha\wedge\bar\alpha + 2c_2(t)\, \omega_0\wedge i\,\beta\wedge\bar\beta + 2c_3(t)\, \omega_0\wedge i\,\gamma\wedge\bar\gamma \\
\nonumber & + & 2d(t)\, \omega_0\wedge i\,\alpha\wedge\bar\beta + 2\overline{d(t)}\, \omega_0\wedge i\,\beta\wedge\bar\alpha \\
\nonumber & + & 2c_1(t)\,c_2(t)\,i\,\alpha\wedge\bar\alpha\wedge i\,\beta\wedge\bar\beta + 2c_2(t)\,c_3(t)\,i\,\beta\wedge\bar\beta\wedge i\,\gamma\wedge\bar\gamma + 2c_1(t)\,c_3(t)\,i\,\alpha\wedge\bar\alpha\wedge i\,\gamma\wedge\bar\gamma
\\
\nonumber & + & 2c_1(t)\,d(t)\,i\,\alpha\wedge\bar\alpha\wedge i\,\alpha\wedge\bar\beta + 2c_1(t)\,\overline{d(t)}\,i\,\alpha\wedge\bar\alpha\wedge i\,\beta\wedge\bar\alpha \\
 \nonumber   & + & 2c_2(t)\,d(t)\,i\,\beta\wedge\bar\beta\wedge i\,\alpha\wedge\bar\beta + 2c_2(t)\,\overline{d(t)}\,i\,\beta\wedge\bar\beta\wedge i\,\beta\wedge\bar\alpha - 2|d(t)|^2\,i\,\alpha\wedge\bar\alpha\wedge i\,\beta\wedge\bar\beta\\
\nonumber & + & 2c_3(t)\,d(t)\,i\,\alpha\wedge\bar\beta\wedge i\,\gamma\wedge\bar\gamma + 2c_3(t)\,\overline{d(t)}\,i\,\beta\wedge\bar\alpha\wedge i\,\gamma\wedge\bar\gamma.\end{eqnarray}

\noindent After removing the vanishing terms (that are products containing two equal factors chosen from $\alpha, \beta, \bar\alpha, \bar\beta$) and regrouping the remaining ones, we get
\begin{eqnarray}\label{eqn:omega_t_11_J0_square}\nonumber (\omega_t^{1,\,1})^2 & = & \omega_0^2 + 2\,[c_1(t) + c_2(t) + c_1(t)\,c_2(t) - |d(t)|^2]\,i\,\alpha\wedge\bar\alpha\wedge i\,\beta\wedge\bar\beta \\
\nonumber & + & 2\,[c_1(t) + c_3(t) + c_1(t)\,c_3(t)]\,i\,\alpha\wedge\bar\alpha\wedge i\,\gamma\wedge\bar\gamma + 2\,[c_2(t) +c_3(t) + c_2(t)\,c_3(t)]\,i\,\beta\wedge\bar\beta\wedge i\,\gamma\wedge\bar\gamma \\
     & + & 2\,d(t)\,[1 + c_3(t)]\,i\,\alpha\wedge\bar\beta\wedge i\,\gamma\wedge\bar\gamma + 2\,\overline{d(t)}\,[1 + c_3(t)]\,i\,\beta\wedge\bar\alpha\wedge i\,\gamma\wedge\bar\gamma.\end{eqnarray}

We can now show, using the identities $d\alpha = d\beta = 0$, $\bar\partial\gamma = 0$ and $\partial\gamma = -\alpha\wedge\beta$ (cf. (\ref{eqn:alpha,beta,gamma_d})), that every term on the r.h.s. of (\ref{eqn:omega_t_11_J0_square}) is at least $\partial\bar\partial$-closed. We have already seen that $\partial\bar\partial\omega_0^2=0$. We get furthermore
\vspace{1ex}

$\bar\partial(i\,\alpha\wedge\bar\alpha\wedge i\,\beta\wedge\bar\beta) = 0$ since the forms $\alpha,\,\bar\alpha,\,\beta,\,\bar\beta$ are all $\bar\partial$-closed,

\vspace{1ex}

$\bar\partial(i\,\alpha\wedge\bar\alpha\wedge i\,\gamma\wedge\bar\gamma) = -i\,\alpha\wedge\bar\alpha\wedge i\,\gamma\wedge\overline{\partial\gamma} = i\,\alpha\wedge\bar\alpha\wedge i\,\gamma\wedge\bar\alpha\wedge\bar\beta =0$ since $\bar\alpha\wedge\bar\alpha = 0$,

\vspace{1ex}

$\bar\partial(i\,\beta\wedge\bar\beta\wedge i\,\gamma\wedge\bar\gamma) = -i\,\beta\wedge\bar\beta\wedge i\,\gamma\wedge\overline{\partial\gamma} = i\,\beta\wedge\bar\beta\wedge i\,\gamma\wedge\bar\alpha\wedge\bar\beta =0$ since $\bar\beta\wedge\bar\beta = 0$,

\vspace{1ex}

$\bar\partial(i\,\alpha\wedge\bar\beta\wedge i\,\gamma\wedge\bar\gamma) = -i\,\alpha\wedge\bar\beta\wedge i\,\gamma\wedge\overline{\partial\gamma} = i\,\alpha\wedge\bar\beta\wedge i\,\gamma\wedge\bar\alpha\wedge\bar\beta =0$ since $\bar\beta\wedge\bar\beta = 0$,

\vspace{1ex}

$\bar\partial(i\,\beta\wedge\bar\alpha\wedge i\,\gamma\wedge\bar\gamma) = -i\,\beta\wedge\bar\alpha\wedge i\,\gamma\wedge\overline{\partial\gamma} = i\,\beta\wedge\bar\alpha\wedge i\,\gamma\wedge\bar\alpha\wedge\bar\beta =0$ since $\bar\alpha\wedge\bar\alpha = 0$.

\vspace{2ex}

\noindent We conclude from these identities and from (\ref{eqn:omega_t_11_J0_square}) that $\partial\bar\partial(\omega_t^{1,\,1})^2 = 0$, so $\omega_t^{1,\,1}$ is indeed a Gauduchon metric on $X_0$ for all $t\in\Delta$ close to $0$.  \end{proof}

We now observe that, in a certain sense, there are as ``many'' Aeppli-Gauduchon classes of the type $[(\omega_t^{1,\,1})^2]_A$ as elements in the Gauduchon cone ${\cal G}_{X_0}$. 
\begin{Lem}\label{Lem:Aeppli_omega_t_11^2} For every $t\in\Delta$ sufficiently close to $0$, the Aeppli-Gauduchon class $[(\omega_t^{1,\,1})^2]_A\in{\cal G}_{X_0}$ 
satisfies the following identity  
\begin{eqnarray}\label{eqn:Aeppli_omega_t_11^2}
\nonumber\frac{1}{2}[(\omega_t^{1,\,1})^2]_A & 
= &  (1+c_1(t))(1+c_3(t))\,[i\,\alpha\wedge\bar\alpha\wedge i\,\gamma\wedge\bar\gamma]_A 
+ (1+c_2(t))(1+c_3(t))\,[i\,\beta\wedge\bar\beta\wedge i\,\gamma\wedge\bar\gamma]_A \\
   & + & d(t)(1+c_3(t))\,[i\,\alpha\wedge\bar\beta\wedge i\,\gamma\wedge\bar\gamma ]_A 
  +  \overline{d(t)}(1+c_3(t))\,[ i\,\beta\wedge\bar\alpha\wedge i\,\gamma\wedge\bar\gamma]_A.
  \end{eqnarray}
\noindent Note that since the classes $[i\,\alpha\wedge\bar\alpha\wedge i\,\gamma\wedge\bar\gamma]_A$, $[i\,\beta\wedge\bar\beta\wedge i\,\gamma\wedge\bar\gamma]_A$,
$[i\,\alpha\wedge\bar\beta\wedge i\,\gamma\wedge\bar\gamma]_A$, $[i\,\beta\wedge\bar\alpha\wedge i\,\gamma\wedge\bar\gamma]_A$ generate $H^{2,\,2}_A(X_0,\,\C)$ over $\C$, 
the real classes $[i\,\alpha\wedge\bar\alpha\wedge i\,\gamma\wedge\bar\gamma]_A$, $[i\,\beta\wedge\bar\beta\wedge i\,\gamma\wedge\bar\gamma]_A$, 
$[i\,\alpha\wedge\bar\beta\wedge i\,\gamma\wedge\bar\gamma + i\,\beta\wedge\bar\alpha\wedge i\,\gamma\wedge\bar\gamma]_A$ 
and $\frac{1}{2i}\,([i\,\alpha\wedge\bar\beta\wedge i\,\gamma\wedge\bar\gamma - i\,\beta\wedge\bar\alpha\wedge i\,\gamma\wedge\bar\gamma]_A)$ generate $H^{2,\,2}_A(X_0,\,\R)$
over $\R$.

\end{Lem}

\begin{proof} Identity (\ref{eqn:Aeppli_omega_t_11^2}) follows from (\ref{eqn:omega_t_11_J0_square}) after noticing that, since $\alpha\wedge\beta = -\partial\gamma$, we have

\vspace{1ex}

\hspace{6ex}  $i\,\alpha\wedge\bar\alpha\wedge i\,\beta\wedge\bar\beta = \partial\gamma\wedge\bar\partial\bar\gamma = \partial(\gamma\wedge\bar\partial\bar\gamma) \in\mbox{Im}\,\partial\subset\mbox{Im}\,\partial + \mbox{Im}\,\bar\partial$,

\vspace{1ex}

\noindent hence $[i\,\alpha\wedge\bar\alpha\wedge i\,\beta\wedge\bar\beta]_A=0$.  \end{proof}

\subsection{Use of the $sGG$ property}

Recall that the Iwasawa manifold $X_0$ and all its small deformations $X_t$ are {\it sGG manifolds} ([PU14]). As such, there are {\it canonical surjections}
\begin{equation}\label{eqn:can-surj_Xt}P_t : H^4_{DR}(X,\,\R)\twoheadrightarrow H^{2,\,2}_A(X_t,\,\R), \hspace{2ex} \{\Omega\}_{DR}\mapsto [\Omega^{2,\,2}_t]_A,\end{equation}
\noindent where $\Omega^{2,\,2}_t$ is the component of $J_t$-bidegree $(2,\,2)$ of $\Omega$, while $X$ is the $C^{\infty}$ manifold underlying the fibres $X_t$.
Moreover, for every fixed Hermitian metric $\omega_t$ on $X_t$, 
there is a lift of $P_t$ naturally associated with $\omega_t$, namely an injection
\begin{equation}\label{eqn:injection_A-DR_X_t}Q_{\omega_t}: H^{2,\,2}_A(X_t,\,\R) \hookrightarrow H^4_{DR}(X,\,\R), \hspace{2ex} [\Omega^{2,\,2}]_A\mapsto\{\Omega\}_{DR},\end{equation}
\noindent such that $P_t\circ Q_{\omega_t}: H^{2,\,2}_A(X_t,\,\R) \longrightarrow H^{2,\,2}_A(X_t,\,\R)$ is the identity map, defined in the following way (cf. [PU14, $\S. 5.1$). For every class $[\Omega^{2,\,2}]_A\in H^{2,\,2}_A(X_t,\,\R)$, let $\Omega^{2,\,2}_A$ be the (unique) Aeppli-harmonic representative of $[\Omega^{2,\,2}]_A$ w.r.t. the Aeppli Laplacian $\Delta_{A,\,\omega_t}$ associated with the metric $\omega_t$.\footnote{See [Sch07] for the definition of the Aeppli Laplacian.} Let $\Omega^{3,\,1}_A$ be the (unique) minimal $L^2_{\omega_t}$-norm solution of the $\bar\partial$-equation
\begin{equation}\label{eqn:equation_3-1}\bar\partial_t\Omega^{3,\,1}_A = -\partial_t\Omega^{2,\,2}_A.\end{equation}
\noindent This equation is solvable thanks to the sGG property of the manifold $X_t$ for all $t\in\Delta$ sufficiently close to $0$. 
Indeed, $n$-dimensional sGG manifolds are characterised by the fact that every $d$-closed $\partial$-exact $(n,\,n-1)$-form is $\bar\partial$-exact ([PU14, Lemma 1.2]). Here $n=3$, so $\partial_t\Omega^{2,\,2}_A$ is $\bar\partial_t$-exact. Thus, $\Omega^{3,\,1}_A$ exists and is given by the Neumann formula $\Omega^{3,\,1}_A = -\Delta_t^{''-1}\bar\partial_t^{\star}(\partial_t\Omega^{2,\,2}_A)$, where the formal adjoint $\bar\partial_t^{\star}$ of $\bar\partial_t$ and the Laplacian $\Delta''_t = \bar\partial_t\bar\partial_t^{\star} + \bar\partial_t^{\star}\bar\partial_t$ 
are computed w.r.t. the $L^2$ inner product induced by $\omega_t$, while $\Delta_t^{''-1}$ is the Green operator of $\Delta''_t$. 
Finally, we put  
\begin{equation*}
\Omega = \Omega_{\omega_t}:= \Omega^{3,\,1}_A + \Omega^{2,\,2}_A + \overline{\Omega^{3,\,1}_A}\end{equation*}
\noindent which is easily seen to be $d$-closed, to complete the definition (\ref{eqn:injection_A-DR_X_t}) of $Q_{\omega_t}$ (cf. [PU14]).

\begin{Conc}\label{Conc:metric_4-space} With every Hermitian metric $\omega_t$ on a small deformation $X_t$ of the Iwasawa manifold $X = X_0$ there is associated a $4$-dimensional real vector subspace of $H^4_{DR}(X,\,\R)$ as follows
\begin{equation}\label{eqn:metric_4-space}\omega_t\mapsto Q_{\omega_t}(H^{2,\,2}_A(X_t,\,\R))\subset H^4_{DR}(X,\,\R).\end{equation}

\end{Conc}

\vspace{3ex}

Besides the metric-induced injections $Q_{\omega_t}$  of (\ref{eqn:injection_A-DR_X_t}), there are canonical injections as follows.

\begin{Lem}\label{Lem:canonical-injections} $(a)$\,Let $X = X_0$ be the Iwasawa manifold. There is a {\bf canonical} linear {\bf injection}
\begin{equation}\label{eqn:can_inj_X_0}I_0: H^{2,\,2}_A(X_0,\,\C) \longrightarrow H^4_{DR}(X,\,\C).\end{equation}

$(b)$\, Let $\omega = \omega_0: =i\alpha\wedge\bar\alpha + i\beta\wedge\bar\beta + i\gamma\wedge\bar\gamma$ be the metric on the Iwasawa manifold $X = X_0$ 
canonically induced by the complex parallelisable structure of $X$ (cf. (\ref{eqn:omega_t_def})).

The injection $Q_{\omega_0}: H^{2,\,2}_A(X_0,\,\R) \hookrightarrow H^4_{DR}(X,\,\R)$ of (\ref{eqn:injection_A-DR_X_t}) induced by $\omega_0$ 
coincides with the canonical injection $I_0: H^{2,\,2}_A(X_0,\,\R) \hookrightarrow H^4_{DR}(X,\,\R)$ of (\ref{eqn:can_inj_X_0}).
\end{Lem}

\begin{proof}
$(a)$\, The contention follows from the explicit descriptions\eqref{eqn:DeRham} and \eqref{eqn:H22_A_t} of the cohomology groups involved. Specifically, $I_0$ is defined by letting
\begin{equation}\label{eqn:I_0_def}H^{2,\,2}_A(X_0,\,\C)\ni [\Omega^{2,\,2}]_A\mapsto\{\Omega^{2,\,2}\}_{DR} := I_0([\Omega^{2,\,2}]_A)\in H^4_{DR}(X,\,\C)  \end{equation}
\noindent for every $\Omega^{2,\,2}\in\{\alpha\wedge\gamma\wedge\bar\alpha\wedge\bar\gamma,\,\alpha\wedge\gamma\wedge\bar\beta\wedge\bar\gamma,
\,\beta\wedge\gamma\wedge\bar\alpha\wedge\bar\gamma,\,\beta\wedge\gamma\wedge\bar\beta\wedge\bar\gamma\} $ and extending by linearity. 
It is implicit that the forms $\alpha\wedge\gamma\wedge\bar\alpha\wedge\bar\gamma,\,\alpha\wedge\gamma\wedge\bar\beta\wedge\bar\gamma,
\,\beta\wedge\gamma\wedge\bar\alpha\wedge\bar\gamma,\,\beta\wedge\gamma\wedge\bar\beta\wedge\bar\gamma$ are all $d$-closed, as can be readily checked.

\vspace{1ex}

$(b)$\, The representatives $\alpha\wedge\gamma\wedge\bar\alpha\wedge\bar\gamma,\, \alpha\wedge\gamma\wedge\bar\beta\wedge\bar\gamma,
\, \beta\wedge\gamma\wedge\bar\alpha\wedge\bar\gamma,\, \beta\wedge\gamma\wedge\bar\beta\wedge\bar\gamma$ of the four Aeppli classes generating 
$H^{2,\,2}_A(X_0,\,\C)$ are all in $\ker\partial^{\star}\cap\ker\bar\partial^{\star}$ when the adjoints $\partial^{\star}$ 
and $\bar\partial^{\star}$ are computed w.r.t. $\omega_0$. Indeed,

\vspace{1ex}

$(1)$\, the identity $\partial^{\star}(\alpha\wedge\gamma\wedge\bar\alpha\wedge\bar\gamma) = 0$ is equivalent 
to $\langle\langle\alpha\wedge\gamma\wedge\bar\alpha\wedge\bar\gamma,\, \partial u\rangle\rangle = 0$ for all forms $u\in C^{\infty}_{1,\,2}(X,\,\C)$. 
Now, the only generators of $C^{\infty}_{1,\,2}(X,\,\C)$ that are not $\partial$-closed are $\gamma\wedge\bar\alpha\wedge\bar\beta$, $\gamma\wedge\bar\alpha\wedge\bar\gamma$ 
and $\gamma\wedge\bar\beta\wedge\bar\gamma$. When $u$ is one of these forms, we have $\partial u = -\alpha\wedge\beta\wedge\bar\alpha\wedge\bar\beta$, 
or $\partial u = -\alpha\wedge\beta\wedge\bar\alpha\wedge\bar\gamma$, or $\partial u = -\alpha\wedge\beta\wedge\bar\beta\wedge\bar\gamma$ 
and the inner product of any of these forms against $\alpha\wedge\gamma\wedge\bar\alpha\wedge\bar\gamma$ vanishes because they are all part of an $\omega_0$-orthonormal basis 
and the ones do not contain $\gamma$ while the other does. The same argument proves the $\partial^{\star}$-closedness of the 
remaining forms $\alpha\wedge\gamma\wedge\bar\beta\wedge\bar\gamma,\, \beta\wedge\gamma\wedge\bar\alpha\wedge\bar\gamma,\, \beta\wedge\gamma\wedge\bar\beta\wedge\bar\gamma$ 
since they all contain $\gamma$.

\vspace{1ex}

$(2)$\, the identity $\bar\partial^{\star}(\alpha\wedge\gamma\wedge\bar\alpha\wedge\bar\gamma) = 0$ 
is equivalent to $\langle\langle\alpha\wedge\gamma\wedge\bar\alpha\wedge\bar\gamma,\, \bar\partial v\rangle\rangle = 0$ 
for all forms $v\in C^{\infty}_{2,\,1}(X,\,\C)$. The only generators of $C^{\infty}_{2,\,1}(X,\,\C)$ 
that are not $\bar\partial$-closed are $\alpha\wedge\beta\wedge\bar\gamma$, $\alpha\wedge\gamma\wedge\bar\gamma$ 
and $\beta\wedge\gamma\wedge\bar\gamma$. When $v$ is one of these forms, we have $\bar\partial v = -\alpha\wedge\beta\wedge\bar\alpha\wedge\bar\beta$, 
or $\bar\partial v = -\alpha\wedge\gamma\wedge\bar\alpha\wedge\bar\beta$, or $\bar\partial v = -\beta\wedge\gamma\wedge\bar\alpha\wedge\bar\beta$ 
and the inner product of any of these forms against $\alpha\wedge\gamma\wedge\bar\alpha\wedge\bar\gamma$ vanishes 
because they are all part of an $\omega_0$-orthonormal basis and the ones do not contain $\bar\gamma$ while the other does. 
The same argument proves the $\bar\partial^{\star}$-
closedness of the remaining forms $\alpha\wedge\gamma\wedge\bar\beta\wedge\bar\gamma,\, \beta\wedge\gamma\wedge\bar\alpha\wedge\bar\gamma,
\, \beta\wedge\gamma\wedge\bar\beta\wedge\bar\gamma$ since they all contain $\bar\gamma$.

\vspace{1ex}

Now, the forms $\alpha\wedge\gamma\wedge\bar\alpha\wedge\bar\gamma,\, \alpha\wedge\gamma\wedge\bar\beta\wedge\bar\gamma,
\, \beta\wedge\gamma\wedge\bar\alpha\wedge\bar\gamma,\, \beta\wedge\gamma\wedge\bar\beta\wedge\bar\gamma$ are also $\partial\bar\partial$-closed (see lemma~\ref{lem:pluri-closed}), so they must be {\bf Aeppli-harmonic}
\footnote{For the definition of the Aeppli Laplacian $\Delta_A$ (an elliptic operator of order $4$ whose kernel is isomorphic to the corresponding Aeppli cohomology group) and the description of its kernel used here, see [Sch07].} w.r.t. $\omega_0$, i.e.
$$\alpha\wedge\gamma\wedge\bar\alpha\wedge\bar\gamma,\, \alpha\wedge\gamma\wedge\bar\beta\wedge\bar\gamma,\, \beta\wedge\gamma\wedge\bar\alpha\wedge\bar\gamma,
\, \beta\wedge\gamma\wedge\bar\beta\wedge\bar\gamma \in\ker\Delta_{A,\,\omega_0} = \ker(\partial\bar\partial)\cap\ker\partial^{\star}_{\omega_0}\cap\ker\bar\partial^{\star}_{\omega_0}.$$
\noindent Thus, for any class $[\Omega^{2,\,2}]_A = c_1\,[\alpha\wedge\gamma\wedge\bar\alpha\wedge\bar\gamma]_A + c_2\,[\alpha\wedge\gamma\wedge\bar\beta\wedge\bar\gamma]_A
+ c_3\,[\beta\wedge\gamma\wedge\bar\alpha\wedge\bar\gamma]_A + c_4\,[\beta\wedge\gamma\wedge\bar\beta\wedge\bar\gamma]_A \in H^{2,\,2}_A(X_0,\,\R)$ 
with coefficients $c_1,\dots , c_4\in\R$, the Aeppli-harmonic representative w.r.t. $\omega_0$ is
$$\Omega^{2,\,2}_A = c_1\,\alpha\wedge\gamma\wedge\bar\alpha\wedge\bar\gamma + c_2\,\alpha\wedge\gamma\wedge\bar\beta\wedge\bar\gamma
+ c_3\,\beta\wedge\gamma\wedge\bar\alpha\wedge\bar\gamma + c_4\,\beta\wedge\gamma\wedge\bar\beta\wedge\bar\gamma.$$

\noindent Meanwhile, the forms $\alpha\wedge\gamma\wedge\bar\alpha\wedge\bar\gamma,\, \alpha\wedge\gamma\wedge\bar\beta\wedge\bar\gamma,
\, \beta\wedge\gamma\wedge\bar\alpha\wedge\bar\gamma,\, \beta\wedge\gamma\wedge\bar\beta\wedge\bar\gamma$ are all $d$-closed, 
hence $d\Omega^{2,\,2}_A = 0$. Since $\Omega^{2,\,2}_A$ is of pure type, this implies that $\partial_0\Omega^{2,\,2}_A = 0$. 
Consequently, the minimal $L^2$-norm solution $\Omega^{3,\,1}_A$ of equation $\bar\partial_0\Omega^{3,\,1}_A = -\partial_0\Omega^{2,\,2}_A$ (cf. (\ref{eqn:equation_3-1}))
is the zero form. From (\ref{eqn:injection_A-DR_X_t}) and (\ref{eqn:Omega_def}) we get
$$Q_{\omega_0}([\Omega^{2,\,2}]_A) = Q_{\omega_0}([\Omega^{2,\,2}_A]_A) = \{\Omega^{2,\,2}_A\}_{DR}.$$

\noindent Comparing with (\ref{eqn:I_0_def}), we see that $Q_{\omega_0}([\Omega^{2,\,2}]_A) = I_0([\Omega^{2,\,2}]_A)$. \end{proof}

\begin{Cor}\label{Cor:subbundle_H22_H4} Let $(X_t)_{t\in\Delta}$ be the Kuranishi family of the Iwasawa manifold $X=X_0$. Then
$$\Delta\ni t\mapsto H^{2,\,2}_A(X_t,\,\C)$$
\noindent is a $C^{\infty}$ vector bundle of rank $4$ that we shall denote by ${\cal H}^{2,\,2}_A$. 
Moreover, ${\cal H}^{2,\,2}_A$ injects {\bf canonically} as a $C^{\infty}$ vector subbundle of the constant bundle ${\cal H}^4\rightarrow\Delta$ of fibre $H^4_{DR}(X,\,\C)$
in the following way: for every $t\in\Delta$ sufficiently close to $0$, we define the {\bf canonical} linear {\bf injection}
\begin{equation}\label{eqn:can_inj_X_t}I_t: H^{2,\,2}_A(X_t,\,\C) \longrightarrow H^4_{DR}(X,\,\C)   \hspace{2ex} \mbox{by}\hspace{2ex} I_t=Q_{\omega_t},\end{equation}
\noindent the injection (\ref{eqn:injection_A-DR_X_t}) induced by the canonical metric $\omega_t= i\alpha_t\wedge\bar\alpha_t + i\beta_t\wedge\bar\beta_t 
+ i\gamma_t\wedge\bar\gamma_t$ of (\ref{eqn:omega_t_def}) on $X_t$. 
\end{Cor}

\begin{proof}
Let $(\gamma_t)_{t\in\Delta}$ be any $C^{\infty}$ family of Hermitian metrics on the fibres $(X_t)_{t\in\Delta}$ and let $(\Delta_{A,\, t})_{t\in\Delta}$ 
be the associated $C^{\infty}$ family of elliptic Aeppli Laplacians inducing Hodge isomorphisms $\ker\Delta_{A,\, t}\simeq H^{2,\,2}_A(X_t,\,\C)$ for $t\in\Delta$ ([Sch07]). 
Meanwhile, $\mbox{dim}_{\C}H^{2,\,2}_A(X_t,\,\C) = 4$ for all $t\in\Delta$ ([Ang11, $\S.4.3$]). Since the dimension of the kernel of $\Delta_{A,\, t}$ is independent of $t\in\Delta$,
we infer by ellipticity from [KS60] that $\Delta\ni t\mapsto\ker\Delta_{A,\, t}\simeq H^{2,\,2}_A(X_t,\,\C)$ is a $C^{\infty}$ vector bundle of rank $4$.

The last statement follows from Lemma \ref{Lem:canonical-injections} and from the $C^{\infty}$ dependence on $t$ of the injections $I_t$ 
(itself a consequence of the $C^{\infty}$ dependence on $t$ of each of the forms $\alpha_t, \beta_t, \gamma_t$).   \end{proof}

\begin{Rem}\label{Rem:canonical_injection_metric_t} 
Note that for $t\in\Delta_{[\gamma]}\setminus\{0\}$, $I_t$ cannot be defined by analogy with definition (\ref{eqn:I_0_def}) of $I_0$ 
since the representatives $\alpha_t\wedge\gamma_t\wedge\bar\alpha_t\wedge\bar\gamma_t,\, \alpha_t\wedge\gamma_t\wedge\bar\beta_t\wedge\bar\gamma_t,
\, \beta_t\wedge\gamma_t\wedge\bar\alpha_t\wedge\bar\gamma_t,\, \beta_t\wedge\gamma_t\wedge\bar\beta_t\wedge\bar\gamma_t$ of the classes generating $H^{2,\,2}_A(X_t,\,\C)$
(cf. \eqref{eqn:H22_A_t} ) are not $d$-closed.
\end{Rem}

For future reference, we notice the following trivialisation of the vector bundle $\Delta\ni t\mapsto H^{2,\,2}_A(X_t,\,\C)$. The following definition is meaningful thanks to Lemma \ref{lem:pluri-closed} of the following subsection.

\begin{Def}\label{Def:H22_A_t-H22_A_0} For every $t\in\Delta$, we consider the isomorphism of complex vector spaces
\begin{equation}\label{eqn:H22_A_t-H22_A_0}B_t: H^{2,\,2}_A(X_t,\,\C) \longrightarrow H^{2,\,2}_A(X_0,\,\C)\end{equation}
\noindent defined by $[\alpha_t\wedge\gamma_t\wedge\bar\alpha_t\wedge\bar\gamma_t]_A\mapsto [\alpha\wedge\gamma\wedge\bar\alpha\wedge\bar\gamma]_A$, 
$[\alpha_t\wedge\gamma_t\wedge\bar\beta_t\wedge\bar\gamma_t]_A\mapsto [\alpha\wedge\gamma\wedge\bar\beta\wedge\bar\gamma]_A$,
$[\beta_t\wedge\gamma_t\wedge\bar\alpha_t\wedge\bar\gamma_t]_A\mapsto [\beta\wedge\gamma\wedge\bar\alpha\wedge\bar\gamma]_A$,
$[\beta_t\wedge\gamma_t\wedge\bar\beta_t\wedge\bar\gamma_t]_A\mapsto [\beta\wedge\gamma\wedge\bar\beta\wedge\bar\gamma]_A$.
\end{Def} 

\begin{Cor}\label{Cor:metric_4-space_t} With every Aeppli-Gauduchon class of the shape $[(\omega_t^{1,\,1})^2]_A\in{\cal G}_{X_0}$ 
(for $t\in\Delta$) on the Iwasawa manifold $X = X_0$ there is associated a $4$-dimensional real vector subspace of $H^4_{DR}(X,\,\R)$ as follows
\begin{equation}\label{eqn:metric_4-space_t}{\cal G}_{X_0}\ni[(\omega_t^{1,\,1})^2]_A\mapsto \widetilde{H^{2,\,2}_t}
:=Q_{\omega_t^{1,\,1}}\bigg(H^{2,\,2}_A(X_0,\,\R)\bigg)\subset H^4_{DR}(X,\,\R),\end{equation}
\noindent where $Q_{\omega_t^{1,\,1}}: H^{2,\,2}_A(X_0,\,\R)\hookrightarrow H^4_{DR}(X,\,\R)$ is the injective linear map of (\ref{eqn:injection_A-DR_X_t}) 
defined by the metric $\omega_t^{1,\,1}$ on $X_0$. 

\end{Cor}

\subsection{The Hodge bundles ${\cal H}^{2,\,1}_{[\gamma]}\simeq{\cal H}^{2,\,2}_A$ and ${\cal H}^4$ over $\Delta_{[\gamma]}$}\label{subsection:Hodge-bundles}

The following description of the Aeppli cohomology groups of bidegree $(2,\,2)$ of the small deformations $X_t$ with $t\in\Delta$ of the Iwasawa manifold $X=X_0$ will be used several times in this section. 

\begin{Lem}\label{lem:pluri-closed}

 For every $t\in\Delta$, the forms $\alpha_t\wedge\bar\alpha_t\wedge\gamma_t\wedge\bar\gamma_t,\beta_t\wedge\bar\beta_t\wedge\gamma_t\wedge\bar\gamma_t, \alpha_t\wedge\bar\beta_t\wedge\gamma_t\wedge\bar\gamma_t$ and $\beta_t\wedge\bar\alpha_t\wedge\gamma_t\wedge\bar\gamma_t$ are $\partial_t\bar\partial_t$-closed and
\begin{equation}\label{eqn:H22_A_t}H^{2,\,2}_A(X_t,\,\C) = \bigg\langle[\alpha_t\wedge\gamma_t\wedge\bar\alpha_t\wedge\bar\gamma_t]_A,
\, [\alpha_t\wedge\gamma_t\wedge\bar\beta_t\wedge\bar\gamma_t]_A,\, [\beta_t\wedge\gamma_t\wedge\bar\alpha_t\wedge\bar\gamma_t]_A,
\, [\beta_t\wedge\gamma_t\wedge\bar\beta_t\wedge\bar\gamma_t]_A\bigg\rangle.\end{equation}

\end{Lem}

\begin{proof}
We spell out the details of the pluriclosedness argument, that is similar for the four forms, when $t\in\Delta_{[\gamma]}$. It goes
\begin{eqnarray}
\nonumber\partial_t(\alpha_t\wedge\gamma_t\wedge\bar\alpha_t\wedge\bar\gamma_t) & = & -\alpha_t\wedge\partial_t\gamma_t\wedge\bar\alpha_t\wedge\bar\gamma_t
- \alpha_t\wedge\gamma_t\wedge\bar\alpha_t\wedge\overline{\bar\partial_t\gamma_t}\\
\nonumber & = & -\sigma_{12}(t)\, \alpha_t\wedge(\alpha_t\wedge\beta_t)\wedge\bar\alpha_t\wedge\bar\gamma_t
-  \alpha_t\wedge\gamma_t\wedge\bar\alpha_t\wedge(\overline{\sigma_{2\bar{2}}(t)}\,\bar\beta_t\wedge\beta_t) \\
\nonumber & = & - \overline{\sigma_{2\bar{2}}(t)}\, \alpha_t\wedge\bar\alpha_t\wedge\beta_t\wedge\bar\beta_t\wedge\gamma_t, \end{eqnarray}
\noindent where we used the structure equations~\eqref{eq:structure} to get the second line above. So, $\partial_t(\alpha_t\wedge\gamma_t\wedge\bar\alpha_t\wedge\bar\gamma_t)\neq 0$ when $t\neq 0$, but ~\eqref{eq:structure} implies that $\bar\partial_t\gamma_t$ comes from a $2$-form on $B_t$, so applying $\bar\partial_t$ we get $\bar\partial_t\partial_t(\alpha_t\wedge\gamma_t\wedge\bar\alpha_t\wedge\bar\gamma_t)= 0$.

 Now, it was shown in [Ang11, Remark 5.2] that 
\begin{equation*}
H^{1,\,1}_{BC}(X_t,\,\C) = \bigg\langle[\alpha_t\wedge\bar\alpha_t]_{BC},\, [\alpha_t\wedge\bar\beta_t]_{BC},
\, [\beta_t\wedge\bar\alpha_t]_{BC},\, [\beta_t\wedge\bar\beta_t]_{BC}\bigg\rangle, \hspace{3ex} t\in\Delta.\end{equation*}
\noindent This implies (\ref{eqn:H22_A_t}) via the non-degenerate duality $H^{1,\,1}_{BC}(X_t,\,\C)\times H^{2,\,2}_A(X_t,\,\C)\longrightarrow\C$, $([u]_{BC},\,[v]_A)\mapsto\int_Xu\wedge v$, recalled in (\ref{eqn:duality}).
\end{proof}

\begin{Obs}\label{Obs:H21_gamma-H22_isomorphism} 
\begin{enumerate}
 \item[$(a)$]On the Iwasawa manifold $X_0$, there is a canonical isomorphism
\begin{equation}\label{eqn:H21_gamma-H22_isomorphism_0}H^{2,\,1}_{[\gamma]}(X_0,\,\C) \stackrel{\simeq}{\longrightarrow}H^{2,\,2}_A(X_0,\,\C)\end{equation}
\noindent defined by $[\Gamma]_{\bar\partial}\mapsto[\Gamma\wedge\bar\gamma]_A$ for $\Gamma\in\{\alpha\wedge\gamma\wedge\bar\alpha, \alpha\wedge\gamma\wedge\bar\beta, 
\beta\wedge\gamma\wedge\bar\alpha, \beta\wedge\gamma\wedge\bar\beta\}$.
\item[$(b)$] In the Kuranishi family $(X_t)_{t\in\Delta}$ of the Iwasawa manifold $X_0$, there is a canonical isomorphism
\begin{equation}\label{eqn:H21_gamma-H22_isomorphism_t}A_t : H^{2,\,1}_{[\gamma]}(X_t,\,\C) \stackrel{\simeq}{\longrightarrow}H^{2,\,2}_A(X_t,\,\C) \hspace{3ex}
\mbox{for every} \hspace{2ex} t\in\Delta_{[\gamma]} \end{equation}
\noindent defined by $[\Gamma]_{\bar\partial}\mapsto[\Gamma\wedge\bar\gamma_t]_A$ for $\Gamma\in\{\Gamma_1(t),\,\Gamma_2(t),\, \Gamma_3(t),\, \Gamma_4(t)\}$ (see (\ref{eqn:Gamma_j-forms_t}) and (\ref{eqn:H21_gamma_t})). (Note that $A_t$ depends anti-holomorphically on $t$.)

In particular, the rank-four $C^{\infty}$ vector bundles $\Delta_{[\gamma]}\ni t\mapsto H^{2,\,1}_{[\gamma]}(X_t,\,\C)$ (of Definition \ref{Def:H21_gamma_t}) and $\Delta_{[\gamma]}\ni t\mapsto H^{2,\,2}_A(X_t,\,\C)$ (of Corollary \ref{Cor:subbundle_H22_H4}) are canonically isomorphic, i.e. ${\cal H}^{2,\,1}_{[\gamma]}\simeq{\cal H}^{2,\,2}_A$.
\end{enumerate}
\end{Obs}

\begin{proof}
Part $(a)$ is a special case of part $(b)$. To prove $(b)$, we note that in conjunction with the description of $H^{2,\,1}_{[\gamma]}(X_t,\,\C)$ given at the end of $\S.$\ref{subsection:2-1} as $H^{2,\,1}_{[\gamma]}(X_t,\,\C)= \langle [\Gamma_1(t)]_{\bar\partial},\, [\Gamma_2(t)]_{\bar\partial},\, [\Gamma_3(t)]_{\bar\partial},
\, [\Gamma_4(t)]_{\bar\partial}\rangle \subset H^{2,\,1}_{\bar\partial}(X_t,\,\C)$ for all $t\in\Delta_{[\gamma]}$, (\ref{eqn:H22_A_t}) proves the isomorphism (\ref{eqn:H21_gamma-H22_isomorphism_t}). Indeed, $\Gamma_1(t)\wedge\bar\gamma_t = \alpha_t\wedge\gamma_t\wedge\bar\alpha_t\wedge\bar\gamma_t$, 
$\Gamma_2(t)\wedge\bar\gamma_t = \alpha_t\wedge\gamma_t\wedge\bar\beta_t\wedge\bar\gamma_t$, $\Gamma_3(t)\wedge\bar\gamma_t 
= \beta_t\wedge\gamma_t\wedge\bar\alpha_t\wedge\bar\gamma_t$, $\Gamma_4(t)\wedge\bar\gamma_t = \beta_t\wedge\gamma_t\wedge\bar\beta_t\wedge\bar\gamma_t$
and all these forms are $\partial\bar\partial$-closed as proved in Lemma \ref{lem:pluri-closed}.

\end{proof}

\subsection{Bringing the families of metrics $(\omega_t)_{t\in\Delta_{[\gamma]}}$ and $(\omega_t^{1,\,1})_{t\in\Delta_{[\gamma]}}$ together}

We can now describe a VHS parametrised by Aeppli-Gauduchon classes on $X_0$. It is related to the VHS of weight $2$ 
induced by the holomorphic family $(B_t)_{t\in\Delta_{[\gamma]}}$ of $2$-dimensional complex tori. Since the $B_t$'s are K\"ahler, we get a weight-two Hodge decomposition

\begin{equation}\label{eqn:VHS_weight2_tori}H^2(B,\,\C) \simeq H^{2,\,0}(B_t,\,\C)\oplus H^{1,\,1}(B_t,\,\C) \oplus H^{0,\,2}(B_t,\,\C),  \hspace{3ex} t\in\Delta,\end{equation}   
\noindent where $B$ stands for the $C^{\infty}$ manifold underlying the complex tori $B_t$ and $H^2(B,\,\C): = H^2_{DR}(B_t,\,\C)$ is the fibre of the constant bundle ${\cal H}^2(B)$ over $\Delta$ defined by the De Rham cohomology of degree $2$ of the tori $B_t$ with $t\in\Delta$. As usual, we get holomorphic vector bundles
\begin{equation}\label{eqn:Hodge-filtration_B}F^1{\cal H}^2(B):=\bigg(\Delta\ni t\mapsto H^{2,\,0}(B_t,\,\C)\oplus H^{1,\,1}(B_t,\,\C)\bigg) 
\supset F^2{\cal H}^2(B):=\bigg(\Delta\ni t\mapsto H^{2,\,0}(B_t,\,\C)\bigg)\end{equation}

\noindent that constitute the Hodge filtration associated with the VHS (\ref{eqn:VHS_weight2_tori}). 
Let $D$ be the Gauss-Manin connection of the constant bundle ${\cal H}^2(B)$. It satisfies the transversality condition
\begin{equation}\label{eqn:transversality_B}D_{[\theta]}F^2{\cal H}^2(B)_t\subset  F^1{\cal H}^2(B)_t, \hspace{3ex}  t\in\Delta_{[\gamma]},\end{equation}
\noindent for all $[\theta]\in T^{1,\,0}_t\Delta_{[\gamma]} \simeq H^{1,\,1}(B_t,\,\C)\simeq H^{2,\,2}_A(X_t,\,\C) \simeq H^{2,\,1}_{[\gamma]}(X_t,\,\C)$.

 The second isomorphism of vector spaces on the previous line is a consequence of the description of $H^{1,\,1}(B_t,\,\C)$ as \begin{eqnarray}\label{eqn:H11_Bt_description}\nonumber H^{1,\,1}(B_t,\,\C) & = & \bigg\langle [\alpha_t\wedge\bar\alpha_t]_{\bar\partial},\, [\alpha_t\wedge\bar\beta_t]_{\bar\partial},
\, [\beta_t\wedge\bar\alpha_t]_{\bar\partial},\, [\beta_t\wedge\bar\beta_t]_{\bar\partial}\bigg\rangle \\
\nonumber & \simeq & \bigg\langle[\alpha_t\wedge\gamma_t\wedge\bar\alpha_t\wedge\bar\gamma_t]_A,\, [\alpha_t\wedge\gamma_t\wedge\bar\beta_t\wedge\bar\gamma_t]_A,
\, [\beta_t\wedge\gamma_t\wedge\bar\alpha_t\wedge\bar\gamma_t]_A,\, [\beta_t\wedge\gamma_t\wedge\bar\beta_t\wedge\bar\gamma_t]_A\bigg\rangle\\
  & = & H^{2,\,2}_A(X_t,\,\C)\simeq Q_{\omega_t}(H^{2,\,2}_A(X_t,\,\C)):=\widetilde{H^{2,\,2}_{\omega_t}}\subset H^4_{DR}(X,\,\C),  \hspace{3ex} t\in\Delta_{[\gamma]}, 
\end{eqnarray}

\noindent where the first identity on the last line is (\ref{eqn:H22_A_t}) and  $Q_{\omega_t}: H^{2,\,2}_A(X_t,\,\C)\hookrightarrow H^4_{DR}(X,\,\C)$ 
is the complexification of the injective linear map of (\ref{eqn:injection_A-DR_X_t}) defined by the Gauduchon metric $\omega_t$ of (\ref{eqn:omega_t_def}) on $X_t$ 
(and also denoted by $I_t$ in (\ref{eqn:can_inj_X_t})).

On the other hand,
\begin{eqnarray}\label{eqn:H20_B_H3_X} H^{2,\,0}(B_t,\,\C) \simeq  H^{3,\,0}(X_t,\,\C)\hookrightarrow H^3_{DR}(X,\,\C),   \hspace{3ex} t\in\Delta,\end{eqnarray}   
\noindent since $H^{2,\,0}(B_t,\,\C) = \langle[\alpha_t\wedge\beta_t]_{\bar\partial}\rangle$ and $H^{3,\,0}(X_t,\,\C) 
= \langle[\alpha_t\wedge\beta_t\wedge\gamma_t]_{\bar\partial}\rangle$, while the $\C$-line $H^{3,\,0}(X_t,\,\C)$ injects canonically into $H^3_{DR}(X,\,\C)$
as observed in Lemma \ref{Lem:Hn0_injection}. We get a canonical injection of holomorphic vector bundles
\begin{eqnarray}\label{eqn:F2H2_bundle-inj_H3}\nonumber j:F^2{\cal H}^2(B)\hookrightarrow {\cal H}^3\end{eqnarray}

\noindent such that $j_t:H^{2,\,0}(B_t,\,\C)\hookrightarrow H^3(X,\,\C)$ is the composition of the maps (\ref{eqn:H20_B_H3_X}) for every $t\in\Delta$.

 Together with (\ref{eqn:H11_Bt_description}), this gives an injection of holomorphic vector bundles

\begin{eqnarray}\label{eqn:F1H2_bundle-inj_H3-H4}\nonumber j\oplus Q :F^1{\cal H}^2(B)\hookrightarrow {\cal H}^3\oplus{\cal H}^4\end{eqnarray}

\noindent such that $(j\oplus Q)_t= j_t \oplus Q_{\omega_t}$ for all $t\in\Delta_{[\gamma]}$.

\vspace{2ex}

 We now anticipate the definition of what will be called later the {\bf complexified parameter set}:

$$\widetilde{{\cal G}_0} := \{[\omega_0^2]_A - t_{11}\,[i\beta\wedge\bar\alpha\wedge i\gamma\wedge\bar\gamma]_A + t_{22}\,[i\alpha\wedge\bar\beta\wedge i\gamma\wedge\bar\gamma]_A - t_{12}\,[i\beta\wedge\bar\beta\wedge i\gamma\wedge\bar\gamma]_A + t_{21}\,[i\alpha\wedge\bar\alpha\wedge i\gamma\wedge\bar\gamma]_A\,\mid\, t\in\Delta_{[\gamma]}\}$$
\hspace{3ex} $ \subset  H^{2,\,2}_A(X_0,\,\C).$

\vspace{2ex}

\noindent Recall the identification $\Delta_{[\gamma]} = \{t=(t_{11},\,t_{12},\,t_{21},\,t_{22})\in H^{0,\,1}_{[\gamma]}(X_0,\,T^{1,\,0}X_0) \,;\, |t|<\varepsilon\}$ for some small $\varepsilon>0$ when $H^{0,\,1}_{[\gamma]}(X_0,\,T^{1,\,0}X_0)$ is identified with $\C^4$ by the basis specified in (\ref{eqn:H^01_gamma_generators}). The set $\widetilde{{\cal G}_0}$ is a complexification of the {\bf parameter set} 

\begin{equation}\label{eqn:parameter-set_def}{\cal G}_0:= \{[(\omega_t^{1,\,1})^2]_A \,\mid\, t\in\Delta_{[\gamma]}\}\subset{\cal G}_{X_0}.\end{equation}

\noindent Thus, $\widetilde{{\cal G}_0}$ is a subset of the complexified Gauduchon cone $\widetilde{{\cal G}_{X_0}}\subset  H^{2,\,2}_A(X_0,\,\C)$ (cf. Defintion \ref{Def:G-cone-complexified}) of the Iwasawa manifold $X=X_0$.

\begin{Conc}\label{Conc:VHS_two-families-metrics} Let $(X_t)_{t\in\Delta_{[\gamma]}}$ be the local universal family of essential deformations of the Iwasawa manifold $X=X_0$.

$(i)$\, Our discussion so far can be summed up in the following diagram for all $t\in\Delta_{[\gamma]}$.

\vspace{3ex}

\noindent $\begin{CD}
H^{2,\,2}_A(X_t,\,\C)              @>\simeq>B_t>          H^{2,\,2}_A(X_0,\,\C) \\
@V\simeq VQ_{\omega_t}V                                       @V\simeq VQ_{\omega_t^{1,\,1}}V \\
H^4(X,\,\C)\supset H^{2,\,2}_A(X_t,\,\C)\simeq Q_{\omega_t}(H^{2,\,2}_A(X_t,\,\C)):=\widetilde{H^{2,\,2}_{\omega_t}}  @>\simeq>>      
Q_{\omega_t^{1,\,1}}(H^{2,\,2}_A(X_0,\,\C)):=\widetilde{H^{2,\,2}_t}\subset H^4(X,\,\C),
\end{CD}$

\vspace{3ex}
\noindent where the isomorphism $\widetilde{H^{2,\,2}_{\omega_t}}\to \widetilde{H^{2,\,2}_t}$ is the composition $Q_{\omega_t^{1,\,1}}\circ B_t\circ Q_{\omega_t}^{-1}$.

 $(ii)$\, Moreover, we get a $C^{\infty}$ vector subbundle of rank $4$ of the constant bundle ${\cal H}^4$:
\begin{equation}\label{eqn:subbundle_G_rank4_X0}\Delta_{[\gamma]}\simeq\widetilde{{\cal G}_0}\ni t\mapsto Q_{\omega_t}(H^{2,\,2}_A(X_t,\,\C))\subset H^4(X,\,\C),\end{equation}
\noindent denoted henceforth by $\widetilde{{\cal H}^{2,\,2}_\omega}$, and a {\bf holomorphic} vector subbundle of rank $1$ of the constant bundle ${\cal H}^3(X)$:

\begin{equation}\label{eqn:subbundle_G_rank1_X0}\Delta_{[\gamma]}\simeq\widetilde{{\cal G}_0}\ni t\mapsto H^{2,\,0}(B_t,\,\C)\stackrel{j_t}{\hookrightarrow} H^3(X,\,\C),\end{equation}
\noindent denoted henceforth by ${\cal H}^{2,\,0}(B) = F'_{{\cal G}}{\cal H}$, such that the following complex vector bundle of rank $5$, denoted henceforth by $F_{{\cal G}}{\cal H}^4:={\cal H}^{2,\,0}(B)\oplus\widetilde{{\cal H}^{2,\,2}_\omega}$,
\begin{equation}\label{eqn:subbundle_G_rank5_X0}\Delta_{[\gamma]}\simeq\widetilde{{\cal G}_0}\ni t\mapsto H^{2,\,0}(B_t,\,\C)\oplus Q_{\omega_t}(H^{2,\,2}_A(X_t,\,\C))
\subset H^3(X,\,\C) \oplus H^4(X,\,\C)\end{equation}
\noindent is a {\bf holomorphic} subbundle of the constant bundle ${\cal H}^3\oplus{\cal H}^4$ of fibre $ H^3(X,\,\C) \oplus H^4(X,\,\C)$ and is $C^\infty$ {\bf isomorphic} to $F^1{\cal H}^2(B)$.

 $(iii)$\, In particular, the vector bundles (\ref{eqn:subbundle_G_rank1_X0}) and (\ref{eqn:subbundle_G_rank4_X0}) define a VHS parametrised by the subset 
\begin{equation}\label{eqn:parametrising-subset}\Delta_{[\gamma]}\simeq\widetilde{{\cal G}_0}\subset\widetilde{{\cal G}_{X_0}}\end{equation} 
\noindent whose corresponding Hodge filtration $F_{{\cal G}}{\cal H}^4\supset F'_{{\cal G}}{\cal H}^4$ is $C^\infty$ {\bf isomorphic} to the Hodge filtration $F^1{\cal H}^2(B)\supset F^2{\cal H}^2(B)$ 
associated with the holomorphic family $(B_t)_{t\in\Delta_{[\gamma]}}$ of base tori of the family  $(X_t)_{t\in\Delta_{[\gamma]}}$.

\end{Conc}

 Only the {\it holomorphic} nature of the above vector bundle isomorphisms still needs a proof that is provided in the next subsection.

\subsection{Holomorphicity of the Hodge filtration parametrised by ${\cal G}_0$}\label{subsection:holomorphicity_Hodge_G}

We prove in this subsection that the Hodge filtration

$${\cal H}^3 \oplus {\cal H}^4 \supset F_{{\cal G}}{\cal H}^4 \supset F'_{{\cal G}}{\cal H}^4$$

\noindent constructed in the previous subsection (cf. Conclusion \ref{Conc:VHS_two-families-metrics}) consists of {\bf holomorphic vector subbundles} of the constant bundle ${\cal H}^3 \oplus {\cal H}^4$ of fibre $H^3(X,\,\C) \oplus H^4(X,\,\C)$ over $\overline{{\cal G}_0}$.

Our starting point is the following simple observation.

\begin{Lem}\label{Lem:H31_BC_injects_H22_A} For every $t\in\Delta$, there is a canonical linear injection

\begin{equation}\label{eqn:H31_BC_injects_H22_A}H^{3,\,1}_{BC}(X_t,\,\C)\hookrightarrow H^{2,\,2}_A(X_t,\,\C).\end{equation}

\end{Lem}

\noindent {\it Proof.} From [Ang14, p. 83] we infer that $H^{3,\,1}_{BC}(X_t,\,\C) =\langle[\alpha_t\wedge\beta_t\wedge\gamma_t\wedge\bar\alpha_t]_{BC},\,[\alpha_t\wedge\beta_t\wedge\gamma_t\wedge\bar\beta_t]_{BC}\rangle$ for all $t\in\Delta$. Coupled with (\ref{eqn:H22_A_t}), this allows us to explicitly define the canonical linear injection by

$$[\alpha_t\wedge\beta_t\wedge\gamma_t\wedge\bar\alpha_t]_{BC}\mapsto[\alpha_t\wedge\bar\alpha_t\wedge\gamma_t\wedge\bar\gamma_t]_A \hspace{2ex} \mbox{and} \hspace{2ex} [\alpha_t\wedge\beta_t\wedge\gamma_t\wedge\bar\beta_t]_{BC}\mapsto[\beta_t\wedge\bar\beta_t\wedge\gamma_t\wedge\bar\gamma_t]_A.$$

\noindent The forms $\alpha_t, \beta_t, \gamma_t$ are canonically associated with the complex structure of $X_t$, which makes the above linear injection canonical.  \hfill $\Box$

\vspace{2ex}

 Since $F'_{{\cal G}}{\cal H} = {\cal H}^{2,\,0}(B)$ is a holomorphic subbundle of ${\cal H}^3(X)$, we are reduced to proving the following

\begin{Lem}\label{Lem:F_cal-G_H4_hol} The holomorphic structure of the vector bundle $F_{{\cal G}}{\cal H}^4:={\cal H}^{2,\,0}(B)\oplus\widetilde{{\cal H}^{2,\,2}_\omega}$ is the restriction of the holomorphic structure of the constant bundle ${\cal H}^3\oplus{\cal H}^4$.

\end{Lem}

\noindent {\it Proof.} We have to show that for any $C^\infty$ section $s$ of $\widetilde{{\cal H}^{2,\,2}_\omega}$, the a priori ${\cal H}^3(X)\oplus{\cal H}^4(X)$-valued $(0,\,1)$-form $D''s$ is actually $F_{{\cal G}}{\cal H}^4$-valued, where $D''$ is the canonical $(0,\,1)$-connection of the constant bundle ${\cal H}^3(X)\oplus{\cal H}^4(X)$. Thanks to (\ref{eqn:H22_A_t}), it suffices to prove that all the anti-holomorphic first-order derivatives of each of the classes $[\alpha_t\wedge\gamma_t\wedge\bar\alpha_t\wedge\bar\gamma_t]_A,\,[\alpha_t\wedge\gamma_t\wedge\bar\beta_t\wedge\bar\gamma_t]_A,\, [\beta_t\wedge\gamma_t\wedge\bar\alpha_t\wedge\bar\gamma_t]_A,\, [\beta_t\wedge\gamma_t\wedge\bar\beta_t\wedge\bar\gamma_t]_A$ lie in $F_{{\cal G}}{\cal H}^4$.  

 We now study these classes individually. By way of example, we compute derivatives at $t=0$.

From (\ref{eqn:forms_t-forms_0}), we infer that the only terms in

$$\alpha_t\wedge\gamma_t\wedge\bar\alpha_t\wedge\bar\gamma_t = (\alpha + t_{11}\,\bar\alpha + t_{12}\,\bar\beta)\wedge(\gamma + t_{31}\,\bar\alpha + t_{32}\,\bar\beta - D(t)\,\bar\gamma)\wedge(\bar\alpha + \bar{t}_{11}\,\alpha + \bar{t}_{12}\,\beta)\wedge(\bar\gamma + \bar{t}_{31}\,\alpha + \bar{t}_{32}\,\beta - \overline{D(t)}\,\gamma)$$

\noindent that are linear in the $\bar{t}_{i\lambda}$'s are

$$\bar{t}_{12}\,\alpha\wedge\gamma\wedge\beta\wedge\bar\gamma  \hspace{3ex} \mbox{and} \hspace{3ex} \bar{t}_{32}\,\alpha\wedge\gamma\wedge\bar\alpha\wedge\beta.$$

\noindent So, the non-trivial anti-holomorphic first-order derivatives at $t=0$ are

\begin{equation}\label{eqn:1st_deriv_G}\frac{\partial(\alpha_t\wedge\gamma_t\wedge\bar\alpha_t\wedge\bar\gamma_t)}{\partial\bar{t}_{12}}_{|t=0} = -\alpha\wedge\beta\wedge\gamma\wedge\bar\gamma \hspace{3ex} \mbox{and} \hspace{3ex} \frac{\partial(\alpha_t\wedge\gamma_t\wedge\bar\alpha_t\wedge\bar\gamma_t)}{\partial\bar{t}_{32}}_{|t=0} = \alpha\wedge\beta\wedge\gamma\wedge\bar\alpha.\end{equation}

\noindent Note that $\alpha\wedge\beta\wedge\gamma\wedge\bar\gamma$ is not $d$-closed, so it defines no class in $H^{3,\,1}_{BC}(X_0,\,\C)$. However, $\alpha\wedge\beta\wedge\gamma\wedge\bar\gamma$ is the image under the multiplication by $\gamma\wedge\bar\gamma$ of $\alpha\wedge\beta$ whose Dolbeault cohomology class $[\alpha\wedge\beta]$ is the (unique up to a multiplicative constant) generator of $H^{2,\,0}(B_0,\,\C)$. Meanwhile, $\alpha\wedge\beta\wedge\gamma\wedge\bar\alpha$ is $d$-closed and its Bott-Chern cohomology class is one of the generators of $H^{3,\,1}_{BC}(X_0,\,\C)$ (cf. proof of Lemma \ref{Lem:H31_BC_injects_H22_A}) which injects canonically into $H^{2,\,2}_A(X_0,\,\C)$ by Lemma \ref{Lem:H31_BC_injects_H22_A}. Under this injection, $[\alpha\wedge\beta\wedge\gamma\wedge\bar\alpha]_{BC}$ identifies with its image $[\alpha\wedge\bar\alpha\wedge\gamma\wedge\bar\gamma]_A$ in $H^{2,\,2}_A(X_0,\,\C)$, which in turn identifies with its image in $\widetilde{H^{2,\,2}_{\omega_0}}=Q_{\omega_0}(H^{2,\,2}_A(X_0,\,\C))$ under the canonical injection $Q_{\omega_0} = I_0:H^{2,\,2}_A(X_0,\,\C)\hookrightarrow H^4(X,\,\C)$ of Lemma \ref{Lem:canonical-injections}.

 The upshot is that after all these identifications, we have

 $$\frac{\partial\,[\alpha_t\wedge\gamma_t\wedge\bar\alpha_t\wedge\bar\gamma_t]_A}{\partial\bar{t}_{i\lambda}}_{|t=0}\in (F_{{\cal G}}{\cal H}^4)_0 = H^{2,\,0}(B_0,\,\C)\oplus\widetilde{H^{2,\,2}_{\omega_0}}$$ 

\noindent for all indices $i, \lambda$. 

Similarly, for the remaining 3 generators of $H^{2,\,2}_A(X_t,\,\C)$, we get from (\ref{eqn:forms_t_explicit}) that the only terms linear in the $\bar{t}_{i\lambda}$'s in $\alpha_t\wedge\gamma_t\wedge\bar\beta_t\wedge\bar\gamma_t$ are $\bar{t}_{22}\,\alpha\wedge\gamma\wedge\beta\wedge\bar\gamma$ and $\bar{t}_{32}\,\alpha\wedge\gamma\wedge\bar\beta\wedge\beta$; in $\beta_t\wedge\gamma_t\wedge\bar\alpha_t\wedge\bar\gamma_t$ are $\bar{t}_{11}\,\beta\wedge\gamma\wedge\alpha\wedge\bar\gamma$ and $\bar{t}_{31}\,\beta\wedge\gamma\wedge\bar\alpha\wedge\alpha$; and in $\beta_t\wedge\gamma_t\wedge\bar\beta_t\wedge\bar\gamma_t$ are $\bar{t}_{21}\,\beta\wedge\gamma\wedge\alpha\wedge\bar\gamma$ and $\bar{t}_{31}\,\beta\wedge\gamma\wedge\bar\beta\wedge\alpha$. Thus, the only non-zero anti-holomorphic first-order derivatives at $t=0$ of these terms are 

$$\pm\,\alpha\wedge\beta\wedge\gamma\wedge\bar\gamma, \hspace{3ex} \pm\,\alpha\wedge\beta\wedge\gamma\wedge\bar\alpha \hspace{3ex} \mbox{and} \hspace{3ex} \pm\,\alpha\wedge\beta\wedge\gamma\wedge\bar\beta.$$

\noindent Note that the only new quantity compared to (\ref{eqn:1st_deriv_G}) is $\alpha\wedge\beta\wedge\gamma\wedge\bar\beta$. It has the same properties as $\alpha\wedge\beta\wedge\gamma\wedge\bar\alpha$, i.e. it is $d$-closed and its Bott-Chern cohomology class is a generator of $H^{3,\,1}_{BC}(X_0,\,\C)$ (cf. proof of Lemma \ref{Lem:H31_BC_injects_H22_A}). This vector space injects canonically into $H^{2,\,2}_A(X_0,\,\C)$ by Lemma \ref{Lem:H31_BC_injects_H22_A}. So the above argument applies again and yields

$$\frac{\partial\,[\alpha_t\wedge\gamma_t\wedge\bar\beta_t\wedge\bar\gamma_t]_A}{\partial\bar{t}_{i\lambda}}_{|t=0},\,  \frac{\partial\,[\beta_t\wedge\gamma_t\wedge\bar\alpha_t\wedge\bar\gamma_t]_A}{\partial\bar{t}_{i\lambda}}_{|t=0},\,  \frac{\partial\,[\beta_t\wedge\gamma_t\wedge\bar\beta_t\wedge\bar\gamma_t]_A}{\partial\bar{t}_{i\lambda}}_{|t=0}\in (F_{{\cal G}}{\cal H}^4)_0$$ 

\noindent for all indices $i, \lambda$.

 \hfill $\Box$

\subsection{Construction of coordinates on the Gauduchon cone}\label{subsection:coordinates_Gcone}

 Recall the isomorphisms

\vspace{2ex}

$\begin{CD} 
H^{1,\,1}(B_0,\,\C) & = & \bigg\langle[\alpha\wedge\bar\alpha]_{\bar\partial},\, [\alpha\wedge\bar\beta]_{\bar\partial},\, [\beta\wedge\bar\alpha]_{\bar\partial},\, [\beta\wedge\bar\beta]_{\bar\partial}\bigg\rangle\\
@V\simeq V \cdot\wedge\gamma V \\         
H^{2,\,1}_{[\gamma]}(X_0,\,\C)  & = &  \bigg\langle[\alpha\wedge\gamma\wedge\bar\alpha]_{\bar\partial},\, [\alpha\wedge\gamma\wedge\bar\beta]_{\bar\partial},\, [\beta\wedge\gamma\wedge\bar\alpha]_{\bar\partial},\, [\beta\wedge\gamma\wedge\bar\beta]_{\bar\partial}\bigg\rangle\\
@V\simeq V \cdot\wedge\bar\gamma V  \\
 H^{2,\,2}_A(X_0,\,\C) & = & \bigg\langle[\alpha\wedge\gamma\wedge\bar\alpha\wedge\bar\gamma]_A,\, [\alpha\wedge\gamma\wedge\bar\beta\wedge\bar\gamma]_A,\, [\beta\wedge\gamma\wedge\bar\alpha\wedge\bar\gamma]_A,\, [\beta\wedge\gamma\wedge\bar\beta\wedge\bar\gamma]_A\bigg\rangle.
\end{CD}$

\vspace{2ex}

 On the other hand, on the vector space
$$H^{3,\,0}(X_t,\,\C)\oplus H^{2,\,1}_{[\gamma]}(X_t,\,\C)\simeq H^{2,\,0}(B_t,\,\C)\oplus H^{1,\,1}(B_t,\,\C), \hspace{3ex} t\in\Delta,$$
\noindent we have two sesquilinear intersection forms (the first of which was considered in (\ref{eqn:H-def})). The first one is obtained by restriction from $H^3_{DR}(X,\,\C)\times H^3_{DR}(X,\,\C)$ (where $X$ is the differentiable manifold underlying the $X_t$'s) when $H^{3,\,0}(X_t,\,\C)\oplus H^{2,\,1}_{[\gamma]}(X_t,\,\C)$ is viewed as a vector subspace of $H^3_{DR}(X,\,\C)$:
\begin{eqnarray}\label{eqn:H_recall}\nonumber H: \bigg(H^{3,\,0}(X_t,\,\C)\oplus H^{2,\,1}_{[\gamma]}(X_t,\,\C)\bigg) & \times & \bigg(H^{3,\,0}(X_t,\,\C)\oplus H^{2,\,1}_{[\gamma]}(X_t,\,\C)\bigg)\longrightarrow \C, \\
\nonumber (\{u\},\,\{v\}) & \mapsto & -i\,\int\limits_X u\wedge\bar{v}.\end{eqnarray}  
\noindent Its signature is $(-,\,-,\,+,\,+,\,+)$ (cf. Corollary \ref{Cor:H_signature}).

\noindent The second sesquilinear intersection form is obtained by restriction from $H^2_{DR}(B,\,\C)$ (where $B$ is the differentiable manifold underlying the tori $B_t$) when $H^{2,\,0}(B_t,\,\C)\oplus H^{1,\,1}(B_t,\,\C)$ is viewed as a vector subspace of $H^2_{DR}(B,\,\C)$:
\begin{eqnarray}\label{eqn:H_B_def}\nonumber H_B: \bigg(H^{2,\,0}(B_t,\,\C)\oplus H^{1,\,1}(B_t,\,\C)\bigg) & \times & \bigg(H^{2,\,0}(B_t,\,\C)\oplus H^{1,\,1}(B_t,\,\C)\bigg)\longrightarrow \C, \\
    (\{\xi\},\,\{\zeta\}) & \mapsto & \int\limits_B \xi\wedge\bar\zeta.\end{eqnarray}
\noindent Indeed, the coefficient of the integral $\int_B \xi\wedge\bar\zeta$ in the defintion of the sesquilinear intersection form in degree $n$ on an $n$-dimensional compact complex manifold is $(-1)^{\frac{n(n+1)}{2}}\,i^n$, so in the case of $H_B$, where $n=\mbox{dim}_\C B_t=2$, this coefficient equals $1$.

 In particular, on the vector space

\begin{equation}\label{eqn:H20_H11_B_0_isom} H^{2,\,0}(B_0,\,\C)\oplus H^{1,\,1}(B_0,\,\C)\stackrel{\cdot\wedge\gamma}{\simeq} H^{3,\,0}(X_0,\,\C)\oplus H^{2,\,1}_{[\gamma]}(X_0,\,\C),\end{equation}

\noindent the two sesquilinear intersection forms are given by
$$H_B(\{u\},\,\{v\}) = \int\limits_B u\wedge\bar{v} \hspace{3ex} \mbox{and} \hspace{3ex} H(\{u\wedge\gamma\},\,\{v\wedge\gamma\}) = -\int\limits_B(u\wedge\bar{v})\wedge(i\gamma\wedge\bar\gamma).$$

\begin{Prop}\label{Prop:H_B_signature} The signature of the sesquilinear intersection form $H_B$ defined in (\ref{eqn:H_B_def}) is 

\noindent $(+,\,+,\,-,\,-,\,-)$.

 Specifically, for any Hermitian metric $\rho_t$ on $B_t$, $H^{2,\,0}(B_t,\,\C)\subset H^2_{+}(B_t,\,\C)$, while $H_B$ has signature $(+,\,-,\,-,\,-)$ on $H^{1,\,1}(B_t,\,\C)$.

\end{Prop}

\noindent {\it Proof.} Every class in $H^{2,\,0}(B_t,\,\C)$ has a unique representative which, for bidegree reasons, is a {\it primitive} $(2,\,0)$-form w.r.t. any Hermitian metric we equip $B_t$ with. On the other hand, a well-known formula (cf. e.g. [Voi02, Proposition 6.29, p. 150]) asserts that for any {\it primitive} $(p,\,q)$-form $v$ w.r.t. any Hermitian metric $\omega$ on a complex manifold of dimension $n$, we have

\begin{equation}\label{eqn:star_primitive_formula}\nonumber\star\,v = (-1)^{k(k+1)/2}\,i^{p-q}\,\frac{\omega^{n-p-q}\wedge v}{(n-p-q)!}, \hspace{3ex} \mbox{where} \hspace{1ex} k:=p+q.\end{equation}

\noindent When $p+q=n$ and $(p,\,q)=(2,\,0)$, we get $\star\,v = v$. Therefore, $H^{2,\,0}(B_t,\,\C)\subset H^2_{+}(B_t,\,\C)$.

 Now, recall that $H^{1,\,1}(B_t,\,\C)$ is generated by the classes $[\alpha_t\wedge\bar\alpha_t]_{\bar\partial},\, [\alpha_t\wedge\bar\beta_t]_{\bar\partial},\, [\beta_t\wedge\bar\alpha_t]_{\bar\partial},\, [\beta_t\wedge\bar\beta_t]_{\bar\partial}$. Let us equip $B_t$ with the Hermitian metric

$$\rho_t:=i\alpha_t\wedge\bar\alpha_t + i\beta_t\wedge\bar\beta_t, \hspace{3ex} t\in\Delta.$$

\noindent The associated volume form is $dV_{\rho_t}=\rho_t^2/2! = i\alpha_t\wedge\bar\alpha_t\wedge i\beta_t\wedge\bar\beta_t$. Denoting by $\star = \star_{\rho_t}$ the Hodge star operator induced by $\rho_t$, we can check as in Lemma \ref{Lem:star_generators_21_gamma} that the following identities hold

\begin{eqnarray}\label{eqn:star_generators_H11_B}\nonumber & & \star\,(i\alpha_t\wedge\bar\alpha_t) = i\beta_t\wedge\bar\beta_t,  \hspace{6ex} \star\,(i\beta_t\wedge\bar\beta_t) = i\alpha_t\wedge\bar\alpha_t, \\
& & \star\,(i\alpha_t\wedge\bar\beta_t) = -i\alpha_t\wedge\bar\beta_t,  \hspace{4.5ex} \star\,(i\beta_t\wedge\bar\alpha_t) = -i\beta_t\wedge\bar\alpha_t  \end{eqnarray}

\noindent for every $t\in\Delta$. 

 Indeed, from the definition of the Hodge star operator, we know that

$$u\wedge\overline{\star\,(i\alpha_t\wedge\bar\alpha_t)} = \langle u,\, i\alpha_t\wedge\bar\alpha_t\rangle\,dV_{\rho_t}.$$

\noindent When $u$ is the product of a form chosen from $\alpha_t$, $\beta_t$ and a form chosen from $\bar\alpha_t$, $\bar\beta_t$, the two sides of this identity are non-zero only when $u=i\alpha_t\wedge\bar\alpha_t$. In this case, we get

$$(i\alpha_t\wedge\bar\alpha_t)\wedge\overline{\star\,(i\alpha_t\wedge\bar\alpha_t)} = i\alpha_t\wedge\bar\alpha_t\wedge i\beta_t\wedge\bar\beta_t,$$

\noindent so $\overline{\star\,(i\alpha_t\wedge\bar\alpha_t)}$ must be the form complementary to $i\alpha_t\wedge\bar\alpha_t$. We get $\star\,(i\alpha_t\wedge\bar\alpha_t) = i\beta_t\wedge\bar\beta_t$. The remaining identities in (\ref{eqn:star_generators_H11_B}) are proved in an analogous way.

 The last two identities in (\ref{eqn:star_generators_H11_B}) show that $i\alpha_t\wedge\bar\beta_t$ and $i\beta_t\wedge\bar\alpha_t$ are eigenvectors of $\star$ corresponding to the eigenvalue $-1$, so they represent classes lying in $H^2_{-}(B_t,\,\C)$. Meanwhile, the first two identities in (\ref{eqn:star_generators_H11_B}) can be re-written as

$$\star\,(i\alpha_t\wedge\bar\alpha_t + i\beta_t\wedge\bar\beta_t) = i\alpha_t\wedge\bar\alpha_t + i\beta_t\wedge\bar\beta_t \hspace{2ex}  \mbox{and} \hspace{2ex} \star\,(i\alpha_t\wedge\bar\alpha_t - i\beta_t\wedge\bar\beta_t) = -(i\alpha_t\wedge\bar\alpha_t - i\beta_t\wedge\bar\beta_t).$$

\noindent Therefore, $i\alpha_t\wedge\bar\alpha_t + i\beta_t\wedge\bar\beta_t$ represents a class lying in $H^2_{+}(B_t,\,\C)$ and $i\alpha_t\wedge\bar\alpha_t - i\beta_t\wedge\bar\beta_t$ represents a class lying in $H^2_{-}(B_t,\,\C)$.  \hfill  $\Box$

\vspace{2ex}

A consequence of these considerations is that Proposition \ref{Prop:coordinates} can now be used to construct coordinates on the complexification $\widetilde{{\cal G}_0}\subset\widetilde{{\cal G}_{X_0}}$ of the parameter set ${\cal G}_0 = \{[(\omega_t^{1,\,1})^2]_A\,|\, t\in\Delta_{[\gamma]}\}\subset{\cal G}_{X_0}$ using the symplectic vector space $(H^2(B,\,\C),\,Q_B(\cdot\,,\,\cdot))$ equipped with the bilinear intersection form $Q_B:H^2(B,\,\C)\times H^2(B,\,\C)\to\C$ defined by $Q_B(\{u\},\,\{v\}):=-\int_B u\wedge v$. Consider the following

\vspace{3ex}

\noindent {\bf Setup.} {\it Let $(X_t)_{t\in\Delta}$ be the Kuranishi family of the Iwasawa manifold $X=X_0$ and let $(B_t)_{t\in\Delta}$ be the associated family of $2$-dimensional Albanese tori. Let $v=(v_t)_{t\in\Delta_{[\gamma]}}$ be a holomorphic section of the vector bundle $\Delta_{[\gamma]}\ni t\mapsto H^{2,\,0}(B_t,\,\C)$ such that each $(2,\,0)$-form $v_t$ is non-vanishing on $B_t$. (We may choose $v_t:=\alpha_t\wedge\beta_t$.)

 Let $\eta_0 = \eta_0^{3,\,0} + \eta_0^{2,\,1} + \overline{\eta_0^{2,\,1}} + \overline{\eta_0^{3,\,0}}\in H^3(X,\,\R)$ be a real class with $\eta_0^{3,\,0}\in H^{3,\,0}(X_0,\,\C), \eta_0^{2,\,1}\in H^{2,\,1}_{[\gamma]}(X_0,\,\C)$ satisfying conditions (\ref{eqn:eta_0_assumption}) of Proposition \ref{Prop:coordinates}. Thanks to isomorphism (\ref{eqn:H20_H11_B_0_isom}), there exist unique classes

$$\eta_{0,\,B}^{2,\,0}\in H^{2,\,0}(B_0,\,\C) \hspace{3ex} \mbox{and} \hspace{3ex} \eta_{0,\,B}^{1,\,1}\in H^{1,\,1}(B_0,\,\C)$$

\noindent such that $\eta_0^{3,\,0} = \eta_{0,\,B}^{2,\,0}\wedge\gamma$ and $\eta_0^{2,\,1} = \eta_{0,\,B}^{1,\,1}\wedge\gamma$. Put

$$\eta_{0,\,B}:= \eta_{0,\,B}^{2,\,0} + \frac{\eta_{0,\,B}^{1,\,1} + \overline{\eta_{0,\,B}^{1,\,1}}}{2} + \overline{\eta_{0,\,B}^{2,\,0}}\in H^2(B_0,\,\R).$$

 Complete $\eta_{0,\,B}$ to a symplectic basis $\{\eta_{0,\,B},\eta_{1,\,B},\dots , \eta_{4,\,B},\,\nu_{0,\,B}, \nu_{1,\,B},\dots , \nu_{4,\,B}\}$ of $(H^2(B_0,\,\R),\,Q_B(\cdot\,,\,\cdot))$. Normalise such that

 $$Q_B(v_t,\,\eta_{0,\,B}) = 1 \hspace{2ex} \mbox{for all}\hspace{1ex} t\in\Delta \hspace{2ex} \mbox{sufficiently close to}\,\,\, 0.$$}

We can now state the result we have been aiming at.

\begin{Prop}\label{Prop:coordinates_A} In the setup described above, the functions

\begin{equation}\label{eqn:coordinates_A_def}\nonumber w_i(t):= Q_B(v_t,\,\eta_{i,\,B}) \hspace{2ex} \mbox{for}\hspace{1ex} t\in\Delta_{[\gamma]} \hspace{2ex} 
\mbox{and}\hspace{1ex}  i\in\{1,\dots , 4\}\end{equation}
\noindent define {\bf holomorphic coordinates} on $\Delta_{[\gamma]}$ in a neighbourhood of\, $0$ and implicitly on the complexified parameter set $\widetilde{{\cal G}_0}$, the complexification of

$${\cal G}_0 = \bigg\{[(\omega_t^{1,\,1})^2]_A\,\,\bigg |\,\, t\in\Delta_{[\gamma]}\bigg\}\subset{\cal G}_{X_0},$$

\noindent in a neighbourhood of $[\omega_0^2]_A$.

\end{Prop}

\noindent {\it Proof.} It runs along the lines of the proof of Proposition \ref{Prop:coordinates}.  \hfill $\Box$

\section{The mirror map}\label{section:mirror-map}

We can now associate with every small deformation $X_t$ of $X_0$ an element in the Gauduchon cone of $X_0$ in which the canonical class $[\omega_0^2]_A = [(\omega_0^{1,\,1})^2]_A$ is a marked point.
\begin{Def}\label{Def:mirror-map_def} Let $(X_t)_{t\in\Delta}$ be the Kuranishi family of the Iwasawa manifold $X=X_0$ and let $(\omega_t^{1,\,1})_{t\in\Delta_{[\gamma]}}$ be the smooth family of canonical Gauduchon metrics on $X_0$ constructed in Proposition \ref{Prop:omega_t_11}. For every $t\in\Delta_{[\gamma]}$, let $[(\omega_t^{1,\,1})^2]_A\in{\cal G}_{X_0} = {\cal G}_X$ be the associated Aeppli cohomology class.

We define the {\bf positive mirror map} of $X = X_0$ by
\begin{equation}\label{eqn:positive_mirror-map_def1} {\cal M} : \Delta_{[\gamma]}\longrightarrow{\cal G}_X, \hspace{3ex} t\mapsto [(\omega_t^{1,\,1})^2]_A,\end{equation}
\noindent where ${\cal G}_X$ is the Gauduchon cone of $X = X_0$ (i.e. the open subset of $H^{2,\,2}_A(X,\,\R)$ consisting of real {\bf positive} classes). Thus, the parameter subset of the Gauduchon cone of $X$ defined in (\ref{eqn:parameter-set_def}) is ${\cal G}_0 = {\cal M}(\Delta_{[\gamma]})$.

\end{Def}

\vspace{3ex}

From (\ref{eqn:c_j-d_t}) and from Lemma \ref{Lem:Aeppli_omega_t_11^2} we get the following formula for the positive mirror map after recalling that $t_{3,\,1} = t_{3,\,2} =0$ when $t\in\Delta_{[\gamma]}$:

\begin{eqnarray}\label{eqn:positive_mirror-map_formula}
\nonumber {\cal M}(t) & = &  2\,\bigg(1-|t_{11}|^2 - |t_{21}|^2\bigg)\,\bigg(1-|t_{11}\,t_{22} - t_{12}\,t_{21}|^2\bigg)\,\bigg[i\,\alpha\wedge\bar\alpha\wedge i\,\gamma\wedge\bar\gamma\bigg]_A \\
\nonumber & + & 2\,\bigg(1-|t_{12}|^2 - |t_{22}|^2\bigg)\,\bigg(1-|t_{11}\,t_{22} - t_{12}\,t_{21}|^2\bigg)\,\bigg[i\,\beta\wedge\bar\beta\wedge i\,\gamma\wedge\bar\gamma\bigg]_A \end{eqnarray}
\begin{eqnarray}\nonumber  & - & 2\,\bigg(t_{12}\,\bar{t}_{11} + t_{22}\,\bar{t}_{21}\bigg)\,\bigg(1-|t_{11}\,t_{22} - t_{12}\,t_{21}|^2\bigg)\,\bigg[i\,\alpha\wedge\bar\beta\wedge i\,\gamma\wedge\bar\gamma\bigg]_A \\
     & - & 2\,\bigg(t_{11}\,\bar{t}_{12} + t_{21}\,\bar{t}_{22}\bigg)\,\bigg(1-|t_{11}\,t_{22} - t_{12}\,t_{21}|^2\bigg)\,\bigg[i\,\beta\wedge\bar\alpha\wedge i\,\gamma\wedge\bar\gamma\bigg]_A,     \hspace{3ex} t\in\Delta_{[\gamma]}.\end{eqnarray}

\noindent Alternatively, formula (\ref{eqn:omega_t_11_J0_square}) yields for every $t\in\Delta_{[\gamma]}$

\begin{eqnarray}\label{eqn:positive_mirror-map_formula_bis}\nonumber {\cal M}(t) & = & [\omega_0^2]_A  \\
\nonumber & + & 2\,\bigg(c_1(t) + c_3(t) + c_1(t)\,c_3(t)\bigg)\,\bigg[i\,\alpha\wedge\bar\alpha\wedge i\,\gamma\wedge\bar\gamma\bigg]_A \\
\nonumber & + & 2\,\bigg(c_2(t) +c_3(t) + c_2(t)\,c_3(t)\bigg)\,\bigg[i\,\beta\wedge\bar\beta\wedge i\,\gamma\wedge\bar\gamma\bigg]_A \\
     & + & 2\,d(t)\,\bigg(1 + c_3(t)\bigg)\,\bigg[i\,\alpha\wedge\bar\beta\wedge i\,\gamma\wedge\bar\gamma\bigg]_A + 2\,\overline{d(t)}\,\bigg(1 + c_3(t)\bigg)\,\bigg[i\,\beta\wedge\bar\alpha\wedge i\,\gamma\wedge\bar\gamma\bigg]_A,  \end{eqnarray} 

\noindent where $c_j(t)$ and $d(t)$ are defined by (\ref{eqn:c_j-d_t}) with $t_{3,\,1} = t_{3,\,2} =0$ when $t\in\Delta_{[\gamma]}$.


Since $\Delta_{[\gamma]}$ is an open subset in a vector space of complex dimension $4$ (see Definition \ref{Def:Delta_gamma}) while ${\cal G}_X$ is an open subset in a vector space of real dimension $4$, we rebalance the two sides of (\ref{eqn:positive_mirror-map_def1}) by complexifying the latter set.

\begin{Def}\label{Def:G-cone-complexified} Let $X = X_0$ be the Iwasawa manifold.

\vspace{1ex}

$(i)$\, We know from (\ref{eqn:H22_A_t}) that $H^{2,\,2}_A(X_0,\,\C)$ injects canonically (and $\C$-linearly) into $H^4_{DR}(X,\,\C)$. 
Similarly, $H^{2,\,2}_A(X_0,\,\R)$ injects canonically (and $\R$-linearly) into $H^4_{DR}(X,\,\R)$. 
On the other hand, we know that the image of $H^4(X,\,\Z)$ in $H^4_{DR}(X,\,\R)$ under the natural map $H^4(X,\,\Z)\hookrightarrow H^4_{DR}(X,\,\R)$ is a {\bf lattice}. 
We put

\begin{equation*}\label{eqn:H22_Z_def}H^{2,\,2}_A(X_0,\,\Z):= H^{2,\,2}_A(X_0,\,\R)\cap H^4(X,\,\Z)\subset H^{2,\,2}_A(X_0,\,\R).\end{equation*}

\noindent Thus $H^{2,\,2}_A(X_0,\,\Z)$ is a {\bf lattice} in $H^{2,\,2}_A(X_0,\,\R)$.

\vspace{2ex}

$(ii)$\, We define the {\bf complexified Gauduchon cone} of the Iwasawa manifold $X=X_0$ by

\begin{equation*}\label{eqn:G-cone-complexified}\widetilde{\cal G}_{X_0}:= {\cal G}_{X_0}\oplus H^{2,\,2}_A(X_0,\,\R)\slash 2\pi i\,H^{2,\,2}_A(X_0,\,\Z).\end{equation*}

\vspace{2ex}

$(iii)$\, We define the {\bf mirror map} $\widetilde{\cal M}:\Delta_{[\gamma]}\longrightarrow\widetilde{\cal G}_{X_0}$ of $X=X_0$ by

\begin{eqnarray*}\label{mirror-map_def1}\nonumber\widetilde{\cal M}(t) & = & [\omega_0^2]_A - t_{11}\,[i\,\beta\wedge\bar\alpha\wedge i\,\gamma\wedge\bar\gamma]_A 
+ t_{22}\,[i\,\alpha\wedge\bar\beta\wedge i\,\gamma\wedge\bar\gamma]_A \\
& - & t_{12} \,[i\,\beta\wedge\bar\beta\wedge i\,\gamma\wedge\bar\gamma]_A + t_{21}\,[i\,\alpha\wedge\bar\alpha\wedge i\,\gamma\wedge\bar\gamma]_A.\end{eqnarray*}

\noindent Thus, the positive mirror map ${\cal M}$ is a kind of ``squared absolute value'' of $\widetilde{\cal M}$.

\vspace{2ex}

$(iv)$\, We define the {\bf complexified parameter set} by $\widetilde{{\cal G}_0}:= \widetilde{\cal M}(\Delta_{[\gamma]})$. It contains the marked point $[\omega_0^2]_A$ of the Gauduchon cone ${\cal G}_X$. 

\end{Def}

 Thus, if the radius of $\Delta_{[\gamma]}$ as an open ball about the origin in $H^{0,\,1}(X,\,T^{1,\,0}_X)$ is small enough, $\widetilde{\cal M}$ defines a biholomorphism between $\Delta_{[\gamma]}$ and the open subset $\widetilde{{\cal G}_0}\subset\widetilde{\cal G}_X\subset H^{2,\,2}_A(X_0,\,\C)$.

\vspace{3ex}

Our discussion can be summed up as follows.

\begin{The}\label{The:mirror-map} The mirror map $\widetilde{\cal M} : \Delta_{[\gamma]}\longrightarrow\widetilde{\cal G}_X$ of the Iwasawa manifold $X=X_0$ 
enjoys the following properties.

\vspace{1ex}

$(i)$\, $\widetilde{\cal M}$ is holomorphic and defines a biholomorphism onto its image if the radius of $\Delta_{[\gamma]}$ as an open ball in $H^{0,\,1}(X_0,\,T^{1,\,0}X_0)$ is small enough;

\vspace{1ex}

$(ii)$\, $\widetilde{\cal M}(0) = [\omega_0^2]_A\in{\cal G}_X$, where $\omega_0$ is the Gauduchon metric on $X$ canonically induced by the complex parallelisable structure of $X$ 
(cf. (\ref{eqn:omega_t_def}));

\vspace{1ex}

$(iii)$\, The composition of the canonical isomorphism $A_t$ observed in (\ref{eqn:H21_gamma-H22_isomorphism_t}) with $B_t$ defined in (\ref{eqn:H22_A_t-H22_A_0}) and with the Kodaira-Spencer and the Calabi-Yau isomorphisms is the following canonical isomorphism \begin{equation*}\label{eqn:dM_t}T^{1,\,0}_t\Delta_{[\gamma]}\simeq H^{2,\,1}_{[\gamma]}(X_t,\,\C) \substack{\simeq \\ 
\longrightarrow \\ A_t}H^{2,\,2}_A(X_t,\,\C)  \substack{\simeq \\ \longrightarrow \\ B_t}H^{2,\,2}_A(X_0,\,\C) = T^{1,\,0}_{\widetilde{\cal M}(t)}\widetilde{{\cal G}_X}, 
\hspace{3ex} [\Gamma]_{\bar\partial}\mapsto[\Gamma\wedge\bar\gamma_t]_A = A_t([\Gamma]_{\bar\partial}), \end{equation*}

\noindent that coincides at $t=0$ with the differential map $d\widetilde{\cal M}_0$ of $\widetilde{\cal M}$ and depends anti-holomorphically on $t$;

\vspace{1ex}

$(iv)$\, On the {\bf metric side of the mirror}, there is a variation of Hodge structures (VHS)

$${\cal H}^3\oplus{\cal H}^4\supset F_{{\cal G}}{\cal H}^4={\cal H}^{2,\,0}(B)\oplus\widetilde{{\cal H}^{2,\,2}_\omega} \supset F'_{{\cal G}}{\cal H}^4 = {\cal H}^{2,\,0}(B)$$

\noindent parametrised by $\widetilde{{\cal G}_0} = \widetilde{{\cal M}}(\Delta_{[\gamma]})\simeq\Delta_{[\gamma]}$ whose $4$-dimensional fibre over any point $\widetilde{{\cal M}}(t)\in\widetilde{{\cal G}_0}$ is the vector subspace $\widetilde{H^{2,\,2}_{\omega_t}}:=Q_{\omega_t}(H^{2,\,2}_A(X_t,\,\C))\subset H^4(X,\,\C)$ defined in Conclusion \ref{Conc:VHS_two-families-metrics}. Moreover, there exists a $C^\infty$ isomorphism of VHS between this VHS and the VHS

$${\cal H}^3\supset F^2{\cal H}^3_{[\gamma]}\supset F^3{\cal H}^3$$

\noindent parametrised by $\Delta_{[\gamma]}$ and defined on the {\bf complex-structure side of the mirror} in Theorem \ref{The:VHS_3_Delta}. 

This isomorphism is {\bf holomorphic} between the $1$-dimensional parts ${\cal H}^{2,\,0}(B)$, resp. $F^3{\cal H}^3$ (it is the multiplication by $\gamma_t$), while the isomorphism between the rank-$4$ vector bundles ${\cal H}^{2,\,1}_{[\gamma]}$ and ${\cal H}^{2,\,2}_A$ (defining, up to identifications, the $4$-dimensional parts of these VHS's) is {\bf anti-holomorphic} (given by the $A_t$'s, the multiplication by $\bar\gamma_t$).

 Moreover, each of the two Hodge filtrations $F^2{\cal H}^3_{[\gamma]}\supset F^3{\cal H}^3$ and $F_{{\cal G}}{\cal H}^4\supset F'_{{\cal G}}{\cal H}^4$ is $C^\infty$  isomorphic to the Hodge filtration $F^1{\cal H}^2(B)\supset F^2{\cal H}^2(B)$ associated with the family $(B_t)_{t\in\Delta_{[\gamma]}}$ of Albanese tori of the small essential deformations $(X_t)_{t\in\Delta_{[\gamma]}}$ of the Iwasawa manifold $X=X_0$.

\vspace{1ex}

$(v)$\, There is a bijection 

\begin{equation}\label{eqn:bijection_coordinates}\bigg(z_1(t),\,z_2(t),\,z_3(t),\,z_4(t)\bigg) \mapsto \bigg(w_1(t),\,w_2(t),\,w_3(t),\,w_4(t) \bigg),  \hspace{3ex} t\in\Delta_{[\gamma]}\end{equation}

\noindent depending holomorphically on $t$ between the holomorphic coordinates defined in Proposition \ref{Prop:coordinates} on $\Delta_{[\gamma]}$ in a neighbourhood of $0$ and the holomorphic coordinates defined in Proposition \ref{Prop:coordinates_A} on $\{[(\omega_t^{1,\,1})^2]_A\,\slash\, t\in\Delta_{[\gamma]}\}\subset{\cal G}_{X_0}$ in a neighbourhood of $[\omega_0^2]_A$.

\end{The}

\begin{proof} $(i)$ and $(ii)$ follow from the construction. To prove $(iii)$, we start by recalling that with the notation $\alpha_1:=\alpha$, $\alpha_2:=\beta$, $\xi_1:=\xi_\alpha$, $\xi_2:=\xi_\beta$, $\xi_3:=\xi_\gamma$, 
the space $H^{0,\,1}(X,\,T^{1,\,0}X)$ consists of the objects 
$$\sum\limits_{\substack{i=1,2,3\\ \lambda=1,2}}t_{i\lambda}\,\xi_i\otimes\bar\alpha_{\lambda}$$
\noindent where the $t_{i\lambda}$ define holomorphic coordinates on $\Delta$. Also recall that $t_{31}=t_{32}=0$ on $\Delta_{[\gamma]}$. 
Thus, the holomorphic tangent space to $\Delta_{[\gamma]}$ at $0$ is generated by $\partial/\partial t_{11}$, $\partial/\partial t_{12}$, 
$\partial/\partial t_{21}$, $\partial/\partial t_{22}$ and the images of these vector fields under the composition of the Kodaira-Spencer map $\rho$ 
with the Calabi-Yau isomorphism $T_\Omega$ ($=\cdot\lrcorner(\alpha\wedge\beta\wedge\gamma$))

\begin{eqnarray}\label{eqn:images_CY_isomorphism}\nonumber T^{1,\,0}_0\Delta_{[\gamma]} \underset{\simeq}{\overset{\rho}{\longrightarrow}}
H_{[\gamma]}^{0,\,1}(X,\,T^{1,\,0}X) & \underset{\simeq}{\overset{T_{\Omega}}{\longrightarrow}} & H^{2,\,1}_{[\gamma]}(X,\,\C)\end{eqnarray}
\noindent (cf. (\ref{eqn:CY_isomorphism})) are spelt out as follows
\begin{eqnarray}\label{eqn:vector-fields_images}\nonumber \frac{\partial}{\partial t_{11}} \mapsto [\xi_1\otimes\bar\alpha_1
= \xi_\alpha\otimes\bar\alpha] \mapsto -[\beta\wedge\gamma\wedge\bar\alpha]_{\bar\partial}, \hspace{3ex} \frac{\partial}{\partial t_{12}}
\mapsto [\xi_1\otimes\bar\alpha_2 = \xi_\alpha\otimes\bar\beta] \mapsto -[\beta\wedge\gamma\wedge\bar\beta]_{\bar\partial}, \\
  \frac{\partial}{\partial t_{21}} \mapsto [\xi_2\otimes\bar\alpha_1 = \xi_\beta\otimes\bar\alpha] 
  \mapsto [\alpha\wedge\gamma\wedge\bar\alpha]_{\bar\partial}, \hspace{3ex} \frac{\partial}{\partial t_{22}} \mapsto [\xi_2\otimes\bar\alpha_2 
  = \xi_\beta\otimes\bar\beta] \mapsto [\alpha\wedge\gamma\wedge\bar\beta]_{\bar\partial}.\end{eqnarray}
\noindent We get
\begin{eqnarray}\nonumber d\widetilde{\cal M}\bigg(\frac{\partial}{\partial t_{11}}\bigg) & = & \frac{\partial\widetilde{\cal M}}{\partial t_{11}} 
= -[i\,\beta\wedge\bar\alpha\wedge i\,\gamma\wedge\bar\gamma]_A = A_0(-[\beta\wedge\gamma\wedge\bar\alpha]_{\bar\partial}) \simeq A_0\bigg(\frac{\partial}{\partial t_{11}}\bigg), \\
\nonumber d\widetilde{\cal M}\bigg(\frac{\partial}{\partial t_{22}}\bigg) & = & \frac{\partial\widetilde{\cal M}}{\partial t_{22}}
= [i\,\alpha\wedge\bar\beta\wedge i\,\gamma\wedge\bar\gamma]_A = A_0([\alpha\wedge\gamma\wedge\bar\beta]_{\bar\partial}) \simeq A_0\bigg(\frac{\partial}{\partial t_{22}}\bigg), \\  
\nonumber d\widetilde{\cal M}\bigg(\frac{\partial}{\partial t_{12}}\bigg) & = & \frac{\partial\widetilde{\cal M}}{\partial t_{12}}
= -\,[i\,\beta\wedge\bar\beta\wedge i\,\gamma\wedge\bar\gamma]_A = A_0(-[\beta\wedge\gamma\wedge\bar\beta]_{\bar\partial}) \simeq A_0\bigg(\frac{\partial}{\partial t_{12}}\bigg),\\
\nonumber d\widetilde{\cal M}\bigg(\frac{\partial}{\partial t_{21}}\bigg) & = & \frac{\partial\widetilde{\cal M}}{\partial t_{21}}
= [i\,\alpha\wedge\bar\alpha\wedge i\,\gamma\wedge\bar\gamma]_A = A_0([\alpha\wedge\gamma\wedge\bar\alpha]_{\bar\partial}) 
\simeq A_0\bigg(\frac{\partial}{\partial t_{21}}\bigg),\end{eqnarray}    
\noindent where $\simeq$ stands for the identifications under (\ref{eqn:vector-fields_images}).

We conclude that $d\widetilde{\cal M}_0 = A_0$, so part $(iii)$ is proved at $t=0$.

$(iv)$\, is contained in Theorem \ref{The:VHS_3_Delta}, Corollary \ref{Cor:holomorphic-bundle-isomorphisms} and Conclusion \ref{Conc:VHS_two-families-metrics}.

$(v)$\, is contained in Propositions \ref{Prop:coordinates} and \ref{Prop:coordinates_A}.  

\end{proof}

\section{Appendix}\label{section:appendix}

 We spell out the details of the computations of the first-order anti-holomorphic partial derivatives of the forms $\Gamma_j(t)$ defined in (\ref{eqn:Gamma_j-forms_t}) for $j\in\{1,2,3,4\}$. 

 Recall the following identities proved in (\ref{eqn:forms_t-forms_0}): 

\begin{equation}\label{eqn:forms_t_explicit}\nonumber\alpha_t = \alpha + t_{11}\,\bar\alpha + t_{12}\,\bar\beta, \hspace{2ex} \beta_t = \beta + t_{21}\,\bar\alpha + t_{22}\,\bar\beta, \hspace{2ex} \gamma_t = \gamma + t_{31}\,\bar\alpha + t_{32}\,\bar\beta - D(t)\,\bar\gamma.\end{equation}

\noindent So we get

\begin{eqnarray}\nonumber \Gamma_1(t) & = & (\alpha + t_{11}\,\bar\alpha + t_{12}\,\bar\beta)\wedge(\gamma + t_{31}\,\bar\alpha + t_{32}\,\bar\beta - D(t)\,\bar\gamma)\wedge(\bar\alpha + \bar{t}_{11}\,\alpha + \bar{t}_{12}\,\beta) \\
\nonumber & - & \frac{\sigma_{2\bar{2}}(t)}{\bar\sigma_{12}(t)}\,(\alpha + t_{11}\,\bar\alpha + t_{12}\,\bar\beta)\wedge(\beta + t_{21}\,\bar\alpha + t_{22}\,\bar\beta)\wedge(\bar\gamma + \bar{t}_{31}\,\alpha + \bar{t}_{32}\,\beta - \overline{D(t)}\,\gamma) \\
 \nonumber & = & -\bigg[\alpha\wedge\bar\alpha + \bar{t}_{12}\,\alpha\wedge\beta - |t_{11}|^2\,\alpha\wedge\bar\alpha + t_{11}\,\bar{t}_{12}\,\bar\alpha\wedge\beta - t_{12}\,\bar\alpha\wedge\bar\beta - t_{12}\,\bar{t}_{11}\,\alpha\wedge\bar\beta -|t_{12}|^2\,\beta\wedge\bar\beta\bigg] \\
\nonumber &  & \hspace{60ex} \wedge  (\gamma + t_{31}\,\bar\alpha + t_{32}\,\bar\beta - D(t)\,\bar\gamma)  \\
\nonumber & - & \frac{\sigma_{2\bar{2}}(t)}{\bar\sigma_{12}(t)}\,\bigg[\alpha\wedge\beta + t_{21}\,\alpha\wedge\bar\alpha + t_{22}\,\alpha\wedge\bar\beta + t_{11}\,\bar\alpha\wedge\beta + t_{11}\,t_{22}\,\bar\alpha\wedge\bar\beta - t_{12}\,\beta\wedge\bar\beta - t_{12}\,t_{21}\,\bar\alpha\wedge\bar\beta\bigg] \\
\nonumber & &  \hspace{60ex} \wedge (\bar\gamma + \bar{t}_{31}\,\alpha + \bar{t}_{32}\,\beta - \overline{D(t)}\,\gamma).\end{eqnarray} 

\noindent After expanding and grouping the terms, we get

\begin{Lem}\label{Lem:Gamma_1_t_explicit} For every $t\in\Delta_{[\gamma]}$, the $J_t$-$(2,\,1)$-form $\Gamma_1(t)$ of (\ref{eqn:Gamma_j-forms_t}) is explicitly given by the following formula in terms of a basis of $3$-forms generated by $\alpha,\beta,\gamma$ and their conjugates:

\begin{eqnarray}\label{eqn:Gamma_1_t_explicit}\nonumber \Gamma_1(t) & = & \nonumber - \bar{t}_{12} \,\alpha\wedge\beta\wedge\gamma  \\
\nonumber & - & D(t)\,\bigg(t_{12} + \frac{\sigma_{2\bar{2}}(t)}{\bar\sigma_{12}(t)}\bigg)\,\bar\alpha\wedge\bar\beta\wedge\bar\gamma - \bigg(1 - |t_{11}|^2 - \frac{\sigma_{2\bar{2}}(t)}{\bar\sigma_{12}(t)}\,t_{21}\,\overline{D(t)}\bigg)\,\alpha\wedge\bar\alpha\wedge\gamma \\
\nonumber & - & \bigg[t_{32}\,(1-|t_{11}|^2) + t_{12}\,\bar{t}_{11}\,t_{31} - \frac{\sigma_{2\bar{2}}(t)}{\bar\sigma_{12}(t)}\,D(t)\,\bar{t}_{31}\bigg]\,\alpha\wedge\bar\alpha\wedge\bar\beta \\
\nonumber & - & \bigg[\bar{t}_{12}\,t_{31} - \frac{\sigma_{2\bar{2}}(t)}{\bar\sigma_{12}(t)}\,(t_{21}\,\bar{t}_{32} + t_{11}\,\bar{t}_{31})\bigg]\,\alpha\wedge\beta\wedge\bar\alpha - \bigg[(|t_{11}|^2-1)\,D(t) + \frac{\sigma_{2\bar{2}}(t)}{\bar\sigma_{12}(t)}\,t_{21}\bigg]\,\alpha\wedge\bar\alpha\wedge\bar\gamma \\ 
\nonumber & - & \bigg[\bar{t}_{12}\,t_{32} - \frac{\sigma_{2\bar{2}}(t)}{\bar\sigma_{12}(t)}\,(t_{22}\,\bar{t}_{32} + t_{12}\,\bar{t}_{31})\bigg]\,\alpha\wedge\beta\wedge\bar\beta + \bigg[\bar{t}_{12}\,D(t) - \frac{\sigma_{2\bar{2}}(t)}{\bar\sigma_{12}(t)}\bigg]\,\alpha\wedge\beta\wedge\bar\gamma  \\
\nonumber & - & \bigg[t_{11}\,\bar{t}_{12} - \frac{\sigma_{2\bar{2}}(t)}{\bar\sigma_{12}(t)}\,t_{11}\,\overline{D(t)}\bigg]\,\bar\alpha\wedge\beta\wedge\gamma - \bigg[t_{11}\,\bar{t}_{12}\,t_{32} - |t_{12}|^2\,t_{31} - \frac{\sigma_{2\bar{2}}(t)}{\bar\sigma_{12}(t)}\,D(t)\,\bar{t}_{32}\bigg]\,\bar\alpha\wedge\beta\wedge\bar\beta \\
\nonumber & + & \bigg[D(t)\,t_{11}\,\bar{t}_{12} - \frac{\sigma_{2\bar{2}}(t)}{\bar\sigma_{12}(t)}\,t_{11}\bigg]\,\bar\alpha\wedge\beta\wedge\bar\gamma + \bigg[t_{12}+ \frac{\sigma_{2\bar{2}}(t)}{\bar\sigma_{12}(t)}\,|D(t)|^2\bigg]\,\bar\alpha\wedge\bar\beta\wedge\gamma \\
\nonumber & + & \bigg[t_{12}\,\bar{t}_{11} + \frac{\sigma_{2\bar{2}}(t)}{\bar\sigma_{12}(t)}\,t_{22}\,\overline{D(t)}\bigg]\,\alpha\wedge\bar\beta\wedge\gamma + \frac{\sigma_{2\bar{2}}(t)}{\bar\sigma_{12}(t)}\,\overline{D(t)}\,\alpha\wedge\beta\wedge\gamma + \bigg[t_{12}\,\bar{t}_{11}\,D(t) + \frac{\sigma_{2\bar{2}}(t)}{\bar\sigma_{12}(t)}\,t_{22}\bigg]\,\alpha\wedge\bar\beta\wedge\bar\gamma \\
 \nonumber  & + & \bigg[|t_{12}|^2 - \frac{\sigma_{2\bar{2}}(t)}{\bar\sigma_{12}(t)}\,t_{12}\,\overline{D(t)}\bigg]\,\beta\wedge\bar\beta\wedge\gamma - \bigg[|t_{12}|^2\,D(t) - \frac{\sigma_{2\bar{2}}(t)}{\bar\sigma_{12}(t)}\,t_{12}\bigg]\,\beta\wedge\bar\beta\wedge\bar\gamma.\end{eqnarray}

\vspace{2ex}

 Analogous formulae hold for the $J_t$-$(2,\,1)$-forms $\Gamma_2(t)$, $\Gamma_3(t)$, $\Gamma_4(t)$ of (\ref{eqn:Gamma_j-forms_t}). Each formula contains on the r.h.s. a single term featuring an isolated anti-holomorphic factor $\bar{t}_{i\lambda}$ (i.e. an anti-holomorphic factor $\bar{t}_{i\lambda}$ that is not multiplied by any other $t_{j\mu}$ or $\bar{t}_{j\mu}$). These terms are, respectively,

$$-\bar{t}_{22}\,\alpha\wedge\beta\wedge\gamma,  \hspace{3ex} \bar{t}_{11}\,\alpha\wedge\beta\wedge\gamma, \hspace{3ex} \bar{t}_{21}\,\alpha\wedge\beta\wedge\gamma.$$

\end{Lem}

\vspace{2ex}

 On the other hand, the dependence on $t$ of the $C^\infty$ functions $\sigma_{12}(t)$, $\sigma_{1\bar{1}}(t), \sigma_{1\bar{2}}(t), \sigma_{2\bar{1}}(t), \sigma_{2\bar{2}}(t)$ can be made explicit using computations from [Ang14]. Indeed, consider the following functions of $t$ (cf. [Ang14, p. 76] where the notation $\alpha,\beta,\gamma$ was used instead of $a(t),b(t),c(t)$ featuring below):

\vspace{3ex}

$\displaystyle a(t)= \frac{1}{1 - |t_{22}|^2 -t_{21}\,\bar{t}_{12}}, \hspace{2ex} b(t)=t_{21}\,\bar{t}_{11} + t_{22}\,\bar{t}_{21}, \hspace{2ex} c(t)=\frac{1}{1 - |t_{11}|^2 -a(t)\,b(t)\,(t_{11}\,\bar{t}_{12} + t_{12}\,\bar{t}_{22}) -  \bar{t}_{12}\,\bar{t}_{21}},$

\vspace{1ex}

$\displaystyle \lambda_1(t) = - t_{11}\,(1 + a(t)\,\bar{t}_{12}\,t_{21} + a(t)\,|t_{22}|^2), \hspace{2ex} \lambda_2(t) = a(t)\,(t_{11}\,\bar{t}_{12} + t_{12}\,\bar{t}_{22}),$

\vspace{1ex}

$\displaystyle \lambda_3(t) = -t_{12}\,(1 + a(t)\,\bar{t}_{12}\,t_{21} + a(t)\,|t_{22}|^2), \hspace{2ex} \mu_0(t)= b(t)\,c(t), \hspace{2ex} \mu_1(t) = \lambda_1(t)\,b(t)\,c(t) - t_{21},$

\vspace{1ex}

$\displaystyle \mu_2(t) = 1 + \lambda_2(t)\,b(t)\,c(t), \hspace{2ex} \mu_3(t) = \lambda_3(t)\,b(t)\,c(t) - t_{22}.$

\vspace{3ex}

\noindent Then, for all $t$ in Nakamura's class $(ii)$, we have the explicit formulae (cf. e.g. [Ang14, p.77]):

$$\sigma_{12}(t) = -c(t) + t_{21}\,\bar\lambda_3(t)\,\bar{c}(t) + t_{22}\,\bar{a}(t)\,\bar\mu_3(t), \hspace{2ex}$$ 

\noindent and

\begin{eqnarray}\nonumber \sigma_{1\bar{1}}(t) = t_{21}\,\overline{c(t)\,(1 + t_{21}\,\bar{t}_{12}\,a(t) + |t_{22}|^2\,a(t))}, & & \sigma_{1\bar{2}}(t) = t_{22}\,\overline{c(t)\,(1 + t_{21}\,\bar{t}_{12}\,a(t) + |t_{22}|^2\,a(t))}, \\
\nonumber \sigma_{2\bar{1}}(t) = -t_{11}\,c(t)\,(1 + t_{21}\,\bar{t}_{12}\,a(t) + |t_{22}|^2\,a(t)), & & \sigma_{2\bar{2}}(t) = -t_{12}\,c(t)\,(1 + t_{21}\,\bar{t}_{12}\,a(t) + |t_{22}|^2\,a(t)).\end{eqnarray}

\noindent This explicitly yields

\begin{equation}\label{eqn:sigma_22bar_explicit-formula}\sigma_{2\bar{2}}(t) = - t_{12}\,\,\frac{1 + \frac{t_{21}\,\bar{t}_{12}}{1 - |t_{22}|^2 - t_{21}\,\bar{t}_{12}} + \frac{|t_{22}|^2}{1 - |t_{22}|^2 - t_{21}\,\bar{t}_{12}}}{1 - |t_{11}|^2 - \frac{(t_{11}\,\bar{t}_{12} + t_{12}\,\bar{t}_{22})( t_{21}\,\bar{t}_{11} + t_{22}\,\bar{t}_{21})}{1 - |t_{22}|^2 - t_{21}\,\bar{t}_{12}} - t_{12}\,\bar{t}_{21}}\end{equation}

\noindent and analogous formulae for $\sigma_{1\bar{1}}(t)$, $\sigma_{1\bar{2}}(t)$, $\sigma_{2\bar{1}}(t)$ with a different holomorphic factor $\pm t_{i\lambda}$ and a possibly conjugated big fraction. 

 The conclusion is the following

\begin{Lem}\label{Lem:sigma_22bar_vanishing-anti-hol-deriv} For all $t$ in Nakamura's class $(ii)$ and for all $i,\lambda$, we have

\begin{equation}\label{eqn:sigma_22bar_vanishing-anti-hol-deriv} \frac{\partial\sigma_{1\bar{1}}}{\partial\bar{t}_{i\lambda}}(0) = \frac{\partial\sigma_{1\bar{2}}}{\partial\bar{t}_{i\lambda}}(0) = \frac{\partial\sigma_{2\bar{1}}}{\partial\bar{t}_{i\lambda}}(0) = \frac{\partial\sigma_{2\bar{2}}}{\partial\bar{t}_{i\lambda}}(0) = 0.\end{equation}

The same conclusion holds for all $t$ in Nakamura's class $(iii)$, hence in particular for all $t\in\Delta_{[\gamma]}$, by very similar computations.

\end{Lem}

\noindent {\it Proof.} Whenever some $\bar{t}_{i\lambda}$ features in formula (\ref{eqn:sigma_22bar_explicit-formula}) for $\sigma_{2\bar{2}}(t)$ or in one of its analogues for $\sigma_{1\bar{1}}(t)$, $\sigma_{1\bar{2}}(t)$ and $\sigma_{2\bar{1}}(t)$, it is multiplied by a factor $t_{j\mu}$ or $\bar{t}_{j\mu}$ which vanishes at $t=0$, while the denominators on the r.h.s. of (\ref{eqn:sigma_22bar_explicit-formula}) equal $1$ at $t=0$.  \hfill $\Box$

\vspace{3ex}

\noindent {\bf References.} \\

\noindent [Aep62]\, A. Aeppli --- {\it  Some Exact Sequences in Cohomology Theory for K\"ahler Manifolds} --- Pacific J. Math., {\bf 12} (1962).

\vspace{1ex}

\noindent [AB91]\, L. Alessandrini, G. Bassanelli --- {\it Compact $p$-K\"ahler Manifolds} --- Geometriae Dedicata {\bf 38} (1991) 199-210.

\vspace{1ex}

\noindent [Ang11]\, D. Angella --- {\it The Cohomologies of the Iwasawa Manifold and of Its Small Deformations} --- J. Geom. Anal. (2011) DOI: 10.1007/s12220-011-9291-z.

\vspace{1ex}

\noindent [Ang14]\, D. Angella --- {\it Cohomological Aspects in Complex Non-K\"ahler Geometry} --- LNM 2095, Springer (2014).

\vspace{1ex}

\noindent [Bog78]\, F. Bogomolov --- {\it Hamiltonian K\"ahler Manifolds} --- Soviet Math. Dokl. {\bf 19} (1978), 1462-1465.

\vspace{1ex}

\noindent [BG83]\, R. Bryant, P. Griffiths --- {\it Some Observations on the Infinitesimal Period Relations for Regular Threefolds with Trivial Canonical Bundle} --- in ``Arithmetic and Geometry'' (papers dedicated to Shafarevich), vol. II, 77-102, Progr. Math. {\bf 36} Birkh\"auser (1983).

\vspace{1ex}

\noindent [COUV16]\, M. Ceballos, A. Otal, L. Ugarte, R. Villacampa --- {\it Invariant Complex Structures on $6$-nilmanifolds: Classification, Fr\"olicher Spectral Sequence and Special Hermitian Metrics} --- J. Geom. Anal. {\bf 26} (2016), no. 1, 252-286.

\vspace{1ex}

\noindent [DGMS75]\, P. Deligne, Ph. Griffiths, J. Morgan, D. Sullivan --- {\it Real Homotopy Theory of K\"ahler Manifolds} --- Invent. Math. {\bf 29} (1975), 245-274.

\vspace{1ex}

\noindent [DP04]\, J.-P. Demailly, M. Paun --- {\it Numerical Characterization of the K\"ahler Cone of a Compact K\"ahler Manifold} --- Ann. of Math. {\bf 159} (2004), 1247-1274.

\vspace{1ex}

\noindent [Gau77]\, P. Gauduchon --- {\it Le th\'eor\`eme de l'excentricit\'e nulle} --- C.R. Acad. Sc. Paris, S\'erie A, t. {\bf 285} (1977), 387-390.

\vspace{1ex}

\noindent [Gri68]\, P. Griffiths --- {\it Periods of Integrals on Algebraic Manifolds I, II} --- Amer. J. Math. {\bf 90} (1968), 568-626 and 805-865.

\vspace{1ex}

\noindent [KS60]\, K. Kodaira, D.C. Spencer --- {\it On Deformations of Complex Analytic Structures, III. Stability Theorems for Complex Structures} -- Ann. of Math. {\bf 71}, no. 1 (1960), 43-76.

\vspace{1ex}

\noindent [LTY15]\, S.-C. Lau, L.-S. Tseng, S.-T. Yau --- {\it Non-K\"ahler SYZ Mirror Symmetry} --- Commun. Math. Phys. {\bf 340} (2015), 145-170.

\vspace{1ex}

\noindent [Nak75]\, I. Nakamura --- {\it Complex Parallelisable Manifolds and their Small Deformations} --- J. Diff. Geom. {\bf 10} (1975) 85-112.

\vspace{1ex}

\noindent [Pop13a]\, D. Popovici --- {\it Deformation Limits of Projective Manifolds\!\!: Hodge Numbers and Strongly Gauduchon Metrics} --- Invent. Math. {\bf 194} (2013), 515-534.

\vspace{1ex}

\noindent [Pop13b]\, D. Popovici --- {\it Holomorphic Deformations of Balanced Calabi-Yau $\partial\bar\partial$-Manifolds} --- arXiv e-print AG 1304.0331v1, to appear in the Annales de l'Institut Fourier.

\vspace{1ex}

\noindent [Pop14]\, D. Popovici --- {\it Deformation Openness and Closedness of Various Classes of Compact Complex Manifolds; Examples} --- Ann. Sc. Norm. Super. Pisa Cl. Sci. (5), Vol. XIII (2014), 255-305.

\vspace{1ex}

\noindent [Pop15]\, D. Popovici --- {\it Aeppli Cohomology Classes Associated with Gauduchon Metrics on Compact Complex Manifolds} --- Bull. Soc. Math. France {\bf 143} (3), (2015), p. 1-37.

\vspace{1ex}

\noindent [Pop17]\, D. Popovici --- {\it The Albanese Map of sGG Manifolds and Self-Duality of the Iwasawa Manifold} --- arXiv e-print AG 1706.09746v1, to appear in the Rivista di Matematica della Universit\`a di Parma.

\vspace{1ex}

\noindent [PU14]\, D. Popovici, L. Ugarte --- {\it Compact Complex Manifolds with Small Gauduchon Cone} --- Proc. London Math. Soc. (3) {\bf 116} (2018), no. 5, 1161-1186.

\vspace{1ex}

\noindent [Rei87]\,M. Reid --- {\it The Moduli Space of $3$-folds with $K=0$ May Nevertheless be Irreducible} --- Math. Ann. {\bf 278} (1987), 329-334.

\vspace{1ex}

\noindent [Sch07]\, M. Schweitzer --- {\it Autour de la cohomologie de Bott-Chern} --- arXiv e-print math. AG/0709.3528v1.

\vspace{1ex}

\noindent [Tia87]\, G. Tian --- {\it Smoothness of the Universal Deformation Space of Compact Calabi-Yau Manifolds and Its Petersson-Weil Metric} --- Mathematical Aspects of String Theory (San Diego, 1986), Adv. Ser. Math. Phys. 1, World Sci. Publishing, Singapore (1987), 629--646.

\vspace{1ex}

\noindent [Tod89]\, A. N. Todorov --- {\it The Weil-Petersson Geometry of the Moduli Space of $SU(n\geq 3)$ (Calabi-Yau) Manifolds I} --- Comm. Math. Phys. {\bf 126} (1989), 325-346.

\vspace{1ex}

\noindent [Voi96]\, C. Voisin --- {\it Sym\'etrie miroir} --- Panoramas et Synth\`eses, {\bf 2}, Soci\'et\'e Math\'ematique de France, Paris, 1996.

\vspace{1ex}

\noindent [Voi02]\, C. Voisin --- {\it Hodge Theory and Complex Algebraic Geometry. I.} --- Cambridge Studies in Advanced Mathematics, 76, Cambridge University Press, Cambridge, 2002.

\vspace{1ex}

\noindent [Wu06]\, C.-C. Wu --- {\it On the Geometry of Superstrings with Torsion} --- thesis, Department of Mathematics, Harvard University, Cambridge MA 02138, (April 2006).

\vspace{6ex}

\noindent Universit\'e Paul Sabatier, Institut de Math\'ematiques de Toulouse,

\noindent 118 route de Narbonne, 31062 Toulouse, France

\noindent Email: popovici@math.univ-toulouse.fr

\end{document}